\documentclass[11pt, %dvipdfmx
]{amsart}

\usepackage{xcolor}
\usepackage[margin=30truemm]{geometry}

\usepackage[all]{xy}
\usepackage{yfonts}
\usepackage{amsmath}
\usepackage{amsfonts}
\usepackage{amssymb}
\usepackage{verbatim}
\usepackage{bm}
\usepackage{amsthm}
\usepackage{amscd}
\usepackage{amsmath}
\usepackage{graphicx,xcolor}
\usepackage{mathrsfs}
\usepackage{upgreek}
\usepackage{tikz}
\tikzset{dynkdot/.style={circle,draw,scale=.38}}
\definecolor{refkey}{gray}{0.5}
\definecolor{labelkey}{gray}{0.5}
\usepackage[
bookmarks=false,
colorlinks=true,
debug=true,
naturalnames=true,
pdfnewwindow=true,
citecolor=blue,
linkcolor=blue,
urlcolor = blue,
]{hyperref}
\usepackage{todonotes}
\usepackage{mathtools}
\newcommand{\seq}{\coloneqq}

\numberwithin{equation}{section}
\newtheorem{Thm}{Theorem}[section]
\newtheorem{Cor}[Thm]{Corollary}
\newtheorem{Prop}[Thm]{Proposition}
\newtheorem{Lem}[Thm]{Lemma}
\newtheorem{Conj}[Thm]{Conjecture}

\newtheorem*{Claim}{Claim}

\theoremstyle{definition}
\newtheorem{Def}[Thm]{Definition}
\newtheorem{Rem}[Thm]{Remark}
\newtheorem{Ex}[Thm]{Example}
\newcommand{\ul}{\underline}
\newcommand{\ol}{\overline}

\newcommand{\im}{\imath}
\newcommand{\jm}{\jmath}

\newcommand{\Ker}{\mathop{\mathrm{Ker}}}

\newcommand{\id}{\mathrm{id}}

\newcommand{\evt}{\mathrm{ev}_{t=1}}

\newcommand{\sgn}{\mathrm{sgn}}

\newcommand{\N}{\mathbb{N}}
\newcommand{\Z}{\mathbb{Z}}
\newcommand{\Q}{\mathbb{Q}}
\newcommand{\C}{\mathbb{C}}
\newcommand{\F}{\mathbb{F}}

\newcommand{\kk}{\Bbbk}
\newcommand{\sg}{\mathsf{g}}
\newcommand{\sn}{\mathsf{n}}
\newcommand{\sh}{\mathsf{h}}
\newcommand{\sP}{\mathsf{P}}

\newcommand{\sW}{\mathsf{W}}
\newcommand{\sfc}{\mathsf{c}}
\newcommand{\sC}{\mathsf{C}}
\newcommand{\sfd}{\mathsf{d}}
\newcommand{\sfp}{\mathsf{p}}
\newcommand{\sD}{\mathsf{D}}
\newcommand{\sfSg}{\mathsf{\Sigma}}

\newcommand{\cQ}{\mathcal{Q}}

\newcommand{\cM}{\mathcal{M}}
\newcommand{\cY}{\mathcal{Y}}
\newcommand{\cT}{\mathcal{T}}
\newcommand{\cK}{\mathcal{K}}
\newcommand{\cA}{\mathcal{A}}

\newcommand{\Cc}{\mathscr{C}}
\newcommand{\Nn}{\mathscr{N}}

\newcommand{\fg}{\mathfrak{g}}

\newcommand{\fD}{\mathfrak{D}}
\newcommand{\SG}{\mathfrak{S}}
\newcommand{\fe}{\mathfrak{e}}

\newcommand{\bfi}{{\boldsymbol{i}}}

\newcommand{\bfa}{{\boldsymbol{a}}}
\newcommand{\bfb}{{\boldsymbol{b}}}
\newcommand{\bfc}{{\boldsymbol{c}}}
\newcommand{\bfg}{{\boldsymbol{g}}}
\newcommand{\bfe}{{\boldsymbol{e}}}
\newcommand{\bfn}{{\boldsymbol{n}}}
\newcommand{\bftau}{{\boldsymbol{\tau}}}

\newcommand{\bfL}{\mathbf{L}}

\newcommand{\bT}{\mathbf{T}}

\newcommand{\tvee}{\widetilde{\vee}}
\newcommand{\tc}{\widetilde{c}}
\newcommand{\tm}{\widetilde{m}}
\newcommand{\tC}{\widetilde{C}}

\newcommand{\tG}{\widetilde{G}}
\newcommand{\tD}{\widetilde{D}}
\newcommand{\tB}{\widetilde{B}}

\newcommand{\tPsi}{\widetilde{\Psi}}
\newcommand{\trPsi}{\widetilde{\Psi}^{\mathrm{res}}}

\newcommand{\teta}{\widetilde{\eta}}
\newcommand{\tSigma}{\widetilde{\Sigma}}
\newcommand{\tsfSg}{\widetilde{\mathsf{\Sigma}}}

\newcommand{\hD}{\widehat{\Delta}}
\newcommand{\hI}{\widehat{I}}

\newcommand{\tbfB}{\widetilde{\mathbf{B}}^*}

\renewcommand{\deg}{\mathop{\mathsf{deg}}\nolimits}

\newenvironment{red}{\relax\color{red}}{\relax}
\newenvironment{blue}{\relax\color{blue}}{\hspace*{.5ex}\relax}

\newcommand{\ber}{\begin{red}}
\newcommand{\er}{\end{red}}
\newcommand{\beb}{\begin{blue}}
\newcommand{\eb}{\end{blue}}

\title[Isomorphisms among quantum Grothendieck rings and cluster algebras]
{Isomorphisms among quantum Grothendieck rings and cluster algebras}
\date{\today}

\author[R.~Fujita]{Ryo Fujita}
\address[R.~Fujita]{Research Institute for Mathematical Sciences, Kyoto University, Oiwake-Kitashirakawa, Sakyo, Kyoto, 606-8502, Japan \& Institut de Math\'{e}matiques de Jussieu-Paris Rive Gauche, Universit\'{e} Paris Cit\'e, F-75013, Paris, France}
\email{rfujita@kurims.kyoto-u.ac.jp}

\author[D.~Hernandez]{David Hernandez}
\address[D.~Hernandez]{Universit\'{e} Paris Cit\'e and Sorbonne Universit\'{e}, CNRS, IMJ-PRG, IUF, F-75013, Paris, France}
\email{david.hernandez@imj-prg.fr}

\author[S.-j.~Oh]{Se-jin Oh}
\address[S.-j.~Oh]{Ewha Womans University Seoul, 52 Ewhayeodae-gil, Daehyeon-dong, Seodaemun-gu, Seoul, South Korea}
\email{sejin092@gmail.com}

\author[H.~Oya]{Hironori Oya}
\address[H.~Oya]{Department of Mathematics, Tokyo Institute of Technology, 2-12-1 Ookayama, Meguro-ku, Tokyo, 152-8551, Japan}
\email{hoya@math.titech.ac.jp}

\makeatletter
\@namedef{subjclassname@2020}{%
  \textup{2020} Mathematics Subject Classification}
\makeatother

\begin{document}

\maketitle

\begin{abstract}
We establish a cluster theoretical interpretation of the isomorphisms of \cite{FHOO} among quantum Grothendieck rings of representations 
of quantum loop algebras. Consequently, we obtain a quantization of the monoidal categorification theorem of \cite{KKOP2}. 
We establish applications of these new ingredients. First we solve long-standing problems for any non-simply-laced 
quantum loop algebras:  the positivity of $(q,t)$-characters of all simple modules, and the analog of Kazhdan--Lusztig conjecture for all reachable modules (in the cluster monoidal categorification). 
We also establish the conjectural quantum $T$-systems for the $(q,t)$-characters of Kirillov--Reshetikhin modules.
Eventually, we show that our isomorphisms arise from explicit birational transformations of variables, which we call substitution formulas. This reveals new non-trivial relations among $(q, t)$-characters of simple modules. 
\end{abstract}

\tableofcontents

\section*{Introduction}

Consider a complex finite-dimensional simple Lie algebra $\fg$ and $q$ a generic quantum parameter. 
The associated quantum loop algebra $U_{q}(L\fg)$ is a quantum affinization of $\fg$ with a 
structure of Hopf algebra. 
In particular, finite-dimensional representations over $U_{q}(L\fg)$
form a rigid monoidal category $\Cc_{\fg}$, whose structure is quite intricate as it is 
neither semisimple nor braided. This category has many applications and has 
been intensively studied from various perspectives. However, several fundamental questions remain open, 
such as the dimension and $q$-character (in the sense of Frenkel--Reshetikhin~\cite{FR99}) of simple modules in non-simply-laced types. 
In the present paper, we establish that a Kazhdan--Lusztig type approach is available to solve this problem for a very large family of simple modules.

The quantum Grothendieck ring $\cK_t(\Cc_\fg)$ of $\Cc_{\fg}$ is a non-commutative 
deformation of the Grothendieck ring $K(\Cc_\fg)$ of $\Cc_{\fg}$ inside a quantum torus $\cY_t$. It was introduced in 
 \cite{Nak04, VV03} for simply-laced types and then in \cite{Her04} for non-simply-laced types with a different method.
This ring has a canonical basis $\bfL_{t,\fg}$ whose elements $L_t(m)$ are 
called $(q,t)$-characters of simple modules and can be calculated by the Kazhdan--Lusztig type algorithm. 
Here $m$ belongs to a set $\cM$ which serves as a parameter set for the simple modules $L(m)$ in $\Cc_{\fg}$
(which is comparable with the set of Drinfel'd polynomials). 

The quantum Grothendieck ring serves as a useful tool to study the simple modules in $\Cc_{\fg}$. Indeed, when $\fg$ is simply-laced, Nakajima proved that the $(q,t)$-characters of simple modules specialize to the $q$-characters of simple modules, based on the geometry of quiver varieties~\cite{Nak04}. Namely, the $q$-characters of simple modules can be calculated by the Kazhdan--Lusztig type algorithm. More precisely, this algorithm computes the Jordan--H\"{o}lder multiplicity 
$P_{m,m'}$ of the simple module $L(m')$ occurring in the standard module $M(m)$, that is a 
tensor product of fundamental representations, whose $q$-character is known by Frenkel--Mukhin \cite{FM01}.
Since we have
$$
[M(m)] = [L(m)] + \sum_{m' \in \cM\colon m' < m} P_{m,m'} [L(m')] 
$$ 
in the Grothendieck ring $K(\Cc_{\fg})$ for a certain partial ordering on $\cM$, this algorithm enables us to compute all the simple $q$-characters in principle.  

Here we emphasize that the quiver varieties play an essential role to guarantee the validity of the algorithm.
When $\fg$ is non-simply-laced, the above theory is not applicable
for the absence of a fully developed theory of quiver varieties.  
However, one can still formulate a conjectural Kazhdan--Lusztig type algorithm for general $\fg$.

\begin{Conj}[Analog of Kazhdan--Lusztig conjecture, {\cite[Conjecture 7.3]{Her04}}] \label{Conj0:KL}
Under the specialization $\cK_t(\Cc_\fg) \to K(\Cc_{\fg})$ at $t=1$, the element $L_t(m)$ corresponds to 
the simple class $[L(m)]$ for any $m \in \cM$.  
\end{Conj}

Since the $(q,t)$-characters of simple modules can be computed algorithmically as in the usual Kazhdan--Lusztig theory,
Conjecture~\ref{Conj0:KL} enables us to compute all the simple $q$-characters algorithmically once it is verified. Beyond simply-laced types, this conjecture has been proved for type $\mathrm{B}$ and for certain remarkable monoidal subcategories of $\Cc_\fg$ in \cite{FHOO} that are small (in the sense that they contain only a finite number of fundamental representations). 

A related problem is the following positivity conjecture which is known to be true for simply-laced types \cite{Nak04, VV03}, and was formulated for non-simply-laced types 
almost 20 years ago in \cite{Her04}. 
For types $\mathrm{CFG}$, the statement is only known 
for fundamental representations and was derived from a computer calculation in 
\cite{Her05}.

\begin{Conj}[Positivity of $(q,t)$-characters] \label{posqt}
In the quantum torus $\cY_t$, the elements $L_t(m)$ have non-negative coefficients. 
\end{Conj}

%When $\fg$ is simply-laced, all of these conjectures are known to be true by \cite{Nak04, VV03} using the geometry of quiver varieties. On the other hand,  they are still unsolved for non-simply-laced $\fg$.

In this paper, using the new ingredients that we explain below, we prove Conjecture~\ref{posqt} for all simple objects and Conjecture~\ref{Conj0:KL} for a large family of simple objects. 

\begin{Thm}[= Corollary~\ref{Cor:pos} \& Corollary~\ref{Cor:KLr}]\hfill \label{Thm0:main}
\begin{enumerate}
\item For general $\fg$, {\rm Conjecture~\ref{posqt}} holds.
\item For general $\fg$, {\rm Conjecture~\ref{Conj0:KL}} holds for all reachable modules.
\end{enumerate}
\end{Thm}

A simple module is said to be reachable if its class (or one of its spectral parameter shift) is a cluster monomial for the cluster algebra structure established in \cite{HL16} on a subcategory $\Cc^-$ of $\Cc_\fg$. This is a large family of simple modules, including many interesting modules, all 
Kirillov--Reshetikhin modules for instance. This cluster algebra structure, and some of its variations, will play a crucial role in our proofs.

Another important ingredient for our purposes is a collection of isomorphisms between the quantum Grothendieck rings for non-simply-laced quantum loop algebras and their unfolded simply-laced ones established in \cite{FHOO}. 
Let us explain this. Now we assume that $\fg$ is of non-simply-laced type.
Then we choose another simple Lie algebra $\sg$ of simply-laced type whose Dynkin diagram is obtained by \emph{unfolding} the Dynkin diagram of $\fg$ (see Figure~\ref{Fig:unf}).  
It was proved in \cite{FHOO} that there exists a (non-unique) isomorphism of $\Z[t^{\pm 1/2}]$-algebras
\begin{equation} \label{eq:CZisom}
\cK_{t}(\Cc_{\sg}) \simeq \cK_{t}(\Cc_{\fg})
\end{equation}
which induces a bijection between the $(q,t)$-characters of simple modules. The specialization at $t=1$ of the isomorphism (\ref{eq:CZisom}) yields an isomorphism between the usual Grothendieck rings 
which is non-trivial. For example, it does not respect the classes of fundamental modules nor of standard modules.  In the light of the new results we will explain, we 
understand why the combinatorics of this bijection are intricate, as they are related to the cluster algebra structures.

We establish that the quantum Grothendieck ring of the category $\Cc^-$ (as well as some variations $\Cc_{\le \xi}$ of this category, that we introduce) 
has a structure of a quantum cluster algebras. This generalizes previous results \cite{HL15, Qin17, Bit}. 
We give a cluster theoretical interpretation of the isomorphisms among quantum Grothendieck rings constructed above, together with their behaviors on canonical bases. 
Consequently, we derive the following.

\begin{Thm}[= Theorem~\ref{Thm:main}] The quantum cluster monomials of $\cK_{t}(\Cc^-)$ belong to the canonical basis $\bfL_{t,\fg}$ of $(q,t)$-characters of simple modules. 
\end{Thm}

This result is a quantum version of the monoidal categorification theorem of \cite{KKOP2} which states that the (classical) cluster monomial in $\cK(\Cc^-)$ are classes 
of simple modules. 
We also have an analogous statement for the whole category $\Cc_\fg$ (see Theorem~\ref{Thm:qmc_CZ}, where we actually work with a monoidal skeleton $\Cc_{\Z}$ of $\Cc_\fg$ for simplicity).
Let us sum up the situation in the following diagram: 

\vspace{0.5cm}

$$\begin{xymatrix}{ & K(\Cc_{\fg})  
& \ar[l]_{t=1}\cK_{t}(\Cc_{\fg}) \ar@/^2pc/@{<->}[rr]^{\cong} & \ar@{-->}[l]_-{\cong}  \mathcal{A}_t\ar@{-->}[r]^-{\cong}  &  \cK_{t}(\Cc_{\sg}) 
\\ \bfL_{\fg}\ar@{^{(}->}[ur]&\ar@{_{(}->}[l] {\bf M}\ar@{^{(}->}[u]   & \bfL_{t,\fg} \ar@{^{(}->}[u] 
& {\bf M}_t\ar@{^{(}->}[u]\ar@{^{(}-->}[r]\ar@{_{(}-->}[l]  \ar@{->}@/^2pc/[ll]^{t=1} & \bfL_{t,\sg} \ar@{^{(}->}[u] \ar@/^2pc/@{<->}[ll]^-{1:1}}\end{xymatrix}$$
\vspace{0.5cm}

\noindent
Here $\bfL_{\fg}$ is the basis of the Grothendieck ring $K(\Cc_{\fg})$ consisting of simple classes. 
It contains the set of classes of reachable modules ${\bf M}$ by \cite{KKOP2}. 
A quantum cluster algebra $\mathcal{A}_t$ is isomorphic to $\cK_{t}(\Cc_{\fg})$ and to $ \cK_{t}(\Cc_{\sg})$ as proved in Sections~\ref{ssecfive} and \ref{secsix}. The cluster algebra $\mathcal{A}_t$ contains its set of quantum cluster monomials ${\bf M}_t$ which specializes at $t = 1$ to ${\bf M}$. The composition of the isomorphisms to $\mathcal{A}_t$ recovers the isomorphism~\eqref{eq:CZisom} of
\cite{FHOO} as proved in Section \ref{secsix}. This isomorphism induces a bijection between the canonical 
bases $\bfL_{t,\sg}$, $\bfL_{t,\fg}$ as mentioned above. We prove in Section \ref{secsix} that ${\bf M}_t$ is sent to a subset of $\bfL_{t,\sg}$, and so also to a subset of $\bfL_{t,\fg}$. 
Combining all these results, we obtain that the elements in $\bfL_{t,\fg}$ that come from elements 
in ${\bf M}_t$ are evaluated at $t = 1$ to elements in $\bfL_{\fg}$, that is Theorem \ref{Thm0:main} (2). %the second statement in Theorem \ref{Thm0:main}.
(In the main body of the paper, we actually work with the analog of the above diagram for the categories $\Cc_{\le \xi}$.)

Theorem \ref{Thm0:main} (1) is proved for all simple modules by identifying the coefficients occurring in the $(q,t)$-characters 
with certain structure constants of the quantum Grothendieck ring with respect to the canonical basis $\bfL_{t, \fg}$. 
%This is possible thanks to the first property for Kirillov--Reshetikhin modules as they are reachable simple modules. 
For this identification, the $(q, t)$-characters of particular reachable simple modules, called Kirillov--Reshetikhin modules, play a key role. Our quantum version of the monoidal categorification theorem implies that they satisfy the quantum $T$-systems, as conjectured in \cite{FHOO}, which help us to investigate the structure of these $(q, t)$-characters.

%We prove that they satisfy the quantum $T$-systems, as conjectured in \cite{FHOO}, {\color{red} which is one of their crucial properties}.

\medskip

As another application of the cluster theoretical interpretation of our isomorphisms among quantum Grothendieck rings, we show that our isomorphisms  arise from explicit birational transformations of the variables in the quantum torus $\cY_t$ (Theorem \ref{Thm:Substitution}). We call these birational transformations as \emph{substitution formulas}.  Indeed, in \cite[\S10.3]{FHOO}, we provided a way of calculating the correspondence between the $\ell$-highest weight monomials of the $(q, t)$-characters of simple modules which are mutually related under our isomorphism, while the explicit correspondence among the lower terms had not been known. The substitution formulas in the present paper provide one method to calculate it, and reveal new non-trivial relations among the $(q, t)$-characters of simple modules. 
Note that we already know that the $(q, t)$-characters of simple modules specialize to their $q$-characters at $t=1$ in several cases  (for example, in the case when $\fg$ is of types $\mathrm{ABDE}$ or in the new cases established in the present paper), hence our formulas also imply several non-trivial relations among the $q$-characters of simple modules. It seems to be a new application of cluster algebras to the representation theory of quantum loop algebras, and it might be an interesting problem to further investigate the meaning of our formulas.

\medskip

For a discussion on the recent study of the non-symmetric quantized Coulomb branches  \cite{NW19, NakajimaQCB} by Nakajima and Weekes and the possible relation to our results, see the Introduction of \cite{FHOO}.

\medskip

The paper is organized as follows.
In Section~\ref{secone}, we introduce the quantum cluster algebra $\cA_{\bfi}$ associated to an infinite sequence $\bfi$ with values in $\Delta_0$, the 
vertex set of the Dynkin diagram of $\sg$ (such that each value occurs infinitely many times), imitating the cluster structure of the quantum unipotent groups. In Section~\ref{sectwo}, we study the relation between the quantum cluster algebras $\cA_\bfi$ and $\cA_{\bfi'}$ when $\bfi$ and $\bfi'$ are related by simple operations  (commutation moves, braid moves and forward shifts in respective Sections~\ref{cmov}, \ref{bmov} and \ref{fshif}). 
In particular, we obtain isomorphisms and embeddings of quantum cluster algebras which are relevant for our purposes. In Section \ref{secthree}, we recall that by \cite{GLS13, KKKO18, KK19} the quantum unipotent group $\cA_t[N_-]$ associated to $\sg$ has a quantum cluster algebras structure compatible with the dual canonical basis, which is isomorphic to some of the (finite version of the) quantum cluster algebras $\cA_\bfi$. We establish that these isomorphisms are compatible with the transformations of the last section (Corollary~\ref{Cor:bftau}). In Section~\ref{secfour}, we give general reminders on the category $\Cc$ of finite-dimensional representations of a quantum loop algebra, in particular on Kazhdan--Lusztig type conjectures in this context. We also recall the monoidal subcategories  $\Cc_\Z$ and $\Cc^-$ defined in \cite{HL10, HL16} and we introduce subcategories $\Cc_{\le \xi}$ generalizing $\Cc^-$. In Section~\ref{ssecfive}, we establish that the quantum Grothendieck ring of the category $\Cc_{\le \xi}$ has a structure of a quantum cluster algebra 
isomorphic to an algebra $\cA_\bfi$ introduced above (Theorem~\ref{Thm:qcl}).  
Our proof is partly based first on an isomorphism of quantum tori that we establish (Corollary~\ref{Cor:teta}). In Section \ref{secsix}, we give a cluster theoreticalal interpretation of the isomorphisms of quantum Grothendieck rings constructed in \cite{FHOO} (Theorem \ref{Thm:Psi=tau}), together with their canonical basis (Corollary~\ref{Cor:Psi}). This leads to a quantum version of the monoidal categorification theorem for the categories $\Cc_{\le \xi}$ (Theorem~\ref{Thm:main}) and $\Cc_\Z$ (Theorem~\ref{Thm:qmc_CZ}). 
We derive the applications discussed above (Corollary~\ref{Cor:KLr}, Corollary~\ref{Cor:pos}).
In Section \ref{sec:Substitution}, we show that our isomorphisms among quantum Grothendieck rings come from explicit birational transformations among the variables in our quantum tori $\cY_t$, which we call substitution formulas (Theorems \ref{Thm:Substitution} and \ref{Thm:Substitution_cl}). %They reveal non-trivial relations among the $q$-characters of simple modules (for example over $U_{q}(L\mathfrak{sl}_{2n})$ and $U_{q}(L\mathfrak{so}_{2n+1})$).

%%%%%%%%%%%%%%%%%%%%%%%%%%%%%%%%%%%%%%%%%%%%%%%%%%%%%%%%%%%%%%%%%%%%%%%%%%%%%%%%%%
\subsection*{Acknowledgments}
R.~F.~was supported by JSPS Overseas Research Fellowships.  
D.~H.~was supported by the Institut Universitaire de France.
S.-j.\ Oh was supported by
the Ministry of Education of the Republic of Korea and the National Research Foundation of Korea (NRF-2022R1A2C1004045). 
H.~O.~was supported by JSPS Grant-in-Aid for Early-Career Scientists (No.19K14515). 

%%%%%%%%%%%%%%%%%%%%%%%%%%%%%%%%%%%%%%%%%%%%%%%%%%%%%%%%%%%%%%%%%%%%%%%%%%%%%%%%%%

\subsection*{Conventions}
Let $\N \seq \Z_{>0}$ be the set of positive integers and $\N_0 \seq \N \cup \{0\}$.
For $a,b \in \Z$, we set $[a,b] \seq \{ u\in \Z \mid a \le u \le b \}$.
For a statement $\mathtt{P}$, we set $\delta(\mathtt{P})$ to be $1$ or $0$
according that $\mathtt{P}$ is true or not.
As a special case, we use the notation $\delta_{i,j} \seq \delta(i=j)$ (Kronecker's delta).

\section{The quantum cluster algebra $\cA_{\bfi}$}\label{secone}

Let $\Delta$ be the Dynkin graph of a finite-dimensional simple Lie algebra $\sg$ over $\C$, with $\Delta_0$ being its vertex set.
We introduce the quantum cluster algebra $\cA_{\bfi}$ associated to an infinite sequence $\bfi$ with values in $\Delta_0$ (such that each value occurs infinitely many times), imitating the cluster structure of the quantum unipotent groups. The results in this section hold true for $\sg$ of arbitrary types, while in the applications to the representation theory of quantum loop algebras of arbitrary untwisted type we will obtain in this paper, we will only use the quantum cluster algebra $\cA_{\bfi}$ obtained from $\sg$ of simply-laced type (even when we will handle non-simply-laced quantum loop algebras).  
%In this paper, we will use the quantum cluster algebra $\cA_{\bfi}$ obtained from a Lie algebra $\sg$ of simply-laced type, although we obtain applications of our work to quantum affine algebras of arbitrary untwisted type.

\subsection{Notation}
\label{Ssec:notation}
Let $\sC = (\sfc_{ij})_{i,j \in \Delta_0}$ be the Cartan matrix of $\sg$ and $\sD = \mathrm{diag}(\sfd_i \mid i \in \Delta_0)$ its minimal left symmetrizer (i.e. $\min\{\sfd_i \mid i \in \Delta_0\}=1$).
For $i,j\in \Delta_0$, we write $i \sim j$ when $\sfc_{ij}<0$.
Let $\sh^*$ be the dual of a Cartan subalgebra of $\sg$.  
Let $\{ \alpha_i \}_{i \in \Delta_0}$ and $\{ \varpi_i\}_{i \in \Delta_0}$ be the two bases of $\sh^*$ formed by simple roots and fundamental weights respectively. 
They are related by $\alpha_i = \sum_{j \in \Delta_0} \sfc_{ji} \varpi_j$. 
We consider the symmetric bilinear pairing $(\cdot,\cdot)$ on $\sh^*$ given by $(\alpha_i,\alpha_j)=\sfd_i \sfc_{ij}$.  
For each $i \in \Delta_0$, the $i$-th simple reflection $s_i$ is a linear operator on $\sh^*$ given by 
$s_i \varpi_j = \varpi_j - \delta_{i,j} \alpha_i$ for $j \in \Delta_0$.
Let $\sW$ denote the Weyl group of $\sg$, which is the subgroup of $GL(\sh^*)$ generated by the simple reflections $\{s_i\}_{i \in \Delta_0}$. 
Note that the pairing $(-,-)$ is $\sW$-invariant.
The pair $(\sW,\{s_i\}_{i \in \Delta_0})$ forms a finite Coxeter system.
Let $w_\circ$ be the longest element of $\sW$ and $\ell \in \N$ its length.
Let $i \mapsto i^*$ denote the involution of $\Delta$ given by $w_\circ$, namely $w_\circ \alpha_i = - \alpha_{i^*}$. 
%We linearly extend the pairing $(-,-)$ on $\sh^* \seq \sQ\otimes \C$.  
%Let $\{\varpi_i\}_{i \in \Delta_0} \subset \sh^*$ 
%denote the dual basis to $\{\alpha_i\}_{i \in \Delta_0}$ with respect to $(-,-)$.  
%denote the \beb basis such that $(\varpi_i,\alpha_j) = \delta_{i,j} d_i$ where 
%$d_i \seq \dfrac{(\alpha_i,\alpha_i)}{2}$. {\color{green}In this paper, the pairing is set so that $\min\{d_i \mid i \in \Delta_0 \} =1$.} \eb 

%\begin{Rem} 
%The following discussion in the first two sections can be generalized to any Coxeter group a with simply-laced Coxeter graph. 
%\end{Rem}

Consider the set $\Delta_0^{\N}$ of infinite sequences of elements of $\Delta_0$. 
Let $\Delta_0^{(\infty)}$ denote the subset of $\Delta_0^\N$ consisting of sequences $\bfi = (i_u)_{u \in \N}$ satisfying the condition: 
\begin{equation} \label{eq:cond1}
\text{For any $i \in \Delta_0$, we have $|\{u \in \N \mid i_u = i\}|=\infty$.}
\end{equation}

\subsection{Compatible pair $(\tB_{\bfi}, \Lambda_{\bfi})$}
In this section, we fix a sequence $\bfi = (i_u)_{u \in \N} \in \Delta_0^{(\infty)}$. 
For $u \in \N$ and $j \in \Delta_0$, we write
\begin{align*}
u^+ = u^+_{\bfi} &\seq \min\{ k \in \N \mid k >u, i_k = i_u \},\\
u^- = u^-_{\bfi} &\seq \max(\{ k \in \N \mid k < u, i_k=i_u \}\cup\{0\}),\\
u^+(j) = u^+_{\bfi}(j) &\seq \min\{ k \in \N \mid k>u, i_k = j \},\\
u^-(j) = u^-_{\bfi}(j) &\seq \max(\{ k \in \N \mid k < u, i_k=j \}\cup\{0\}).
\end{align*}
For $u \in \N_0$, we define
\[ w_u = w^{\bfi}_u \seq s_{i_1}s_{i_2}\cdots s_{i_u} \in \sW.\]
Consider the $\N \times \N$-matrix $\tB_\bfi = (b_{u,v})_{u,v \in \N}$ given by 
\begin{align*}
b_{u,v} = \begin{cases}
1 &\text{if } v=u^+, \\
-1 &\text{if } u=v^+, \\
\sfc_{i_u,i_v} & \text{if $u<v<u^+<v^+$,} \\
-\sfc_{i_u,i_v} &\text{if $v<u<v^+<u^+$,} \\
0 &\text{otherwise.}
\end{cases}
\end{align*}
Note that $\tB_\bfi$ is skew-symmetrizable by ${\rm diag}(\sfd_{i_u} \mid u\in \N)$.

Let $\Lambda_\bfi = (\Lambda_{u,v})_{u,v \in \N}$ be the skew-symmetric matrix defined by
\begin{equation} \label{eq:Lambda}
\Lambda_{u,v} = - \Lambda_{v,u} \seq (\varpi_{i_u}-w_u\varpi_{i_u}, \varpi_{i_v}+w_v\varpi_{i_v}) \quad \text{for $u\le v$}.
\end{equation}

\begin{Lem} \label{Lem:Lambda} 
Let $u,v \in \N$. If $u < v^+$, we have $\Lambda_{u,v} = (\varpi_{i_u}-w_u\varpi_{i_u}, \varpi_{i_v}+w_v\varpi_{i_v}).$
\end{Lem}
\begin{proof}
When $u \le v$, this is nothing but the definition \eqref{eq:Lambda}.
When $v < u <v^+$, we have $(w_u \varpi_{i_u}, w_v \varpi_{i_v}) = (\varpi_{i_u}, \varpi_{i_v})$ and hence
\[ \Lambda_{u,v} = -(\varpi_{i_u}+w_u\varpi_{i_u}, \varpi_{i_v}-w_v\varpi_{i_v}) = (\varpi_{i_u}-w_u\varpi_{i_u}, \varpi_{i_v}+w_v\varpi_{i_v}).\qedhere \]
\end{proof}

\begin{Prop} \label{Prop:BL}
Let $\bfi = (i_u)_{u \in \N}\in \Delta_0^{(\infty)}$. 
For any $u,v \in \N$, we have
\begin{equation} \label{eq:BL}
\sum_{k \in \N} b_{k,u}\Lambda_{k,v} = 2 \sfd_{i_u} \delta_{u,v}.
\end{equation} 
\end{Prop}
\begin{proof}
Put $i \seq i_u$.
By the definition of $\tB_\bfi$, we have 
\[ \sum_{k \in \N} b_{k,u}\Lambda_{k,v} = \Lambda_{u^-,v} - \Lambda_{u^+,v} + \sum_{j \sim i} \sfc_{ji} \left(\Lambda_{u^-(j),v} - \Lambda_{(u^+)^-(j),v}\right) =: x_{u,v},\]
where we understand $\Lambda_{0,v}=0$ (then, \eqref{eq:Lambda} still holds for $u=0$). 
Note that $(u^+)^-(j)$ denotes the largest integer $u' < u^+$ such that $i_{u'} = j$.

We need to show $x_{u,v}=2 \sfd_{i_u}\delta_{u,v}$.
Since $s_i \varpi_j = \varpi_j - \delta_{ij}\alpha_i$ and $\alpha_i = 2\varpi_i + \sum_{j\sim i} \sfc_{ji} \varpi_j$, we have 
\begin{equation} \label{eq:wspi}
w_u \varpi_i = - w_{u^-} \varpi_i - \sum_{j\sim i} \sfc_{ji} w_{u^-(j)}\varpi_j
\end{equation}
for any $u \in \N$.
Using this identity twice, we obtain 
\[
y_u \seq w_{u^+}\varpi_i - w_{u^-}\varpi_i + \sum_{j \sim i} \sfc_{ji} (w_{(u^+)^-(j)}\varpi_j - w_{u^-(j)}\varpi_j) =0. 
\]
First, consider the case when $u\le v$. If $u^+ \le v$, then we have 
\[
x_{u,v} = (y_u, \varpi_{i_v}+w_v\varpi_{i_v}) = 0.
\]
When $u \le v < u^+$, we have 
\begin{align*}
x_{u,v} &= (\varpi_i - w_{u^-}\varpi_i, \varpi_{i_v}+w_v \varpi_{i_v}) + (\varpi_i + w_{u^+} \varpi_i, \varpi_{i_v}-w_v\varpi_{i_v}) \\
&\quad +\sum_{j\sim i} \sfc_{ji} \left\{ (\varpi_j - w_{u^-(j)}\varpi_j, \varpi_{i_v}+w_v \varpi_{i_v}) + (\varpi_j + w_{(u^+)^-(j)}\varpi_j, \varpi_{i_v}-w_v\varpi_{i_v})\right\} \\
& = (-s_i\varpi_i + w_u \varpi_i, \varpi_{i_v}+w_v\varpi_{i_v}) + (-s_i\varpi_i-w_u\varpi_i,\varpi_{i_v}-w_v\varpi_{i_v}) \\ 
&= 2(-s_i\varpi_i, \varpi_{i_v}) + 2(\varpi_i, \varpi_{i_v}) \\
&= 2(\alpha_i, \varpi_{i_v}) \\
&=2\sfd_i\delta_{i,i_v}.
\end{align*}
Here the first equality follows from Lemma~\ref{Lem:Lambda} (note that $v < u^+ < (u^+)^-(j)^+$).
The second one follows from \eqref{eq:wspi}.  
The third one is because $(w_u \varpi_i, w_v \varpi_{i_v}) = (w_v \varpi_i, w_v \varpi_{i_v})=(\varpi_i, \varpi_{i_v})$. 
To conclude, we note that $\delta_{i,i_v} = \delta_{u,v}$ under the condition $u \le v < u^+$.

Next, assume that $u>v$. If $u^- \ge v$, we have 
\[
x_{u,v} = (y_u, \varpi_{i_v}-w_v\varpi_{i_v}) = 0.
\]
Here we use again Lemma~\ref{Lem:Lambda} (note that $(u^+)^-(j)^+\ge u^-(j)^+>u>v$). 

When $u > v >u^-$, we have 
\begin{align*}
x_{u,v} &= (\varpi_i - w_{u^-} \varpi_i, \varpi_{i_v}+w_v\varpi_{i_v})+
(\varpi_i+w_{u^+}\varpi_i, \varpi_{i_v}-w_v \varpi_{i_v})\\
& \quad +\sum_{j\sim i} \sfc_{ji}  \left\{ - (\varpi_j+w_{u^-(j)}\varpi_j, \varpi_{i_v}-w_v \varpi_{i_v})+(\varpi_j+w_{(u^+)^-(j)}\varpi_j, \varpi_{i_v}-w_v \varpi_{i_v})\right\} \\
&= (y_u+\varpi_i+w_{u^-}\varpi_i, \varpi_{i_v}-w_v \varpi_{i_v}) + (\varpi_i - w_{u^-} \varpi_i, \varpi_{i_v}+w_v\varpi_{i_v}) \\
& = 2 (\varpi_i,\varpi_{i_v})-2(w_{u^-}\varpi_i, w_v\varpi_{i_v}) \\
& = 2 (\varpi_i,\varpi_{i_v})-2(w_{v}\varpi_i, w_v\varpi_{i_v}) \\
&= 0.
\end{align*}
Here the first equality follows from Lemma~\ref{Lem:Lambda} (note that $(u^+)^-(j)^+\ge u^-(j)^+>u>v$).
The fourth one is deduced from our assumption $u > v >u^-$. 

The calculations above complete the proof. 
\end{proof}

Let us take $n \in \N$ and set
\[ J \seq [1,n], \quad J_f \seq \{ k\in [1,n]\mid k^+ > n\}, \quad J_e \seq J \setminus J_f.\]
By truncating $\tB_\bfi$ and $\Lambda_\bfi$, we get a $J \times J_e$-matrix $\tB^n_\bfi \seq (b_{u,v})_{u \in J, v \in J_e}$ and a $J \times J$-matrix $\Lambda^n_\bfi \seq (\Lambda_{u,v})_{u,v \in J}$.

\begin{Cor} \label{Cor:BL}
For any $n \in \N$, the pair $(\tB_\bfi^n, \Lambda_\bfi^n)$ defined as above forms a compatible pair. More precisely, we have
\[
\sum_{k \in J} b_{k,u}\Lambda_{k,v} = 2  \sfd_{i_u} \delta_{u,v} \quad \text{for $u \in J_e, v \in J$}.
\]
\end{Cor}
\begin{proof}
By the definition of $J_e$, we have $b_{k,u}=0$ for $u \in J_e$ and $k \in \N \setminus J$.
Therefore, the result follows from Proposition~\ref{Prop:BL}.
\end{proof}

\begin{Rem}
Corollary~\ref{Cor:BL} is well-known when $(i_1,\ldots,i_n)$ is a reduced word for an element of $\sW$ (see~\cite[Proposition 10.1, Lemma 11.3]{GLS13} and \cite[Proposition 10.4]{GY17}). Also Corollary~\ref{Cor:BL} gives an affirmative answer to \cite[Conjecture 1]{KO22}. 
\end{Rem}

\begin{Def} \label{Def:BL}
Let $\bfi = (i_1, \ldots, i_n)$ be an arbitrary finite sequence in $\Delta_0$.
We extend it to an infinite sequence $\tilde{\bfi} = (i_u)_{u \in \N}$ satisfying \eqref{eq:cond1} and define a compatible pair
\[ (\tB_\bfi, \Lambda_\bfi) \seq (\tB_{\tilde{\bfi}}^{n}, \Lambda_{\tilde{\bfi}}^{n}).\]
Clearly it is independent of the choice of extension $\tilde{\bfi} \in \Delta_0^{(\infty)}$ and hence it is well-defined.
\end{Def}

\subsection{Quantum cluster algebra $\cA_\bfi$} \label{Ssec:Ai}
For the generality of quantum cluster algebras, we refer the reader to Appendix~\ref{sec:QCA}. 
Let us denote by $\cA_\bfi = \cA_t(\tB_\bfi, \Lambda_\bfi)$ the quantum cluster algebra associated with the compatible pair $(\tB_\bfi, \Lambda_\bfi)$ in Proposition~\ref{Prop:BL}. 
It is a $\Z[t^{\pm 1/2}]$-subalgebra of the quantum torus $\cT(\Lambda_\bfi)$.
Here we recall that $\cT(\Lambda_\bfi)$ is the $\Z[t^{\pm 1/2}]$-algebra generated by $\{ X_u^{\pm 1}\}_{u \in \N}$ subject to the relations
\begin{itemize}
\item $X_u X_u^{-1} = X_u^{-1}X_u = 1$ for $u \in \N$,
\item $X_u X_v = t^{\Lambda_{u,v}}X_v X_u$ for $u,v \in \N$.
\end{itemize} 

For each $n \in \N$, we have the quantum cluster algebra $\cA_\bfi^n \seq \cA_t(\tB_\bfi^n, \Lambda_\bfi^n)$ with ambient quantum torus $\cT(\Lambda_\bfi^n)$. 
We naturally regard $\cT(\Lambda_\bfi^n)$ as the $\Z[t^{\pm 1/2}]$-subalgebra of $\cT(\Lambda_\bfi)$ generated by $\{X_u^{\pm 1}\}_{u \in [1,n]}$. 
Thus we have the inclusions
\[ \cA_\bfi^1 \subset \cA_\bfi^2 \subset \cdots \subset \cA_\bfi \subset \cT(\Lambda_\bfi) \quad \text{such that } \bigcup_{n \in \N} \cA_\bfi^n = \cA_\bfi.\]
In particular, each cluster variable (resp.~monomial) of $\cA_\bfi$ is a cluster variable (resp.~monomial) of $\cA_\bfi^n$ for $n \in \N$ large enough.  

\section{Relations among $\cA_\bfi$'s}\label{sectwo}

In this section, we study the relation between the quantum cluster algebras $\cA_\bfi$ and $\cA_{\bfi'}$ when $\bfi$ and $\bfi'$ are related by simple operations  (commutation moves, 
braid moves and forward shifts in respective Sections \ref{cmov}, \ref{bmov} and \ref{fshif}). 
In particular, we obtain isomorphisms and embeddings of quantum cluster algebras which are relevant for our purposes. We ``keep track" of the degrees ($g$-vectors) 
of cluster monomials through these transformations which will be useful in the sequel.

In the present and the following sections, we take $\Delta$ as a Dynkin diagram of finite type $\mathrm{ADE}$. 
With each sequence $\bfi = (i_u)_{u \in \N} \in \Delta_0^{(\infty)}$, we associate the infinite quiver $\Gamma_\bfi$ defined as follows:
The set of vertices of $\Gamma_\bfi$ is simply $\N$. 
For $u,v \in \N$, we assign an arrow $u \to v$ in $\Gamma_\bfi$ when either
\begin{itemize}   
\item $i_u \sim i_v$ and $u < v < u^+ < v^+$, or
\item $i_u = i_v$ and $u = v^+$.
\end{itemize}
Then, we have
\begin{equation} \label{eq:tBdef}
b_{u,v} = |\{\text{arrows $v \to u$ in $\Gamma_\bfi$}\}| - |\{\text{arrows $u \to v$ in $\Gamma_\bfi$}\}|.
\end{equation}

For a set $X$, we denote by $\SG_X$ the group of permutations of $X$. 
For a matrix $A = (A_{u,v})_{u,v \in X}$ and a permutation $\pi \in \SG_X$, we set $
\pi A \seq (A_{\pi^{-1}(u), \pi^{-1}(v)})_{u,v \in X}$.
For each $k \in \N$, let $\sigma_{k} \in \SG_\N$ be the simple transposition of $k$ and $k+1$. 

In what follows, we work with two sequences $\bfi = (i_u)_{u \in \N}, \bfi' = (i'_u)_{u \in \N} \in \Delta_0^{(\infty)}$.

\subsection{Commutation moves}\label{cmov}
First, we consider the case of a simple permutation of two successive elements of the sequence which are not adjacent in the Dynkin diagram. 
Precisely, it is the case when $(i'_k, i'_{k+1}) = (i_{k+1}, i_k)$ with $i_k \neq i_{k+1}$ and $i_k \not \sim i_{k+1}$ for some $k \in \N$, and $i_u = i'_u$ for $u \not \in \{k, k+1\}$. 
In this case, we write $\bfi' = \gamma_k \bfi$.

\begin{Lem} \label{Lem:gamma}
When $\bfi' = \gamma_k \bfi$, we have $\tB_{\bfi'} = \sigma_k \tB_\bfi$ and $\Lambda_{\bfi'} = \sigma_k \Lambda_\bfi$.
\end{Lem}
\begin{proof}
Thanks to the commutation relation $s_{i_k} s_{i_{k+1}} = s_{i_{k+1}}s_{i_k}$, it is immediate that the permutation $\sigma_k$ induces a quiver isomorphism $\Gamma_\bfi \simeq \Gamma_{\bfi'}$. 
Therefore we obtain $\sigma_k \tB_\bfi = \tB_{\bfi'}$.  
To show $\sigma_k \Lambda_\bfi = \Lambda_{\bfi'}$, let $\Lambda_\bfi = (\Lambda_{u,v})_{u,v \in \N}$ and $\Lambda_{\bfi'} = (\Lambda'_{u,v})_{u,v \in \N}$.
For $u,v \not \in \{k,k+1\}$, we have to prove the followings:
\begin{itemize}
\item[(i)] $\Lambda_{u,v} = \Lambda'_{u,v}$, 
\item[(ii)] $\Lambda_{u,k} = \Lambda'_{u,k+1}$, 
\item[(iii)] $\Lambda_{u,k+1} = \Lambda'_{u,k}$,
\item[(iv)] $\Lambda_{k,k+1} = \Lambda'_{k+1,k}$.
\end{itemize} 
For $a \neq k$, we have $w^\bfi_a = w^{\bfi'}_a$, from which (i) follows.
Moreover, we observe $w^\bfi_k \varpi_{i_k} = w^{\bfi}_{k+1}\varpi_{i_k} = w^{\bfi'}_{k+1}\varpi_{i'_{k+1}}$, from which (ii) follows. 
(iii) is dual to (ii).
To see (iv), we may use Lemma~\ref{Lem:Lambda}. 
\end{proof}

Lemma~\ref{Lem:gamma} implies the following (see Appendix~\ref{Ssec:perm}).

\begin{Prop}
Assume that $\bfi, \bfi' \in \Delta_0^{(\infty)}$ are related by $\bfi' = \gamma_k \bfi$ for some $k \in \N$.
Then, we have an isomorphism of $\Z[t^{\pm 1/2}]$-algebras
\[ \gamma_k^* (= \sigma_k^*) \colon \cA_{\bfi'} \simeq \cA_\bfi \qquad \text{given by $X_u \mapsto X_{\sigma_k(u)}$ for all $u \in \N$},\] 
which induces a bijection between the sets of quantum cluster monomials.
\end{Prop}

\begin{Rem}
When $n \neq k$, the isomorphism $\gamma_k^*$ restricts to the isomorphism
\[ \gamma_{k}^* \colon \cA_{\bfi'}^{n} \simeq \cA_{\bfi}^{n}.\]
\end{Rem}

Now, let us explain how the degrees of cluster monomials get modified.

\begin{Lem} \label{Lem:cg}
Assume that $\bfi, \bfi' \in \Delta_0^{(\infty)}$ are related by $\bfi' = \gamma_k \bfi$ for some $k \in \N$.
If $x \in \cA_{\bfi'}$ is a cluster monomial whose degree is $\bfg' = (g'_u)_{u \in \N} \in \Z^{\oplus \N}$, then the element $\gamma_k^* x \in \cA_\bfi$ is the cluster monomial whose degree $\bfg = (g_u)_{u \in \N}$ is given by
\begin{equation} \label{eq:cg} 
g_u = \begin{cases}
g'_{k+1} & \text{if $u=k$}, \\
g'_k & \text{if $u=k+1$}, \\
g'_u & \text{otherwise}.
\end{cases}
\end{equation}
\end{Lem}
\begin{proof}
Immediate from the construction.
\end{proof}      

\begin{Def}[Commutation equivalence] \label{Def:comm-eq}
Let $J = \N$, or $J = [1,n]$ for some $n \in \N$.
Let $\pi \in \SG_J$ be a permutation.
We say that two sequences $\bfi = (i_1,i_2,\ldots)$ and $\bfi' = (i'_1, i'_2, \ldots) \in \Delta_0^J$ are \emph{commutation-equivalent by $\pi$} if  $i_{u} = i'_{\pi(u)}$ for all $u \in J$ and we have [$i_u \neq i_v$ and $i_u \not \sim i_v$] whenever [$u < v$ and $\pi(u) > \pi(v)$]. 
We simply say that $\bfi$ and $\bfi'$ are \emph{commutation-equivalent} if they are commutation-equivalent by some $\pi \in \SG_J$.
\end{Def}

The above discussion can be generalized as follows. 

\begin{Prop} \label{Prop:comm-eq}
Let $J = \N$, or $J = [1,n]$ for some $n \in \N$.
Assume that two sequences $\bfi$ and $\bfi'$ are commutation-equivalent by a permutation $\pi \in \SG_J$.
Then, we have $(\tB_{\bfi'}, \Lambda_\bfi') = (\pi\tB_{\bfi}, \pi\Lambda_{\bfi})$ and hence we have the isomorphism of $\Z[t^{\pm 1/2}]$-algebras $\pi^* \colon \cA_{\bfi'} \simeq \cA_{\bfi}$, which induces a bijection between the sets of quantum cluster monomials.
\end{Prop}

\subsection{Braid moves}\label{bmov}

Next, we consider the case of a braid transformation of three 
successive elements of the sequence. Precisely, we assume that $(i_k, i_{k+1}, i_{k+2}) = (i,j,i)$ with $i\sim j$ for some $k \in \N$, and $\bfi'$ is obtained from $\bfi$ by applying the braid move $(i,j,i) \leadsto (j,i,j)$. Namely we have $i'_l = i_l$ for $l \not\in \{k,k+1,k+2\}$ and 
$(i'_k, i'_{k+1}, i'_{k+2}) = (j,i,j)$.
In this case, we write $\bfi' = \beta_k \bfi$.
Let $\mu_k$ denote the mutation at $k$.

\begin{Lem} \label{Lem:beta1}
When $\bfi' = \beta_k \bfi$, the permutation $\sigma_{k+1}$ induces a quiver isomorphism $\mu_k\Gamma_\bfi \simeq \Gamma_{\bfi'}$. In other words, we have $\tB_{\bfi'} = \sigma_{k+1} \mu_k \tB_\bfi$. 
\end{Lem}
\begin{proof}
We put $K = \{k, k+1, k+2, k^-_\bfi, (k+1)^-_\bfi\} \cap \N$.
This is the neighborhood of $k$ in $\Gamma_\bfi$, that is, the set of vertices which are equal or adjacent to $k$ in the (underlying graph of the) quiver $\Gamma_\bfi$.  
Since $k^-_\bfi = (k+1)^-_{\bfi'}$ and $(k+1)^-_\bfi = k^-_{\bfi'}$, the set $K$ is also the neighborhood of $k$ in the quiver $\Gamma_{\bfi'}$ and that of $\sigma_{k+1}\Gamma_{\bfi'}$.
Note that the mutation $\mu_k$ only changes the full subquiver $\Gamma_\bfi |_K$, and we have 
\[ \mu_k\Gamma_\bfi \setminus (\mu_k\Gamma_\bfi|_K) = \Gamma_\bfi \setminus (\Gamma_\bfi|_K) = \sigma_{k+1}\Gamma_{\bfi'} \setminus (\sigma_{k+1}\Gamma_{\bfi'}|_K).\]
Therefore, we only have to check $\mu_k \Gamma_\bfi |_K = \sigma_{k+1}\Gamma_{\bfi'}|_K$.
We only consider the case when both $k^-_\bfi$ and $(k+1)^-_\bfi$ are nonzero since the proof for the other case is similar, or rather simpler.
In this case, the quiver $\Gamma_\bfi |_K$  is depicted as:
\[
\Gamma_\bfi|_K = \left(
\vcenter{\xymatrix@R=12pt{
(k+1)^-_\bfi \ar[rd] && \ar[ll] k+1 \ar@{-->}[dd]^-{\beta}\\
 & k \ar[ru] \ar[ld]& \\
k^-_\bfi \ar@{-->}[uu]^-{\alpha} & & k+2 \ar[lu]
}}\right),
\]
where we have the arrow $\alpha$ if and only if $k^-_\bfi < (k+1)^-_\bfi$, and we have the arrow $\beta$ if and only if $(k+1)^+_\bfi < (k+2)^+_\bfi$.
Its mutation at $k$ is:
\[
\mu_k \Gamma_\bfi|_K = \left(
\vcenter{
\xymatrix@R=12pt{
(k+1)^-_\bfi \ar@<0.5ex>[dd] && \ar[ld] k+1 \ar@{-->}@<0.5ex>[dd]^-{\beta}\\
 & k \ar[lu] \ar[rd]& \\
k^-_\bfi \ar[ru] \ar@<0.5ex>@{-->}^-{\alpha}[uu] & & k+2 \ar[ll] \ar@<0.5ex>[uu]
}}\right)
=\left(
\vcenter{
\xymatrix@R=12pt{
(k+1)^-_\bfi \ar@{-->}[dd]_-{\alpha'} && \ar[ld] k+1 \\
 & k \ar[lu] \ar[rd]& \\
k^-_\bfi \ar[ru] & & k+2 \ar[ll] \ar@{-->}[uu]_-{\beta'}
}}\right), 
\]
where we have the arrow $\alpha'$ if and only if $k^-_\bfi > (k+1)^-_\bfi$, and we have the arrow $\beta'$ if and only if $(k+1)^+_\bfi > (k+2)^+_\bfi$.
On the other hand, the quiver $\Gamma_{\bfi'}|_K$ is depicted as:
\[
\Gamma_{\bfi'}|_K =
\left(
\vcenter{
\xymatrix@R=12pt{
k^-_{\bfi'} \ar@{-->}[dd]_-{\alpha''} && k+2 \ar[dl]\\
 & k \ar[lu] \ar[rd]& \\
(k+1)^-_{\bfi'} \ar[ru] & & k+1 \ar@{-->}[uu]_-{\beta''} \ar[ll]
}}\right),
\]
where we have the arrow $\alpha''$ if and only if $k^-_{\bfi'} < (k+1)^-_{\bfi'}$, and we have the arrow $\beta''$ if and only if $(k+1)^+_{\bfi'} < (k+2)^+_{\bfi'}$.
Since $(k^-_{\bfi'}, (k+1)^-_{\bfi'}) = ((k+1)^-_\bfi, k^-_\bfi)$ and $((k+1)^+_{\bfi'}, (k+2)^+_{\bfi'}) = ((k+2)^+_\bfi, (k+1)^+_\bfi)$, we find $\mu_k\Gamma_\bfi|_K=\sigma_{k+1}\Gamma_{\bfi'}|_K$ as desired.
\end{proof}
 
\begin{Lem} \label{Lem:beta2}
When $\bfi' = \beta_k \bfi$, we have $\Lambda_{\bfi'} = \sigma_{k+1}\mu_k \Lambda_\bfi$.
\end{Lem}
\begin{proof}
Let $(i,j) \seq (i_k, i_{k+1})$.
As before, we put $\Lambda_\bfi = (\Lambda_{u,v})_{u,v \in \N}$ and $\Lambda_{\bfi'} = (\Lambda'_{u,v})_{u,v \in \N}$.
We have to show $\Lambda'_{u,v} = (\mu_k \Lambda_\bfi)_{\sigma_{k+1}(u),\sigma_{k+1}(v)}$ for $u<v$.
It divides into the following $10$ cases:
\begin{itemize}
\item[(i)] $\Lambda'_{u,v}=\Lambda_{u,v}$ for $u,v \not \in \{k, k+1, k+2\}$,
\item[(ii)] $\Lambda'_{k+2,v} = \Lambda_{k+1,v}$ for $v > k+2$,
\item[(iii)] $\Lambda'_{k+1,v} = \Lambda_{k+2,v}$ for $v > k+2$,
\item[(iv)] $\Lambda'_{k+1, k+2} = \Lambda_{k+2, k+1}$,
\item[(v)] $\Lambda'_{u,k+2} = \Lambda_{u,k+1}$ for $u < k$,
\item[(vi)] $\Lambda'_{u, k+1} = \Lambda_{u, k+2}$ for $u < k$,
\item[(vii)] $\Lambda'_{k,v} = \Lambda_{k^-(j),v}+\Lambda_{k+2,v}-\Lambda_{k,v}$ for $v > k+2$,
\item[(viii)] $\Lambda'_{k,k+2} = \Lambda_{k^-(j), k+1} + \Lambda_{k+2,k+1}-\Lambda_{k,k+1}$, 
\item[(ix)] $\Lambda'_{k,k+1} = \Lambda_{k^-(j), k+2} - \Lambda_{k,k+2}$,
\item[(x)] $\Lambda'_{u,k} = \Lambda_{u,k^-(j)} + \Lambda_{u,k+2} - \Lambda_{u,k}$ for $u < k$.
\end{itemize}
Here $k^-(j) \seq k^-_\bfi(j) = k^-_{\bfi'}(j)$. 
Recall the braid relation $s_i s_j s_i = s_j s_i s_j$.
Then we have $w_a^\bfi = w_a^{\bfi'}$ for any $a \in \N \setminus \{k, k+1\}$, from which (i) follows. 
Also, we observe $w^\bfi_{k+2}\varpi_i = w^{\bfi'}_{k+2}\varpi_i = w^{\bfi'}_{k+1}\varpi_i$, from which (ii) and (v) follow. (iii) and (vi) are dual to them. To check (iv), we may also use Lemma~\ref{Lem:Lambda}.  
To check (vii), we observe $(s_i - s_i s_j s_i)\varpi_i = \alpha_j$ and hence 
\begin{align*}
x &\seq \varpi_j - w^\bfi_{k^-(j)} \varpi_j - w^\bfi_{k+2}\varpi_i + w^\bfi_k\varpi_i \\
&= \varpi_j - w^\bfi_{k-1}( \varpi_j - (s_i - s_i s_j s_i)\varpi_i ) \\
&= \varpi_j - w^\bfi_{k-1}(\varpi_j - \alpha_j) \\
&= \varpi_j - w^{\bfi'}_k\varpi_j.
\end{align*}
Therefore we have
\[ \Lambda_{k^-(j),v}+\Lambda_{k+2,v}-\Lambda_{k,v} = (x, \varpi_{i_v}+w^\bfi_v \varpi_{i_v}) = (\varpi_j - w^{\bfi'}_k\varpi_j, \varpi_{i'_v}+w^{\bfi'}_v \varpi_{i'_v}) = \Lambda'_{k,v}.\]
We can check the remaining (viii), (ix) and (x) in a similar way with the help of Lemma~\ref{Lem:Lambda} (and a trivial fact $\Lambda_{k+2,k+2} = 0$ for (ix)).    
\end{proof}

Lemmas~\ref{Lem:beta1} \& \ref{Lem:beta2} show the equality $(\Lambda_{\bfi'}, \tB_{\bfi'}) = \sigma_{k+1}\mu_k (\Lambda_{\bfi}, \tB_{\bfi})$ in the notation of Appendix~\ref{sec:QCA}.
Therefore we have the following.

\begin{Prop}
Assume that $\bfi, \bfi' \in \Delta_0^{(\infty)}$ are related by $\bfi' = \beta_k \bfi$ for some $k \in \N$. 
Then, we have an isomorphism of $\Z[t^{\pm 1/2}]$-algebras
\[ \beta_k^*(= \mu_k^* \sigma_{k+1}^*)\colon \cA_{\bfi'} \simeq \cA_\bfi \qquad \text{given by $X_u \mapsto \mu_{k}^*(X_{\sigma_{k+1}(u)})$ for all $u \in \N$},\]
which induces a bijection between the sets of quantum cluster monomials. 
\end{Prop}

\begin{Rem}
When $n \not \in \{ k, k+1\}$, the isomorphism $\beta_k^*$ restricts to the isomorphism
\[ \beta_{k}^*  \colon \cA_{\bfi'}^{n} \simeq \cA_{\bfi}^{n}.\]
\end{Rem}

Now, let us explain how the degrees of cluster monomials get modified.

\begin{Lem} \label{Lem:bg}
Assume that $\bfi, \bfi' \in \Delta_0^{(\infty)}$ are related by $\bfi' = \beta_k \bfi$ for some $k \in \N$.
If $x \in \cA_{\bfi'}$ is a cluster monomial whose degree is $\bfg' = (g'_u)_{u \in \N} \in \Z^{\oplus \N}$, then the element $\beta_k^* x \in \cA_\bfi$ is the cluster monomial whose degree $\bfg = (g_u)_{u \in \N}$ is given by
\begin{equation} \label{eq:bg}
g_u = \begin{cases}
-g'_k & \text{if $u=k$}, \\
g'_{\sigma_{k+1}(u)} + \max(g'_k,0) & \text{if $u \in \{ k+2, (k+1)^-_{\bfi}\}$}, \\
g'_{\sigma_{k+1}(u)} + \min(g'_k,0) & \text{if $u \in \{ k+1, k^-_{\bfi}\}$}, \\
g'_u & \text{otherwise}.
\end{cases}
\end{equation}  
\end{Lem}      
\begin{proof}
Apply the formula in Theorem~\ref{Thm:trop}.
\end{proof}

\subsection{Forward shifts}\label{fshif}

Now, we consider a global forward shift of the elements of the sequence. 
Precisely, assume that the sequences $\bfi = (i_u)_{u \in \N}, \bfi' = (i'_u)_{u \in \N} \in \Delta_0^{(\infty)}$ satisfy
\[ i'_u = i_{u+1} \quad \text{for all $u \in \N$}. \] 
In this case, we write $\bfi' = \partial_+ \bfi$.
Let $i \seq i_1$.
Define an increasing sequence of positive integers $(x_n)_{n \in \N}$ by the condition $x_1 = 1$ and $x_n = (x_{n-1})_{\bfi}^+$ for $n >1$, so that $\{ x_n \}_{ n \in \N } = \{k \in \N \mid i_k = i\}$.  
In the following lemmas, we consider the infinite mutation sequence corresponding to  $(x_1, x_2, \ldots)$.
Let $\sigma_+ \in \SG_\N$ be the permutation defined by 
\[ \sigma_+(k) =\begin{cases}
k^+_\bfi-1 & \text{if $i_k = i$}, \\
k-1 & \text{if $i_k \neq i$},
\end{cases} \]
whose inverse $\sigma_+^{-1}$ is given by 
\[ \sigma_+^{-1}(k) =\begin{cases}
(k+1)^-_\bfi & \text{if $i_{k+1} = i$}, \\
k+1 & \text{if $i_{k+1} \neq i$}.
\end{cases} \]

\begin{Lem} \label{Lem:shiftB}
Assume $\bfi' = \partial_+ \bfi$ and keep the above notation.
Then, the limit 
\[ \lim_{n \to \infty} \mu_{x_n}\cdots \mu_{x_2}\mu_{x_1} \tB_\bfi \]
yields a well-defined matrix $\mu_+\tB_\bfi$ and we have $\sigma_+ \mu_+ \tB_\bfi = \tB_{\bfi'}$.
\end{Lem}
\begin{proof}
For $j \in \Delta_0$, we put $X_j \seq \{ k \in \N \mid i_k = j \}$. 
By definition, we have $X_i = \{x_n\}_{n \in \N}$.
Take an index $j \in \Delta_0$ adjacent to $i$, and consider the map $X_i \to X_j \cup \{0\}$ given by $x_n \mapsto (x_n)_\bfi^-(j)$. 
Let $\{ 0 = y_0 < y_1 < y_2 < \cdots\}$ be the natural ordering of the image of this map.  
For each $n \in \N_0$, let $l_n$ be the largest integer $(>0)$ satisfying $(x_{l_n})_\bfi^-(j) = y_n$. 
The quiver $\Gamma_\bfi$ restricted to $X_{ij} \seq X_i \cup X_j$ is depicted as:
\[\vcenter{
\xymatrix@!C=4mm{x_1 \ar@{<-}[r] & \cdots \ar@{<-}[r] & x_{l_0-1} \ar@{<-}[r] & x_{l_0} \ar@{<-}[r] \ar[dr]& x_{l_0+1} \ar@{<-}[r]& \cdots \ar@{<-}[r] & x_{l_1} \ar@{<-}[r] \ar[rd]& x_{l_1+1} \ar@{<-}[r] & \cdots \ar@{<-}[r] & x_{l_2} \ar@{<-}[r] & \cdots & \\
&&&& y_1 \ar@{<-}[r] \ar[rru]& (y_1)^+_\bfi \ar@{<-}[r] & \cdots \ar@{<-}[r] & y_2 \ar@{<-}[r] \ar[rru]& (y_2)^+_\bfi \ar@{<-}[r] & \cdots  &  
}}.
\] 
We easily see that the mutated quiver $\mu_{x_{l_0-1}} \cdots \mu_{x_1}(\Gamma_\bfi |_{X_{ij}})$ is:
\[\vcenter{
\xymatrix@!C=4mm{x_1 \ar@{<-}[r] & \cdots \ar@{<-}[r] & x_{l_0-1} \ar[r] & x_{l_0} \ar@{<-}[r] \ar[dr]& x_{l_0+1} \ar@{<-}[r]& \cdots \ar@{<-}[r] & x_{l_1} \ar@{<-}[r] \ar[rd]& x_{l_1+1} \ar@{<-}[r] & \cdots \ar@{<-}[r] & x_{l_2} \ar@{<-}[r] & \cdots & \\
&&&& y_1 \ar@{<-}[r] \ar[rru]& (y_1)^+_\bfi \ar@{<-}[r] & \cdots \ar@{<-}[r] & y_2 \ar@{<-}[r] \ar[rru]& (y_2)^+_\bfi \ar@{<-}[r] & \cdots  &  
}}.
\]  
Then, applying Lemma~\ref{Lem:mu} below, the quiver $\mu_{x_{l_1-1}} \cdots \mu_{x_1}(\Gamma_\bfi |_{X_{ij}})$ is depicted as:
\[\vcenter{
\xymatrix@!C=4mm{x_1 \ar@{<-}[r] & \cdots \ar@{<-}[r] & x_{l_0-1} \ar@{<-}[r] \ar@{-->}[drr]& x_{l_0} \ar@{<-}[r]& \cdots \ar@{<-}[r]& x_{l_1-1} \ar[r] & x_{l_1} \ar@{<-}[r] \ar[rd]& x_{l_1+1} \ar@{<-}[r] & \cdots \ar@{<-}[r] & x_{l_2} \ar@{<-}[r] & \cdots & \\
&&&& y_1 \ar@{<-}[r] \ar[ru]& (y_1)^+_\bfi \ar@{<-}[r] & \cdots \ar@{<-}[r] & y_2 \ar@{<-}[r] \ar[rru]& (y_2)^+_\bfi \ar@{<-}[r] & \cdots  &  
}},
\] 
where we have the dashed arrow if and only if $l_0 > 1$. 
By repetition, we finally find that the limit $\lim_{n \to \infty}\mu_{x_n} \cdots \mu_{x_1}(\Gamma_\bfi |_{X_{ij}})$ gives a well-defined quiver depicted as:
\[\vcenter{
\xymatrix@!C=4mm{x_1 \ar@{<-}[r] & \cdots \ar@{<-}[r] & x_{l_0-1} \ar@{<-}[r] \ar@{-->}[drr]& x_{l_0} \ar@{<-}[r]& \cdots \ar@{<-}[r]& x_{l_1-1} \ar@{<-}[r] \ar[rrd] & x_{l_1} \ar@{<-}[r] & \cdots \ar@{<-}[r] & x_{l_2-1} \ar@{<-}[r] & x_{l_2} \ar@{<-}[r] & \cdots & \\
&&&& y_1 \ar@{<-}[r] \ar[ru]& (y_1)^+_\bfi \ar@{<-}[r] & \cdots \ar@{<-}[r] & y_2 \ar@{<-}[r] \ar[ru]& (y_2)^+_\bfi \ar@{<-}[r] & \cdots  &  
}},
\] 
which is isomorphic to $\Gamma_{\bfi'}|_{\sigma_+(X_{ij})}$. 
Moreover, for each $n \in \N$, there are no arrows going into the vertex $x_n$ in the quiver $\mu_{x_{n-1}} \cdots \mu_{x_1}\Gamma_{\bfi}$ other than $x_{n-1} \to x_n$ and $x_{n+1} \to x_n$ (see the latter assertion of Lemma~\ref{Lem:mu} below).
Thus, we have 
\[\mu_{x_n} \cdots \mu_{x_1}(\Gamma_\bfi |_{X_{ij}}) = (\mu_{x_n} \cdots \mu_{x_1} \Gamma_\bfi)|_{X_{ij}}.\]
Since the mutations $\{\mu_{x_1}, \mu_{x_2}, \ldots\}$ change only the full subquivers supported on $\bigcup_{j \sim i}X_{ij}$, we find that the limit $\mu_+ \Gamma_\bfi \seq \lim_{n \to \infty} \mu_{x_n} \cdots \mu_{x_1} \Gamma_\bfi$ is well-defined. 
Moreover, the permutation $\sigma_+$ gives an isomorphism of the quivers $\sigma_{+} \colon \mu_+ \Gamma_\bfi \simeq \Gamma_{\bfi'}$.
Thus we obtain the conclusion.
\end{proof}

\begin{Lem} \label{Lem:mu}
Let $n \in \N$ and consider a quiver $Q$ given by
\[ Q = \left(\vcenter{\xymatrix@!C=4mm{
0 \ar[r]& 1 \ar@{<-}[r] & 2 \ar@{<-}[r] & \cdots \ar@{<-}[r] & n \ar@{<-}[r]& n+1 \\
&&& \ar@{<-}[llu] \bullet \ar[rru]& 
}}
\right).\]
Then, we have
\[ \mu_n \cdots  \mu_2 \mu_1 Q = \left(\vcenter{\xymatrix@!C=4mm{
0 \ar@{<-}[r]& 1 \ar@{<-}[r] & 2 \ar@{<-}[r] & \cdots \ar@{<-}[r] & n \ar[r]& n+1 \\
&& \ar@{<-}[llu] \bullet \ar[rru]& 
}}
\right).\]
Moreover, for each $k \in [1,n]$, there are no arrows going into the vertex $k$ in the quiver $\mu_{k-1} \cdots \mu_1Q$ other than $k-1 \to k$ and $k+1 \to k$.
\end{Lem}
\begin{proof} The proof is straightforward:
\begin{align*} \mu_1Q &= \left(\vcenter{\xymatrix@!C=4mm{
0 \ar@{<-}[r] \ar[rrrd]& 1 \ar[r] & 2 \ar@{<-}[r] \ar[rd]& 3 \ar@{<-}[r] & \cdots \ar@{<-}[r]& n+1 \\
&&& \ar[llu] \bullet \ar[rru]& 
}}
\right), \allowdisplaybreaks \\
\mu_2\mu_1Q &= \left(\vcenter{\xymatrix@!C=4mm{
0 \ar@{<-}[r] \ar[rrrd]& 1 \ar@{<-}[r] & 2 \ar[r] \ar@{<-}[rd]& 3 \ar@{<-}[r] \ar[d]& \cdots \ar@{<-}[r]& n+1 \\
&&& \bullet \ar[rru]& 
}}
\right), \allowdisplaybreaks\\
\cdots & \allowdisplaybreaks\\
\mu_{n-1} \cdots \mu_1Q &= \left(\vcenter{\xymatrix@!C=4mm{
0 \ar@{<-}[r] \ar[rrrd]& 1 \ar@{<-}[r] & \cdots \ar@{<-}[r] & n-1 \ar[r] \ar@{<-}[d]& n \ar@{<-}[r] \ar[ld]& n+1 \\
&&& \bullet \ar[rru]& 
}}
\right). \qedhere
\end{align*}
\end{proof}

\begin{Lem} \label{Lem:shiftL}
Assume $\bfi' = \partial_+ \bfi$ and keep the above notation.
Then, the limit 
\[ \lim_{n \to \infty} \mu_{x_n}\cdots \mu_{x_2}\mu_{x_1} \Lambda_\bfi \]
yields a well-defined matrix $\mu_+\Lambda_\bfi$ and we have $\sigma_+ \mu_+ \Lambda_\bfi = \Lambda_{\bfi'}$.
\end{Lem}
\begin{proof}
We put $\Lambda_\bfi = (\Lambda_{u,v})_{u,v \in \N}$ and $\Lambda_{\bfi'} = (\Lambda'_{u,v})_{u,v \in \N}$.
As we have seen in the proof of Lemma~\ref{Lem:shiftB}, there are no arrows going into the vertex $x_k$ in the quiver $\mu_{x_{k-1}} \cdots \mu_{x_1}\Gamma_{\bfi}$ other than $x_{k-1} \to x_k$ and $x_{k+1} \to x_k$ for each $k \in \N$.
Thus, we have 
\[\mu_{x_n} \cdots \mu_{x_2}\mu_{x_1}\Lambda_\bfi = (E^{(1)} E^{(2)} \cdots E^{(n)})^T \Lambda_\bfi (E^{(1)} E^{(2)} \cdots E^{(n)}),\] where the matrix $E^{(k)} = (e^{(k)}_{u,v})_{u,v \in \N}$ is given by 
\[ 
e^{(k)}_{u,v} = \begin{cases}
\delta_{u,v} & \text{if $v \neq x_k$}, \\
-\delta_{u,x_{k}} + \delta_{u,x_{k-1}} + \delta_{u,x_{k+1}} & \text{if $v = x_k$}.
\end{cases}
\]
Therefore, the limit $\lim_{n \to \infty} E^{(1)} E^{(2)} \cdots E^{(n)}$ yields a well-defined matrix $E$ and hence 
\[\mu_+\Lambda_\bfi = \lim_{n \to \infty} \mu_{x_n}\cdots \mu_{x_1} \Lambda_\bfi = E^T \Lambda E\] is also well-defined.
More precisely, letting $E = (e_{u,v})_{u,v \in \N}$, we have
\[ e_{u,v} = \sum_{a_1, \ldots, a_{n-1} \in \N} e^{(1)}_{u,a_1} e^{(2)}_{a_1, a_2} \cdots e^{(n)}_{a_{n-1},v},\]
where $n \in \N$ is an integer such that $v \le x_n$. 
From this, we find that $e_{u,v} = \delta_{u,v}$ if $i_v \neq i \seq i_1$, and that $e_{u,x_n} = -\delta_{u,x_n} + e_{u, x_{n-1}} + \delta_{u,x_{n+1}}$.
By the latter equality, we further compute 
\[ e_{u,x_n} - \delta_{u, x_{n+1}} = e_{u,x_{n-1}} - \delta_{u,x_n} = \cdots = e_{u,x_1} - \delta_{u,x_2} = - \delta_{u,1}, \]
where the last equality follows since $e_{u,x_1} = e^{(1)}_{u,1} = -\delta_{u,1} + \delta_{u,x_2}$.
Thus, we get
\begin{equation} \label{eq:edelta}
e_{u,v} = \begin{cases}
\delta_{u,v} & \text{if $i_v \neq i$}, \\
\delta_{u, v^+} - \delta_{u,1}& \text{if $i_v = i$},
\end{cases}
\end{equation}
where $v^+ = v^+_\bfi$.
Therefore, we have
\[  (\mu_+ \Lambda_\bfi)_{u,v} = (E^T \Lambda_\bfi E)_{u,v} = \begin{cases}
\Lambda_{u,v} & \text{if $i \not \in \{i_u, i_v \}$}, \\
\Lambda_{u^+, v} - \Lambda_{1,v} & \text{if $i_u = i$ and $i_v \neq i$}, \\
\Lambda_{u,v^+} - \Lambda_{u,1} & \text{if $i_u \neq i$ and $i_v = i$}, \\
\Lambda_{u^+, v^+} - \Lambda_{1,v^+} - \Lambda_{u^+, 1} & \text{if $i_u = i_v = i$}.
\end{cases}
\]
From this, the desired equality $(\mu_+ \Lambda_\bfi)_{u,v}  = \Lambda'_{\sigma_+(u), \sigma_+(v)}$ follows for any $u,v \in \N$.  
For example, when $i_u = i_v = i$ and $u<v$, we have
\begin{align*}
 \Lambda'_{\sigma_+(u), \sigma_+(v)} &= (\varpi_i -s_i w^\bfi_{u^+} \varpi_i, \varpi_i + s_i w^\bfi_{v^+}\varpi_i) \\
&= ( (1 - s_i )\varpi_i + s_i (1-w^\bfi_{u^+}) \varpi_i, (1-s_i)\varpi_i + s_i(1+ w^\bfi_{v^+})\varpi_i) \\
&= ((1-s_i)\varpi_i, (1-s_i)\varpi_i) - ((1-w^\bfi_{u^+})\varpi_i, (1-s_i)\varpi_i) \\
& \qquad - ((1-s_i)\varpi_i, (1+w^\bfi_{v^+})\varpi_i) + ((1-w^\bfi_{u^+})\varpi_i, (1+w^\bfi_{v^+})\varpi_i) \\
&= \Lambda_{1, u^+} - \Lambda_{1, v^+} + \Lambda_{u^+, v^+}.
\end{align*}
Here, for the last equality, we used the identity
\[((1-s_i)\varpi_i, (1-s_i)\varpi_i) = ((1-s_i)\varpi_i, 2\varpi_i).\] 
For the other case, we can check the desired equality with a similar computation.
\end{proof}

For each $u \in \N$, there is $n_u \in \N$ such that $x_{n_u} > u$ since $\{x_n\}_{n \in \N}$ is increasing. 
Then, we have $\mu_{x_{n}}^*(X_u) = X_u$ if $n \ge n_u$.
We define a homomorphism of $\Z[t^{\pm 1/2}]$-algebras
\[ \mu_+^* \colon \cA_t(\mu_+(\Lambda_\bfi, \tB_\bfi)) \to \cA_t(\Lambda_\bfi, \tB_\bfi)\]
by $\mu_+^*(X_u) \seq \mu_{x_1}^*\mu_{x_2}^* \cdots \mu_{x_{n_u}}^*(X_u)$ for all $u \in \N$.
Note that $\mu_+^*(X_u)$ is independent of the choice of $n_u$ with $x_{n_u} > u$, and hence $\mu_+^*$ is well-defined.
By construction, it sends each quantum cluster monomial to a quantum cluster monomial. 

Lemmas~\ref{Lem:shiftB} \& \ref{Lem:shiftL} show the equality $(\Lambda_{\bfi'}, \tB_{\bfi'}) = \sigma_+\mu_+(\Lambda_\bfi, \tB_\bfi)$ of compatible pairs. 
Therefore we have obtained the following.
\begin{Prop}
Assume that $\bfi, \bfi' \in \Delta_0^{(\infty)}$ are related by $\bfi' = \partial_+ \bfi$.
Then, we have a homomorphism of $\Z[t^{\pm 1/2}]$-algebras
\[ \partial_+^* (= \mu_+^* \sigma_+^*) \colon \cA_{\bfi'} \to \cA_\bfi \qquad \text{given by $X_u \mapsto \mu_+(X_{\sigma_+^{-1}(u)})$ for all $u \in \N$},\]
which sends each cluster monomial in $\cA_{\bfi'}$ to a cluster monomial in $\cA_\bfi$.
\end{Prop}

Now, let us explain how the degrees of cluster monomials get modified under a certain condition.
For $\bfi \in \Delta_0^{(\infty)}$, let $C_\bfi \subset \Z^{\oplus \N}$ denote a cone given by 
\begin{equation} \label{eq:cone}
C_\bfi \seq \left\{\bfg = (g_u)_{u \in \N} \in \Z^{\oplus \N} \; \middle | \; \sum_{v \ge u, i_v = i_u} g_v \ge 0, \forall u \in \N   \right\} 
= \sum_{u \in \N} \N_0 (\bfe_u - \bfe_{u^-_{\bfi}}),
\end{equation}
where $\{\bfe_u\}_{u \in \N}$ is the natural basis of $\Z^{\oplus \N}$ and we understand $\bfe_0 = 0$.

\begin{Lem} \label{Lem:dg}
Assume that $\bfi, \bfi' \in \Delta_0^{(\infty)}$ are related by $\bfi' = \partial_+ \bfi$. 
Let $x' \in \cA_{\bfi'}$ be a cluster monomial of degree $\bfg' = (g'_u)_{u \in \N}$. Assuming $\bfg' \in C_{\bfi'}$, the degree $\bfg = (g_u)_{u \in \N}$ of the cluster monomial $\partial_+^* x' \in \cA_\bfi$ is given by
\begin{equation} \label{eq:dg}
g_u = \begin{cases}
-\sum_{v \in \N; i'_v = i_1} g'_v & \text{if $u=1$}, \\
g'_{u-1} & \text{if $u > 1$}. \\
\end{cases} 
\end{equation}
In other words, the assignment $\bfg' \mapsto \bfg$ is given by the $\N_0$-linear map $C_{\bfi'} \to C_\bfi$ which sends $\bfe_u - \bfe_{u^-_{\bfi'}} \in C_{\bfi'}$ to $%\bfe_{u+1} - \bfe_{u^-_{\bfi'}+1} = 
\bfe_{u+1} - \bfe_{(u+1)^-_{\bfi}} \in C_{\bfi}$ for all $u \in \N$.
\end{Lem}
\begin{proof}
This is an application of Theorem~\ref{Thm:trop}.
To be more precise, let $\bfg'' = (g''_u)_{u \in \N} \seq (g'_{\sigma_+(u)})_{u \in \N}$ so that we have $\deg\sigma_+^*(x') = \bfg''$ (see Remark~\ref{Rem:permg}).
Take a large integer $n \in \N$ such that $g''_u =0$ for all $u > x_n$.
Then we have $\partial_+^* x' = \mu_{x_1}^*\mu_{x_2}^* \cdots \mu_{x_n}^*\sigma_+^* x'$.
The assumption $\bfg' \in C_{\bfi'}$ implies 
\[
s_{m} \seq \sum_{m \le l \le n} g''_{x_l} \ge 0 
\]
for all $m \in [2,n]$.
Since $g''_{x_n} = s_n \ge 0$, we can apply Remark~\ref{Rem:trop} to get
\[\deg(\mu_{x_n}^* \sigma_+^*x') = E^{(n)} \bfg'', \]
where $E^{(n)}$ is the matrix as in the proof of Lemma~\ref{Lem:shiftL}.
Again, since $(E^{(n)}\bfg'')_{x_{n-1}} = s_{n-1} \ge 0$, we can apply Remark~\ref{Rem:trop} to find 
\[\deg(\mu_{x_{n-1}}^*\mu_{x_n}^* \sigma_+^*x') = E^{(n-1)}E^{(n)} \bfg''. \]
Since $(E^{(m+1)} \cdots E^{(n)}\bfg'')_{x_{m}} = s_m \ge 0$ for any $m \in [2,n-1]$, we can successively apply the similar argument, finally arriving at
\[ \bfg = \deg(\mu_{x_1}^*\mu_{x_2}^* \cdots \mu_{x_n}^*\sigma_+^* x') = E^{(1)}E^{(2)} \cdots E^{(n)}\bfg'' = E\bfg'',\]
where $E = (e_{u,v})_{u,v \in \N}$ is the matrix given by \eqref{eq:edelta}. 
The last equation is equivalent to the desired equation \eqref{eq:dg}.
\end{proof}

We conclude this section by showing relevant technical lemmas. Lemma \ref{Lem:cone} will be used in the proof of Theorem~\ref{Thm:Psi=tau}, and Lemma \ref{Lem:p-inv} will be needed in the proof of Lemma~\ref{Lem:compatD}. 

\begin{Lem} \label{Lem:cone}
Assume that $\bfi, \bfi' \in \Delta_0^{(\infty)}$ are related by $\bfi' = \gamma_k \bfi$ $($resp.~$\bfi' = \beta_k\bfi)$ for some $k \in \N$.
Then, the assignment $\bfg' \mapsto \bfg$ given by the equation \eqref{eq:cg} $($resp.~\eqref{eq:bg}$)$ sends the cone $C_{\bfi'} \subset \Z^{\oplus \N}$ into the cone $C_{\bfi} \subset \Z^{\oplus \N}$.
\end{Lem}
\begin{proof}
To simplify the notation, for $\bfg = (g_u)_{u \in \N}$ and $\bfg' = (g'_u)_{u \in \N}$, we set
\[ 
\Sigma_u \seq \sum_{v \ge u, i_v = i_u} g_v, \quad \Sigma'_u \seq \sum_{v \ge u, i'_v = i'_u} g'_v
\] 
for each $u \in \N$.
We have to show that $\Sigma'_u \ge 0$ $(\forall u \in \N)$ implies $\Sigma_u \ge 0$ $(\forall u \in \N)$ under the assumption.
When $\bfi' = \gamma_k \bfi$, the assertion is trivial.
So we only consider the case when $\bfi' = \beta_k\bfi$ and $\bfg$ is obtained from $\bfg'$ by \eqref{eq:bg}.
Assume $\Sigma'_u \ge 0$ for all $u \in \N$.
\begin{itemize}
\item If $u > k+2$, we have $\Sigma_u = \Sigma'_u \ge 0$.
\item If $u=k+2$, we have $\Sigma_{k+2} = \Sigma'_{k+1} + \max(g'_k,0) \ge 0$. 
\item If $u=k+1$, we have $\Sigma_{k+1} = \delta(g'_k \ge 0) \Sigma'_{k+2} + \delta(g'_k < 0) \Sigma'_k \ge 0$.
\item If $u =k$, we have
\begin{align*} 
\Sigma_k &= \Sigma_{k+2} + g_k \\
&= \Sigma'_{k+1} + \max(g'_k,0) - g'_k \\
&= \Sigma'_{k+1} + \max(-g'_k,0) \ge 0. 
\end{align*}
\item If $u < k$, we have $\Sigma_u = \Sigma'_u \ge 0$ since we see 
\begin{align*}
g_{k+2} + g_{k} + g_{k^-_\bfi} &= g'_{k+1} + \max(g'_k,0) - g'_k + g'_{(k+1)^-_{\bfi'}} + \min(g'_k,0) \\
&= g'_{k+1} + g'_{(k+1)^-_{\bfi'}}  
\end{align*} 
and similarly $g_{k+1} + g_{(k+1)^-_\bfi} = g'_{k+2} + g'_{k} + g'_{k^-_{\bfi'}}$.
\end{itemize}
Thus, we obtain the assertion.
\end{proof}
For $\bfg = (g_u)_{u \in \N}$ and $\bfi \in \Delta_0^{(\infty)}$, write
\begin{equation}
\sfp_{\bfi}(\bfg; i)\seq \sum_{u\in \N, i_u = i} g_u.\label{eq:p-sum}
\end{equation}
The proof of Lemma \ref{Lem:cone} shows the following invariant property of $\sfp_{\bfi}(\bfg; i)$.  
\begin{Lem} \label{Lem:p-inv}
Let $\bfi, \bfi' \in \Delta_0^{(\infty)}$. 
If $\bfi' = \gamma_k \bfi$ for some $k \in \N$, the elements $\bfg, \bfg'\in \Z^{\oplus\N}$ related by the equation \eqref{eq:cg} satisfy
\[
\sfp_{\bfi}(\bfg; i)=\sfp_{\bfi'}(\bfg'; i)
\]
for all $i\in \Delta_0$. 

If $\bfi' = \beta_k\bfi$ for some $k \in \N$, the elements $\bfg, \bfg'\in \Z^{\oplus\N}$ related by the equation \eqref{eq:bg}  satisfy
\[
\sfp_{\bfi}(\bfg; i)=
\begin{cases}
    \sfp_{\bfi'}(\bfg'; i)+\max(-g'_k,0)&\text{if $i=i_k$ and $k^-_\bfi=0$},\\
    \sfp_{\bfi'}(\bfg'; i)-\max(g'_k,0)&\text{if $i=i_{k+1}$ and $(k+1)^-_\bfi=0$},\\
    \sfp_{\bfi'}(\bfg'; i)&\text{otherwise}.
\end{cases}
\]
\end{Lem}

\section{Comparison with dual canonical bases}\label{secthree}

We recall that by \cite{GLS13, KKKO18, KK19} the quantum unipotent group $\cA_t[N_-]$ associated to $\sg$ has a quantum cluster algebra structure 
compatible with the dual canonical basis, which is isomorphic to some of the quantum cluster algebras $\cA_\bfi$. We establish that these isomorphisms are compatible 
with the transformations in the last section: for a composition $\cA_\bfi\simeq \cA_{\bfi'}$ of such transformations, the corresponding transformation induced on $\cA_t[N_-]$ is just the identity (Corollary \ref{Cor:bftau}). Our proof is based on the analysis of the degrees of the cluster monomials.

\subsection{Dual canonical basis}
In this subsection, we briefly recall the PBW parametrization of the dual canonical basis.
For the precise definitions, we refer to \cite[Appendix A]{FHOO} and references therein. 

Let $\cA_t[N_-]$ denote the quantum coordinate ring of the unipotent group $N_- = \exp(\sn_-)$, where $\sn_-$ is the negative part of the Lie algebra $\sg$.
More precisely, we are considering the integral form defined over $\Z[t^{\pm 1/2}]$.
It carries the normalized dual canonical basis $\tbfB$.  
Each reduced word $\bfi = (i_1,\ldots,i_{\ell})$ for the longest $w_\circ$ gives a parametrization $\tbfB = \{ \tG_\bfi(\bfc) \mid \bfc \in \N_0^{\ell}\}$ characterized by
\[ \tG_\bfi(\bfc) \equiv t^{\nu_\bfi(\bfc)}\prod^{\to}_{u \in [1,\ell]} \tD(w_u \varpi_{i_u}, w_{u-1}\varpi_{i_u})^{c_u} \mod \sum_{b \in \tbfB} t\Z[t]b\]
for each $\bfc = (c_1, \ldots, c_{\ell}) \in \N_0^{\ell}$, where
\[ \nu_\bfi(\bfc) = -\frac{1}{2}\sum_{u, v \in [1,\ell]} c_u c_v(w_{u-1}\alpha_{i_u},w_{v-1}\alpha_{i_v}) + \sum_{u \in [1,\ell]} c_u^2, \]
and $\tD(\lambda, \mu)$ denotes the renormalized quantum unipotent minor.
In particular, when $\bfc = \bfe_k \seq (\delta_{u,k})_{u \in [1, \ell]}$, we have
\[ \tG_\bfi(\bfe_k) = \tD(w_k \varpi_{i_k}, w_{k^-}\varpi_{i_k}).\]
\begin{Prop}[Lusztig~{\cite[Chapter 42]{LusztigBook}}] \label{Prop:Lusztig}
Let $\bfi, \bfi' \in \Delta_0^{\ell}$ be two reduced words for $w_\circ$, and $\bfc = (c_1, \ldots, c_{\ell}), \bfc'=(c'_1, \ldots, c'_{\ell}) \in \N_0^{\ell}$.
\begin{enumerate}
\item When $\bfi' = \gamma_k\bfi$ for some $k \in [1, \ell-1]$, we have
\[ \tG_\bfi(\bfc) = \tG_{\bfi'}(\bfc') \quad \text{if and only if} \quad 
\begin{cases}
(c_k,c_{k+1}) = (c'_{k+1},c'_k), \\
c_u = c'_u \text{ if $u \not \in \{ k, k+1\}$}.
\end{cases}
\]
\item When $\bfi' = \beta_k\bfi$ for some $k \in [1,\ell-2]$, we have
\[ \tG_\bfi(\bfc) = \tG_{\bfi'}(\bfc') \quad \text{if and only if}  \quad 
\begin{cases}
c'_k = c_{k+1} + c_{k+2} - \min(c_k, c_{k+2}), \\
c'_{k+1} = \min(c_k, c_{k+2}), \\
c'_{k+2} = c_{k+1} + c_{k} - \min(c_k, c_{k+2}), \\
c'_u = c_u \text{ if $u \not \in \{ k, k+1, k+2\}$}.
\end{cases}
\]
\end{enumerate}
\end{Prop}

\subsection{Cluster structures on $\cA_t[N_-]$}
Recall that for any finite sequence $\bfi$ in $\Delta_0$, we have associated to $\bfi$ a compatible pair $(\tB_\bfi, \Lambda_\bfi)$ in Definition~\ref{Def:BL}.

\begin{Thm}[Gei\ss--Leclerc--Schr\"oer~\cite{GLS13}, Kang--Kashiwara--Kim--Oh~\cite{KKKO18}, Kashiwara--Kim~\cite{KK19}] \label{Thm:CAN}
For each reduced word $\bfi \in \Delta_0^{\ell}$ for the longest element $w_\circ$, there is an isomorphism of $\Z[t^{\pm 1/2}]$-algebras
\[ \varphi_\bfi \colon \cA_t(\tB_\bfi, \Lambda_\bfi) \simeq \cA_t[N_-] \]
which satisfies the following: 
\begin{enumerate}
\item for any $1 \le v < u \le \ell$ with $i_u = i_v$, there exists a cluster variable which corresponds to $\tD(w_u \varpi_{i_u}, w_v\varpi_{i_v})$ under $\varphi_\bfi$,
\item every cluster monomial corresponds to an element of the basis $\tbfB$ under $\varphi_\bfi$,
\item for each $\bfc = (c_u) \in \N_0^{\ell}$, the element $\varphi_{\bfi}^{-1}\tG_\bfi(\bfc)$ is pointed and we have 
\begin{equation} \label{eq:degG}
\deg \varphi_{\bfi}^{-1}\tG_\bfi(\bfc) = \sum_{u \in [1,\ell]}c_u(\bfe_u - \bfe_{u^-})  \in \Z^{\ell},
\end{equation} 
where we understand $\bfe_0 = 0$.
In particular, the map $\deg \circ \varphi_\bfi^{-1} \colon \tbfB \to \Z^{\ell}$ is injective.  
\end{enumerate}
\end{Thm}

\subsection{Change of reduced words}

We establish the compatibility between the transformations $\gamma_k$, $\beta_k$ in the last section and the isomorphisms $\varphi_{\bfi}$.

\begin{Prop}
\label{Prop:tau}
Let $\bfi, \bfi' \in \Delta_0^{\ell}$ be two reduced words for $w_\circ$.
Assume that we have $\bfi' = \tau \bfi$ for some $\tau \in \{ \gamma_1, \ldots, \gamma_{\ell-1}\} \cup\{ \beta_1, \ldots, \beta_{\ell-2}\}$. 
Then the following diagram commutes:
\[ \xymatrix{
\cA_t(\tB_{\bfi'}, \Lambda_{\bfi'}) \ar[r]^-{\varphi_{\bfi'}} \ar[d]_-{\tau^*} & \cA_t[N_-] \ar@{=}[d]\\
\cA_t(\tB_\bfi, \Lambda_\bfi) \ar[r]^-{\varphi_\bfi} & \cA_t[N_-].
}
\]
\end{Prop}
\begin{proof}
Since $\cA_t[N_-]$ is generated by $\{\tG_{\bfi'}(\bfe_u)\}_{u \in [1,\ell]}$,
it suffices to show the equality   
\begin{equation} \label{eq:tau}
\varphi_\bfi \tau^* \varphi_{\bfi'}^{-1}\tG_{\bfi'}(\bfe_u) = \tG_{\bfi'}(\bfe_u) 
\end{equation}
for all $u \in [1,\ell]$.
Since $\tau^*$ respects cluster monomials, 
we know that the LHS of \eqref{eq:tau} belongs to $\tbfB$ by Theorem~\ref{Thm:CAN}.
When $\tau = \gamma_k$ for some $k \in [1, \ell -1]$, we have
\[ \deg \gamma_k^* \varphi_{\bfi'}^{-1}\tG_{\bfi'}(\bfe_u) = \bfe_{\sigma_k(u)} - \bfe_{\sigma_k(u^-_{\bfi'})} = \bfe_{\sigma_k(u)} - \bfe_{\sigma_k(u)^-_{\bfi}} 
= \deg \varphi_\bfi^{-1} \tG_\bfi(\bfe_{\sigma_k(u)}) \]
by Lemma~\ref{Lem:cg} and Theorem~\ref{Thm:CAN}.
Since $\deg \circ \varphi_\bfi^{-1}$ is injective on $\tbfB$, we get $\varphi_\bfi \gamma_k^* \varphi_{\bfi'}^{-1}\tG_{\bfi'}(\bfe_u) = \tG_{\bfi}(\bfe_{\sigma_k(u)})$. 
On the other hand, we know $\tG_\bfi(\bfe_{\sigma_k(u)}) = \tG_{\bfi'}(\bfe_u)$ by Proposition~\ref{Prop:Lusztig}.
Therefore, we obtain \eqref{eq:tau} for $\tau = \gamma_k$.
Next, we consider the case when $\tau = \beta_k$ for some $k \in [1, \ell -2]$.
By Lemma~\ref{Lem:bg} and Theorem~\ref{Thm:CAN}, we have
\[ \deg \beta_k^* \varphi_{\bfi'}^{-1}\tG_{\bfi'}(\bfe_u) = 
\begin{cases}
\bfe_{k+2} - \bfe_{k} & \text{if $u=k$}, \\
\bfe_{k+2} - \bfe_{k^-_\bfi} & \text{if $u=k+1$}, \\
\bfe_{k} - \bfe_{k^-_{\bfi}} & \text{if $u=k+2$}, \\
\bfe_{u} - \bfe_{u^-_\bfi} & \text{otherwise}.
\end{cases}
\]
Again by the injectivity of $\deg \circ \varphi_\bfi^{-1}$ on $\tbfB$, we get $\varphi_\bfi \tau^* \varphi_{\bfi'}^{-1}\tG_{\bfi'}(\bfe_u) = \tG_{\bfi}(\bfc_u)$, where
\[ \bfc_u \seq \begin{cases}
\bfe_{k+2} & \text{if $u=k$}, \\
\bfe_k + \bfe_{k+2} & \text{if $u=k+1$}, \\
\bfe_{k} & \text{if $u=k+2$}, \\
\bfe_u & \text{otherwise}.
\end{cases}
\]
On the other hand, we have $\tG_\bfi(\bfc_u) = \tG_{\bfi'}(\bfe_u)$ by Proposition~\ref{Prop:Lusztig}.
Thus, we obtain \eqref{eq:tau} for $\tau = \beta_k$. 
\end{proof}

For any two reduced words $\bfi, \bfi' \in \Delta_0^{\ell}$ for $w_\circ$, 
we can always find a finite sequence 
$\boldsymbol{\tau} = (\tau_1, \ldots, \tau_l)$ in 
$\{ \gamma_1, \ldots, \gamma_{\ell-1}\} \cup \{\beta_1, \ldots, \beta_{\ell-2}\}$ 
such that $\bfi' = \tau_1 \cdots \tau_l \bfi$.
Then we have the composed isomorphism
\[ \boldsymbol{\tau}^* \seq \tau_l^* \circ \cdots \circ \tau_1^* 
\colon \cA_t(\tB_{\bfi'}, \Lambda_{\bfi'}) \simeq \cA_t(\tB_{\bfi}, \Lambda_\bfi).\]

\begin{Cor} \label{Cor:bftau}
With the above notation, the following diagram commutes:
\[ \xymatrix{
\cA_t(\tB_{\bfi'}, \Lambda_{\bfi'}) \ar[r]^-{\varphi_{\bfi'}} \ar[d]_-{\boldsymbol{\tau}^*} & \cA_t[N_-] \ar@{=}[d]\\
\cA_t(\tB_\bfi, \Lambda_\bfi) \ar[r]^-{\varphi_\bfi} & \cA_t[N_-].
}
\]
In particular, the isomorphism $\boldsymbol{\tau}^* = \varphi_{\bfi} \circ \varphi_{\bfi'}^{-1}$ depends only on the pair $(\bfi, \bfi')$, and not on the sequence $\boldsymbol{\tau}$ satisfying $\bfi' = \tau_1 \cdots \tau_l \bfi$.
\end{Cor}

\begin{Rem}
By definition, the isomorphism $\boldsymbol{\tau}^*$ induces a bijection 
between cluster monomials in $\cA_t(\tB_{\bfi'}, \Lambda_{\bfi'})$ and 
those in $\cA_t(\tB_\bfi, \Lambda_\bfi)$. 
\end{Rem}

\section{Reminders on quantum Grothendieck rings}\label{secfour}

We give general reminders on quantum Grothendieck rings of the category $\Cc$ of finite-dimensional representations of a quantum loop algebra. 
In particular, we recall quantum $T$-systems, Q-data, and Kazhdan--Lusztig type conjectures in this context. We also recall the monoidal subcategories 
$\Cc_\Z$ and $\Cc^-$ of $\Cc$ (as defined in \cite{HL10, HL16}) and we introduce subcategories $\Cc_{\le \xi}$ generalizing $\Cc^-$.

\subsection{Quantum loop algebras and the category $\Cc_\Z$}
\label{Ssec:rep}

Let $\fg$ be a complex finite-dimensional simple Lie algebra (it should not be confused with the Lie algebra $\sg$ of the previous sections).
Let $C = (c_{ij})_{i,j \in I}$ denote the Cartan matrix of $\fg$, where $I$ is the set of Dynkin indices. 
For $i,j \in I$, we write $i \sim j$ if $c_{ij} < 0$. 
Let $r \in \{1,2,3\}$ be the lacing number of $\fg$, and $d \colon I \to \{1,r\}$ the function satisfying $d_i c_{ij} = d_j c_{ji}$ for all $i, j \in I$, i.e., the minimal left symmetrizer of $C$.

Let $U_q(L\fg)$ be the quantum loop algebra associated to $\fg$.
It is a Hopf algebra defined over an algebraic closed field $\kk = \ol{\Q(q)}$, where $q$ is a formal parameter. 
Let $\Cc$ denote the rigid monoidal category of finite-dimensional $U_q(L\fg)$-modules,  with the standard type $1$ condition.  
It is endowed with the contravariant auto-equivalence $\fD^{\pm 1}$ which sends each module to its left/right dual. 
Recall that the isomorphism classes of simple modules of the category $\Cc$ are parametrized by the set $(1+z\kk[z])^I$ of $I$-tuples of monic polynomials (called Drinfeld polynomials) \cite{CP,CP95}.  

In this paper, we restrict ourselves to a nice monoidal subcategory $\Cc_\Z$ of $\Cc$ introduced by Hernandez--Leclerc~\cite{HL10}  as the Serre subcategory generated by a distinguished family of simple modules. 
Precisely, we fix a function $\epsilon \colon I \to \{0,1\}$ satisfying the condition 
\begin{equation} \label{eq:ep}
\epsilon_i \equiv \epsilon_j + \min(d_i, d_j) \pmod 2 \quad \text{whenever $i\sim j$}, 
\end{equation} 
which we call a \emph{parity function}, and let
\[ \hI \seq \{(i,p) \in I \times \Z \mid p \equiv \epsilon_i \pmod 2 \}.\]
 We introduce a formal variable $Y_{i,p}$ for each $(i,p) \in \hI$, and consider the ring of Laurent polynomials $\cY \seq \Z[Y_{i,p}^{\pm 1} \mid (i,p) \in \hI \, ]$.
Let $\cM^* \subset \cY$ be the set of all the Laurent monomials.  
An element $m \in \cM^*$ is written as
\begin{equation} \label{eq:monomial}
m = \prod_{(i,p) \in \hI} Y_{i,p}^{u_{i,p}(m)}.
\end{equation}
We say that $m \in \cM^*$ is \emph{dominant} if $u_{i,p}(m) \ge 0$ for all $(i,p) \in \hI$.
Let $\cM \subset \cM^*$ be the set of dominant monomials.
For each $m \in \cM$, we have a simple module $L(m) \in \Cc$ corresponding to the Drinfeld polynomials $(\prod_{p}(1-q^pz)^{u_{i,p}(m)})_{i \in I}$.
The category $\Cc_\Z$ is defined to be the Serre subcategory of $\Cc$ generated by the simple modules $\{ L(m) \mid m \in \cM \}$.
It is closed under taking tensor products and the auto-equivalences $\fD^{\pm 1}$, and hence forms a rigid monoidal category in itself.
Indeed, we have $\fD^{\pm 1}L(m) \simeq L(\fD^{\pm 1}m)$ for any $m \in \cM$. 
Here on the right hand side $\fD^{\pm1}$ denotes the automorphism of $\cY$ given by $Y_{i,p} \mapsto Y_{i^*,p\pm rh^\vee}$ for $(i,p) \in \hI$, where $i \mapsto i^*$ is the involution of $I$ induced by the longest Weyl group element and $h^\vee$ is the dual Coxeter number of $\fg$.
Moreover, every prime simple module of $\Cc$ (that is a simple module which can not be factorized into a non-trivial tensor product of modules) is in $\Cc_\Z$ after a suitable spectral parameter shift. 

The $q$-character homomorphism $\chi_q$ defined by Frenkel--Reshetikhin~\cite{FR99} gives an injective ring homomorphism $\chi_q \colon K(\Cc_\Z) \to \cY$.  
As a ring, the Grothendieck ring $K(\Cc_\Z)$ is isomorphic to the ring of polynomials in $\{ [L(Y_{i,p})]\}_{(i,p) \in \hI}$.  
A simple module of the form $L(Y_{i,p})$ is called a \emph{fundamental module}.

For each $(i,p) \in I \times \Z$ with $(i, p-d_i) \in \hI$, we define the element $A_{i,p} \in \cM^*$ by
\[ A_{i,p} =  Y_{i, p-d_i} Y_{i,p+d_i} \prod_{(j,s) \in \hI \colon j \sim i, |s-p| < d_i} Y_{j,s}^{-1}, \]
which is a loop analog of the $i$-th simple root \cite{FR99}.
For $m,m' \in \cM^*$, we write $m \le m'$ if $m'm^{-1}$ is a monomial in various $A_{i,p}$ for $(i,p-d_i)\in \hI$.
This defines a partial ordering on $\cM^*$, called the \emph{Nakajima partial ordering}. 
For any $m \in \cM$, we have
\begin{equation} \label{eq:L}
\chi_q(L(m)) = m + \sum_{m' \in \cM^* \colon m' < m} a[m;m'] m'
\end{equation}
for some $a[m;m'] \in \N_0$. 
See \cite{Nak01} and \cite{FM01} for the proof.

\subsection{Quantum Grothendieck ring of $\Cc_\Z$}

The quantum Cartan matrix $C(q) = (C_{ij}(q))_{i,j \in I}$ is a $\Z[q^{\pm 1}]$-valued matrix given by
\[C_{ij}(q) \seq \delta_{i,j} (q^{d_i}+q^{-d_i}) + (1-\delta_{i,j}) [c_{ij}]_q\]
for any $i,j \in I$, where $[k]_q = (q^k-q^{-k})/(q-q^{-1})$ is the standard $q$-integer.
The matrix $C(q)$ is invertible as a $\Q(q)$-valued matrix and we write $\tC(q) = (\tC_{ij}(q))_{i,j \in I}$ for its inverse.
For any $i,j \in I$, we write
\[
\tC_{ij}(q) = \sum_{u \in \Z} \tc_{ij}(u) q^{u} \in \Z(\!( q )\!)
\]
for the Laurent expansion at $q=0$ of the $(i,j)$-entry $\tC_{ij}(q)$.
In this way, we get a collection of integers $\{ \tc_{ij}(u)\mid i,j \in I, u \in \Z\}$.
We define a map $\Nn\colon (I \times \Z)^{2}  \to \Z$ by
\[
\Nn(i,p;j,s) \seq \tc_{ij}(p-s-d_{i})- \tc_{ij}(p-s+d_{i}) - \tc_{ij}(s-p-d_{i})+ \tc_{ij}(s-p+d_{i}). 
\]
It satisfies $\Nn(i,p;j,s) = - \Nn(j,s;i,p)$ for any $(i,p), (j,s) \in I \times \Z$.

Let $t$ be an indeterminate with a formal square root $t^{1/2}$.
We define the quantum torus $\cY_{t}$ to be the $\Z[t^{\pm1/2}]$-algebra presented by 
the set of generators $\{\ul{Y_{i, p}^{\pm 1}} \mid (i,p) \in \hI \ \}$ and the following relations:
\begin{itemize}
\item 
$\ul{Y_{i,p}} \cdot \ul{Y_{i,p}^{-1}}= \ul{Y_{i,p}^{-1}} \cdot \ul{Y_{i,p}}=1$ for each $(i,p) \in \hI$,
\item 
$\ul{Y_{i, p}}\cdot \ul{Y_{j, s}} = t^{\Nn(i,p;j,s)}\ul{Y_{j,s}}\cdot \ul{Y_{i,p}}$ for each $(i,p), (j,s) \in \hI$.
\end{itemize}
Note that $\cY_t$ is a deformation of $\cY$. 
Indeed, there exists a surjective $\Z$-algebra homomorphism $\evt \colon \cY_{t} \to \cY$
given by $t^{1/2} \mapsto 1$ and $\ul{Y_{i,p}} \mapsto Y_{i,p}$ for all $(i,p) \in \hI$.
An element $\tm$ of $\cY_{t}$ is called a {\em monomial} if it is a product of the generators $\ul{Y_{i,p}^{\pm 1}}$ for $(i,p) \in \hI$ and $t^{\pm 1/2}$. 
A monomial $\tm$ is said to be dominant if $\evt(\tm)$ is dominant. 
Following~\cite[\S6.3]{Her04}, we define the $\Z$-algebra anti-involution $\ol{(\cdot)}$ on $\cY_{t}$ by
\[
\ol{t^{1/2}} \seq t^{-1/2}, \qquad \ol{\ul{Y_{i,p}}} \seq \ul{Y_{i,p}}.
\]
This is called the {\em bar involution} on $\cY_{t}$. 
For any $m \in \cM^*$, we denote by $\ul{m}$ the unique monomial in $\cY_t$ satisfying $\ol{\ul{m}}= \ul{m}$ and $\evt(\ul{m}) = m$.
The elements of this form are called \emph{commutative monomials} (cf.~Appendix~\ref{ssec:Qtorus}).
%For example, we have $\ul{Y_{i,p}} = t^{1/2}\tY_{i,p}$.
Note that we have $\ul{(m^{-1})} = (\ul{m})^{-1} (=: \ul{m}^{-1})$.
The commutative monomials form a free basis 
of the $\Z[t^{\pm 1/2}]$-module $\cY_{t}$.
For any $m,m^\prime \in \cM^*$, we have
$$\ul{m\cdot m^\prime} = t^{-\Nn(m, m^\prime)/2}\ul{m} \cdot \ul{m}' = t^{\Nn(m,m^\prime)/2} \ul{m}' \cdot \ul{m}$$
where $\Nn(m,m^\prime) \in \Z$ is a skew-symmetric pairing given by
\begin{equation} \label{eq:Nnmm}
\Nn(m,m^\prime) \seq \sum_{(i,p), (j,s) \in \hI} u_{i,p}(m) u_{j,s}(m^\prime) \Nn(i,p;j,s)
\end{equation}
in the notation of \eqref{eq:monomial}.

For each $i \in I$, denote by $\cK_{i,t}$ the $\Z[t^{\pm 1/2}]$-subalgebra
of $\cY_{t}$ generated by 
\[
\left\{ \ul{Y_{i,p}} (1+t^{-1}\ul{A_{i, p+ d_{i}}^{-1}}) \; \middle| \; (i,p) \in \hI \, \right\} 
\cup \left\{ \ul{Y_{j,s}^{\pm 1}} \; \middle| \; (j,s) \in \hI, j \neq i \right\}.
\]
Following \cite{Nak04, VV03, Her04}, we define the quantum Grothendieck ring $\cK_t(\Cc_\Z)$ to be the $\Z[t^{\pm 1/2}]$-subalgebra of $\cY_t$ given by 
\[
\cK_{t}(\Cc_{\Z}) \seq \bigcap_{i \in I} \cK_{i,t}. 
\]
Note that $\cK_t(\Cc_\Z)$ is stable under the bar involution.
Moreover, we have $\evt(\cK_t(\Cc_\Z)) = \chi_q(K(\Cc_\Z))$.
For future use, we remark the following.

\begin{Lem} \label{Lem:scr}
Assume that two non-zero elements $x \in \cK_t(\Cc_\Z)$ and $y \in \cY_t$ satisfy $xy \in \cK_t(\Cc_\Z)$.
Then we have $y \in \cK_t(\Cc_\Z)$.
\end{Lem}
\begin{proof} This is a quantum analog of an argument presented in \cite{HL16}. Indeed, the assertion follows from the fact that $\cK_{i,t}$ is the kernel of the $t$-deformed $i$-th screening operator $S_{i,t}$. 
Namely, there is a $\cY_t$-bimodule $\cY_{i,t}$ and a $\Z[t^{\pm 1/2}]$-linear map $S_{i,t} \colon \cY_t \to \cY_{i,t}$ such that $S_{i,t}(xy) = S_{i,t}(x)y + x S_{i,t}(y)$ for any $x,y \in \cY_t$, and $\Ker S_{i,t} = \cK_{i,t}$. 
See \cite[Theorem 4.10]{Her04} for details.
\end{proof}

We have the following result due to the second named author which will be crucial for our purposes.

\begin{Thm}[{\cite[Theorem 5.11]{Her04}}{\cite[Theorem 7.5]{Her05}}] \label{Thm:Ft}
For every dominant monomial $m \in \cM$, there exists a unique element $F_{t}(m)$ of $\cK_{t}(\Cc_\Z)$ such that $\ul{m}$ is the unique dominant monomial occurring in $F_{t}(m)$. 
It satisfies $\ol{F_t(m)} = F_t(m)$.
Moreover, the set $\{ F_{t}(m) \mid m \in \cM \}$ forms a $\Z[t^{\pm 1/2}]$-basis of $\cK_{t}(\Cc_\Z)$, and the set $\{F_t(Y_{i,p}) \mid (i,p) \in \hI\}$ generates the $\Z[t^{\pm 1/2}]$-algebra $\cK_t(\Cc_\Z)$.
\end{Thm}

We can construct \emph{the canonical basis} $\bfL_t$ of $\cK_{t}(\Cc_\Z)$, whose member $L_t(m)$ is conjecturally a $t$-analog of the $q$-character $\chi_q(L(m))$.

\begin{Thm}[{\cite[Theorem 8.1]{Nak04}}, {\cite[Theorem 6.9]{Her04}}]
\label{Thm:qtch}
There exists a unique $\Z[t^{\pm 1/2}]$-basis $\bfL_t = \{ L_{t}(m) \mid m \in \cM\}$ of $\cK_t(\Cc_\Z)$ characterized by the following properties: for each $m \in \cM$, we have $\ol{L_{t}(m)} = L_{t}(m)$, and
\[
L_t(m) \equiv t^{\Nn(m)}\prod_{p \in \Z}^{\to}\prod_{i \in I\colon (i,p) \in \hI}F_t(Y_{i,p})^{u_{i,p}(m)} 
\mod \sum_{m' < m}t\Z[t] L_t(m')
\]
in the notation of \eqref{eq:monomial} for some $\Nn(m)\in \frac{1}{2}\Z$. Explicitly, one has
\[ \Nn(m) = -\frac{1}{2} \sum_{(i,p), (j,s) \in \hI \colon p<s}u_{i,p}(m)u_{j,s}(m)\Nn(i,p;j,s).\]
\end{Thm} 

Note that for a fixed $p$, the $F_t(Y_{i,p})$'s ($i\in I$) mutually commute, so the formula for $L_t(m)$ is well-defined.
The element $L_{t}(m)$ is called \emph{the $(q,t)$-character of $L(m)$}.
As a $t$-analog of the equation~\eqref{eq:L}, for each $m \in \cM$, we have
\begin{equation} \label{eq:Lt}
L_t(m) = \ul{m} + \sum_{m' \in \cM^* \colon m' < m} a_t[m;m'] \ul{m}'
\end{equation}
for some $a_t[m;m'] \in \Z[t^{\pm 1/2}]$.

\begin{Conj}[cf.~{\cite[Conjecture 7.3]{Her04}}] \label{Conj:KL}
For all $m \in \cM$, we have
\begin{enumerate}
\item \label{Conj:KL:KL} $\evt(L_{t}(m)) = \chi_{q}(L(m))$, and
\item \label{Conj:KL:pos} $a_t[m;m'] \in \N_0[t^{\pm 1/2}]$ for all $m' \in \cM^*$ with $m' < m$.
\end{enumerate}
\end{Conj}
Conjecture~\ref{Conj:KL} (1) is called the \emph{Kazhdan--Lusztig type conjecture} and was motivated by the results of Nakajima. At this moment, we know that Conjecture~\ref{Conj:KL} is true when $\fg$ is of type $\mathrm{ADE}$ by Nakajima \cite{Nak04}, and when $\fg$ is of type $\mathrm{B}$ by \cite{FHOO}. Moreover, Conjecture~\ref{Conj:KL} (1) for general $\fg$ holds true when $m$ belongs to $\cM_{\cQ}$ for some Q-data $\cQ$ for $\fg$ (see Section \ref{Ssec:Qdata} for the definition of $\cM_{\cQ}$) \cite{FHOO}. 

In this paper, we will prove Conjecture~\ref{Conj:KL} (1) for \emph{reachable} modules $L(m)$ (see Definition \ref{Def:reachable}) and  
Conjecture~\ref{Conj:KL} (2), for general $\fg$. See Corollaries~\ref{Cor:KLr} and \ref{Cor:pos}. 

We have the following positivity result concerning the canonical basis $\bfL_{t}$, which is proved by Varagnolo--Vasserot~\cite{VV03} for type $\mathrm{ADE}$ and by the authors \cite{FHOO} in general.

\begin{Thm}[{\cite[Theorem 4.1]{VV03}, \cite[Corollary 10.7]{FHOO}}] \label{Thm:pos}
The structure constants of $\cK_t(\Cc_\Z)$ with respect to the canonical basis $\bfL_t$ belong to $\N_0[t^{\pm 1/2}]$.  
\end{Thm}

Let $\fD_t^{\pm 1}$ be the $\Z[t^{\pm 1/2}]$-algebra automorphism of $\cY_t$ given by $\fD_t^{\pm 1}\ul{m} = \ul{\fD^{\pm 1} m}$ for all $m \in \cM^*$.
It satisfies $\evt \circ \fD^{\pm 1}_t = \fD^{\pm 1} \circ \evt$ and preserves the subalgebra $\cK_t(\Cc_{\Z}) \subset \cY_t$. 
Moreover, we have 
\[\fD_t^{\pm 1}(F_t(m)) = F_t(\fD^{\pm 1}m)\quad \text{ and } \quad \fD_{t}^{\pm 1}(L_t(m)) = L_t(\fD^{\pm 1}m)\] 
for any $m \in \cM$ (see \cite[Lemma 3.11]{FHOO}).

\subsection{Q-data and associated subcategories}
\label{Ssec:Qdata}

With our simple Lie algebra $\fg$, we associate a unique pair $(\Delta, \sigma)$, which we call the \emph{unfolding} of $\fg$, consisting of a simply-laced Dynkin diagram $\Delta$ and a graph automorphism $\sigma$ of $\Delta$ as given in the Table~\ref{table:cl}, where $\id \colon \Delta_0 \to \Delta_0$ is the identity map and
the automorphisms $\vee$ and $\tvee$ are given by the blue arrows in Figure~\ref{Fig:unf} below.  

\begin{table}[h]
\centering
 { \arraycolsep=1.6pt\def\arraystretch{1.5}
\begin{tabular}{|c|c|c|c|c|c|}
\hline
$r$ & $\fg$ & $\Delta$ (or $\sg$) & $\sigma$ & $h^{\vee}$ & $\ell $ \\
\hline
\hline
& $\mathrm{A}_{n}$ &$\mathrm{A}_{n}$ & $\id$ & $n+1$ & $n(n+1)/2$ \\
 $1$ & $\mathrm{D}_{n}$ & $\mathrm{D}_{n}$ & $\id$ &$2n-2$ & $n(n-1)$ \\
  & $\mathrm{E}_{6,7,8}$ & $\mathrm{E}_{6,7,8}$ & $\id$ & $12, 18, 30$& $36, 63, 120$ \\
\hline
& $\mathrm{B}_{n}$ & $\mathrm{A}_{2n-1}$ & $\vee$ & $2n-1$ & $n(2n-1)$\\
$2$ & $\mathrm{C}_{n}$ & $\mathrm{D}_{n+1}$ & $\vee$ &  $n+1$ & $n(n+1)$ \\
 & $\mathrm{F}_{4}$ & $\mathrm{E}_{6}$ & $\vee$ & $9$ & $36$ \\
\hline
$3$ & $\mathrm{G}_{2}$ & $\mathrm{D}_{4}$ & $\widetilde{\vee}$ &  $4$ & $12$ \\
\hline
\end{tabular}
  }\\[1.5ex]
\caption{Unfoldings and associated numerical data} \label{table:cl} 
\end{table}

\begin{figure}[ht]
\begin{center}
\begin{tikzpicture}[xscale=1.25,yscale=.7]
\node (A2n1) at (-0.2,4.5) {$(\mathrm{A}_{2n-1}, \vee)$};
\node[dynkdot,label={below:\footnotesize$n+1$}] (A6) at (4,4) {};
\node[dynkdot,label={below:\footnotesize$n+2$}] (A7) at (3,4) {};
\node[dynkdot,label={below:\footnotesize$2n-2$}] (A8) at (2,4) {};
\node[dynkdot,label={below:\footnotesize$2n-1$}] (A9) at (1,4) {};
\node[dynkdot,label={above:\footnotesize$n-1$}] (A4) at (4,5) {};
\node[dynkdot,label={above:\footnotesize$n-2$}] (A3) at (3,5) {};
\node (Au) at (2.5, 5) {$\cdots$};
\node (Al) at (2.5, 4) {$\cdots$};
\node[dynkdot,label={above:\footnotesize$2$}] (A2) at (2,5) {};
\node[dynkdot,label={above:\footnotesize$1$}] (A1) at (1,5) {};
\node[dynkdot,label={above:\footnotesize$n$}] (A5) at (5,4.5) {};
\path[-]
 (A1) edge (A2)
 (A3) edge (A4)
 (A4) edge (A5)
 (A5) edge (A6)
 (A6) edge (A7)
 (A8) edge (A9);
\path[-] (A2) edge (Au) (Au) edge (A3) (A7) edge (Al) (Al) edge (A8);
\path[<->,thick,blue] (A1) edge (A9) (A2) edge (A8) (A3) edge (A7) (A4) edge (A6);
\path[->, thick, blue] (A5) edge [loop below] (A5);
\def\Foffset{6.5}
\node (Bn) at (-0.2,\Foffset) {$\mathrm{B}_n$};
\foreach \x in {1,2}
{\node[dynkdot,label={above:\footnotesize$\x$}] (B\x) at (\x,\Foffset) {};}
\node[dynkdot,label={above:\footnotesize$n-2$}] (B3) at (3,\Foffset) {};
\node[dynkdot,label={above:\footnotesize$n-1$}] (B4) at (4,\Foffset) {};
\node[dynkdot,label={above:\footnotesize$n$}] (B5) at (5,\Foffset) {};
\node (Bm) at (2.5,\Foffset) {$\cdots$};
\path[-] (B1) edge (B2) (B2) edge (Bm) (Bm) edge (B3) (B3) edge (B4);
\draw[-] (B4.30) -- (B5.150);
\draw[-] (B4.330) -- (B5.210);
\draw[-] (4.55,\Foffset) -- (4.45,\Foffset+.2);
\draw[-] (4.55,\Foffset) -- (4.45,\Foffset-.2);
\draw[-,dotted] (A1) -- (B1);
\draw[-,dotted] (A2) -- (B2);
\draw[-,dotted] (A3) -- (B3);
\draw[-,dotted] (A4) -- (B4);
\draw[-,dotted] (A5) -- (B5);
\draw[|->] (Bn) -- (A2n1);
\node (Dn1) at (-0.2,0) {$(\mathrm{D}_{n+1}, \vee)$};
\node[dynkdot,label={above:\footnotesize$1$}] (D1) at (1,0){};
\node[dynkdot,label={above:\footnotesize$2$}] (D2) at (2,0) {};
\node (Dm) at (2.5,0) {$\cdots$};
\node[dynkdot,label={above:\footnotesize$n-2$}] (D3) at (3,0) {};
\node[dynkdot,label={above:\footnotesize$n-1$}] (D4) at (4,0) {};
\node[dynkdot,label={above:\footnotesize$n$}] (D6) at (5,.5) {};
\node[dynkdot,label={below:\footnotesize$n+1$}] (D5) at (5,-.5) {};
\path[-] (D1) edge (D2)
  (D2) edge (Dm)
  (Dm) edge (D3)
  (D3) edge (D4)
  (D4) edge (D5)
  (D4) edge (D6);
\path[<->,thick,blue] (D6) edge (D5);
\path[->,thick,blue] (D1) edge [loop below] (D1)
(D2) edge [loop below] (D2)
(D3) edge [loop below] (D3)
(D4) edge [loop below] (D4);
\def\Coffset{1.8}
\node (Cn) at (-0.2,\Coffset) {$\mathrm{C}_n$};
\foreach \x in {1,2}
{\node[dynkdot,label={above:\footnotesize$\x$}] (C\x) at (\x,\Coffset) {};}
\node (Cm) at (2.5, \Coffset) {$\cdots$};
\node[dynkdot,label={above:\footnotesize$n-2$}] (C3) at (3,\Coffset) {};
\node[dynkdot,label={above:\footnotesize$n-1$}] (C4) at (4,\Coffset) {};
\node[dynkdot,label={above:\footnotesize$n$}] (C5) at (5,\Coffset) {};
\draw[-] (C1) -- (C2);
\draw[-] (C2) -- (Cm);
\draw[-] (Cm) -- (C3);
\draw[-] (C3) -- (C4);
\draw[-] (C4.30) -- (C5.150);
\draw[-] (C4.330) -- (C5.210);
\draw[-] (4.55,\Coffset+.2) -- (4.45,\Coffset) -- (4.55,\Coffset-.2);
\draw[-,dotted] (C1) -- (D1);
\draw[-,dotted] (C2) -- (D2);
\draw[-,dotted] (C3) -- (D3);
\draw[-,dotted] (C4) -- (D4);
\draw[-,dotted] (C5) -- (D6);
\draw[|->] (Cn) -- (Dn1);
\node (E6desc) at (6.8,4.5) {$(\mathrm{E}_6, \vee)$};
\node[dynkdot,label={above:\footnotesize$2$}] (E2) at (10.8,4.5) {};
\node[dynkdot,label={above:\footnotesize$4$}] (E4) at (9.8,4.5) {};
\node[dynkdot,label={above:\footnotesize$5$}] (E5) at (8.8,5) {};
\node[dynkdot,label={above:\footnotesize$6$}] (E6) at (7.8,5) {};
\node[dynkdot,label={below:\footnotesize$3$}] (E3) at (8.8,4) {};
\node[dynkdot,label={below:\footnotesize$1$}] (E1) at (7.8,4) {};
\path[-]
 (E2) edge (E4)
 (E4) edge (E5)
 (E4) edge (E3)
 (E5) edge (E6)
 (E3) edge (E1);
\path[<->,thick,blue] (E3) edge (E5) (E1) edge (E6);
\path[->, thick, blue] (E4) edge[loop below] (E4); 
\path[->, thick, blue] (E2) edge[loop below] (E2); 
\def\Foffset{6.5}
\node (F4desc) at (6.8,\Foffset) {$\mathrm{F}_4$};
\foreach \x in {1,2,3,4}
{\node[dynkdot,label={above:\footnotesize$\x$}] (F\x) at (\x+6.8,\Foffset) {};}
\draw[-] (F1.east) -- (F2.west);
\draw[-] (F3) -- (F4);
\draw[-] (F2.30) -- (F3.150);
\draw[-] (F2.330) -- (F3.210);
\draw[-] (9.35,\Foffset) -- (9.25,\Foffset+.2);
\draw[-] (9.35,\Foffset) -- (9.25,\Foffset-.2);
\draw[|->] (F4desc) -- (E6desc);
\path[-, dotted] (F1) edge (E6)
(F2) edge (E5) (F3) edge (E4) (F4) edge (E2);

\node (D4desc) at (6.8,0) {$(\mathrm{D}_{4}, \tvee)$};
\node[dynkdot,label={above:\footnotesize$1$}] (D1) at (7.8,.6){};
\node[dynkdot,label={above:\footnotesize$2$}] (D2) at (8.8,0) {};
\node[dynkdot,label={left:\footnotesize$3$}] (D3) at (7.8,0) {};
\node[dynkdot,label={below:\footnotesize$4$}] (D4) at (7.8,-.6) {};
\draw[-] (D1) -- (D2);
\draw[-] (D3) -- (D2);
\draw[-] (D4) -- (D2);
\path[->,blue,thick]
(D1) edge [bend left=0] (D3)
(D3) edge [bend left=0](D4)
(D4) edge[bend left=90] (D1);
\path[->, thick, blue] (D2) edge[loop below] (D2); 
\def\Goffset{1.8}
\node (G2desc) at (6.8,\Goffset) {$\mathrm{G}_2$};
\node[dynkdot,label={above:\footnotesize$1$}] (G1) at (7.8,\Goffset){};
\node[dynkdot,label={above:\footnotesize$2$}] (G2) at (8.8,\Goffset) {};
\draw[-] (G1) -- (G2);
\draw[-] (G1.40) -- (G2.140);
\draw[-] (G1.320) -- (G2.220);
\draw[-] (8.25,\Goffset+.2) -- (8.35,\Goffset) -- (8.25,\Goffset-.2);
\draw[|->] (G2desc) -- (D4desc);
\path[-, dotted] (D1) edge (G1) (D2) edge (G2);
\end{tikzpicture}
\end{center}
\caption{Unfoldings for non-simply-laced $\fg$} \label{Fig:unf}
\end{figure}
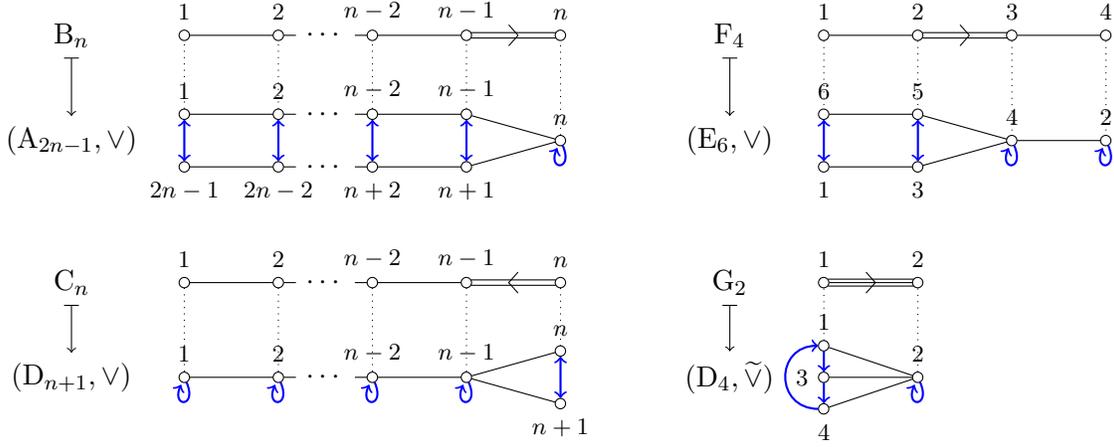

Hereafter, we denote by $\sg$ the simple Lie algebra associated with the simply-laced Dynkin diagram $\Delta$ and apply the notation in \S\ref{Ssec:notation} for this $\sg$ except that we adopt the symbols $\im, \jm, \ldots$ to denote elements of $\Delta_0$ in order to reserve the symbols $i,j,\ldots$ for elements of $I$.
Under the above assignment $\fg \mapsto (\Delta, \sigma)$, we identify the lacing number $r$ of $\fg$ with the order of the automorphism $\sigma$, 
and also identify the set $I$ of Dynkin indices of $\fg$ with the set $\Delta_{0}/\langle \sigma \rangle$ of $\sigma$-orbits in $\Delta_0$ as suggested by the dotted lines in Figure~\ref{Fig:unf}.   
Then the positive integer $d_i \in \{ 1,r \}$ coincides with the cardinality of the $\sigma$-orbit  $i \in I$.     
We denote the natural quotient map $\Delta_0 \to I = \Delta_{0}/\langle \sigma \rangle$
by $\im \mapsto \bar{\im}$. 
%By definition, for $\im \in \Delta_0$ and $i \in I$, we have $i = \bar{\im}$ if and only if $\im \in i$.
%Note that the involution $*$ on $\Delta_{0}$ induces an involution
%on the quotient set $I$, which coincides with the involution $i \mapsto i^*$ that appeared
%in the description of $\fD$ in \S\ref{Ssec:rep}. 

Let us recall the following notion introduced in \cite{FO21} and which will be important in the following. A \emph{Q-datum} for $\fg$ is defined to be a triple $\cQ=(\Delta, \sigma, \xi)$ such that $(\Delta, \sigma)$ is the unfolding of $\fg$ and $\xi \colon \Delta_0 \to \Z$ is a function satisfying the following properties:
\begin{enumerate}
\item For $\im, \jm \in \Delta_{0}$ with $\im \sim \jm$ and $d_{\bar{\im}}=d_{\bar{\jm}}$, we have $|\xi_{\im} - \xi_{\jm}| = d_{\bar{\im}}=d_{\bar{\jm}}$.
\item For $i, j \in I$ with $i \sim j$ and $d_{i} = 1 < d_{j}=r$, there is a unique $\jm \in j$ such that $|\xi_{\im} - \xi_{\jm}| = 1$ and
$\xi_{\sigma^{k}(\jm)} = \xi_{\jm} - 2k$ for any $1 \le k < r$, where $i=\{\im\}$. 
\end{enumerate} 
We refer to the function $\xi$ as a \emph{height function}.  
Recall that we have chosen a parity function $\epsilon \colon I \to \{0,1\}$ satisfying \eqref{eq:ep} in \S\ref{Ssec:rep}. 
In what follows, we always assume without loss of generality that a Q-datum $\cQ=(\Delta, \sigma, \xi)$ satisfies the condition
\begin{equation} \label{eq:xiep}
\xi_\im \equiv \epsilon_{\bar{\im}} \pmod 2 \quad \text{for any $\im \in \Delta_0$}. 
\end{equation}

For a Q-datum $\cQ = (\Delta, \sigma, \xi)$ for $\fg$, we define another Q-datum $\fD^{\pm 1} \cQ = (\Delta, \sigma, \fD^{\pm 1}\xi)$ for $\fg$ by
\begin{equation} \label{eq:DQ} 
(\fD^{\pm 1}\xi)_\im \seq \xi_{\im^*} \pm rh^\vee \quad \text{for $\im \in \Delta_0$}.
\end{equation}
One can easily check that $(\fD^{-1}\xi)_\im < \xi_\im < (\fD\xi)_\im$ holds for each $\im \in \Delta_0$ and that $\fD^{\pm 1}\cQ$ satisfies the condition \eqref{eq:xiep} whenever $\cQ$ does.

Now, given a Q-datum $\cQ=(\Delta, \sigma, \xi)$ for $\fg$, we consider the following sets 
\begin{align*}
\hD_{[\xi]} & \seq \{ (\im, p) \in \Delta_0 \times \Z \mid \xi_\im - p \in 2d_{\bar{\im}}\Z \}, &&\\ 
\hD_{\le \xi} & \seq \{ (\im, p) \in \hD_{[\xi]} \mid p \le \xi_\im \}, 
& \hI_{\le \xi} &\seq f(\hD_{\le \xi}), \\
\hD_{\cQ} & \seq \{ (\im, p) \in \hD_{[\xi]} \mid (\fD^{-1}\xi)_{\im} < p \le \xi_{\im} \}, 
& \hI_\cQ &\seq f(\hD_{\cQ}), 
\end{align*}
where $f \colon \Delta_0 \times \Z \to I \times \Z$ is the \emph{folding map} $(\im,p) \mapsto (\bar{\im},p)$.
Note that we have $\hI_{\cQ} \subset \hI_{\le \xi} \subset \hI$, and $f$ induces the bijections $\hD_{[\xi]} \simeq \hI$, $\hD_{\le \xi} \simeq \hI_{\le \xi}$, and $\hD_{\cQ} \simeq \hI_{\cQ}$.
Let $\cY_{t, \le \xi}$ (resp.~$\cY_{t, \cQ}$) be the $\Z[t^{\pm 1/2}]$-subalgebra of the quantum torus $\cY_t$ generated by all the elements $\ul{Y_{i,p}^{\pm 1}}$ with $(i,p) \in \hI_{\le \xi}$ (resp.~$\hI_\cQ$).
Then the algebra $\cY_{\le \xi} \seq \evt(\cY_{t,\le \xi})$ (resp.~$\cY_{\cQ} \seq \evt(\cY_{t,\cQ})$) is identical to the ring of Laurent polynomials in the variables $Y_{i,p}$ with $(i,p) \in \hI_{\le \xi}$ (resp.~$\hI_{\cQ}$).
We set $\cM_{\le \xi} \seq \cM \cap \cY_{\le \xi}$ and $\cM_{\cQ} \seq \cM \cap \cY_{\cQ}$. 

Let $\Cc_{\le \xi}$ (resp.~$\Cc_{\cQ}$) be the Serre subcategory of $\Cc_\Z$ generated by all the simple modules $L(m)$ with $m \in \cM_{\le \xi}$ (resp.~$\cM_\cQ$).
It is closed under taking tensor products. 
Moreover, the category $\Cc_{\le \xi}$ is stable under the functor $\fD^{-1}$, and we have $\fD^{-1}\Cc_{\le \xi} = \Cc_{\le \fD^{-1}\xi}$.
We define $\cK_t(\Cc_{\le \xi})$ (resp.~$\cK_t(\Cc_\cQ)$) to be the $\Z[t^{\pm 1/2}]$-subalgebra of $\cK_t(\Cc_\Z)$ generated by all the $F_t(Y_{i,p})$ with $(i,p) \in \hI_{\le \xi}$ (resp.~$\hI_\cQ$).
It is endowed with the canonical $\Z[t^{\pm 1/2}]$-basis $\bfL_{t, \le \xi}$ (resp.~$\bfL_{t, \cQ}$) consisting of all the elements $L_t(m)$ with $m \in \cM_{\le \xi}$ (resp.~$\cM_\cQ$). 
Moreover, the algebra $\cK_t(\Cc_{\le \xi})$ is stable under the automorphism $\fD_t^{-1}$, and we have $\fD_t^{-1}\cK_t(\Cc_{\le \xi}) = \cK_t(\Cc_{\le \fD^{-1}\xi})$. 

Let us recall the following truncation procedure.
For an element $y \in \cY_{t}$, we denote by $y_{\le \xi}$
the element of $\cY_{t, \le \xi}$ obtained from $y$ by discarding all the monomials containing  
the factors $\ul{Y_{i, p}^{\pm 1}}$ with $(i,p) \in \hI \setminus \hI_{\le \xi}$.
This assignment $y \mapsto y_{\le \xi}$ defines 
a $\Z[t^{\pm 1/2}]$-linear map $(\cdot)_{\le \xi} \colon \cY_{t} \to \cY_{t, \le \xi}$,
which is not an algebra homomorphism. 
%We denote the similar $\Z$-linear map defined for $\cY$ by the same symbol $(\cdot)_{\le \xi} \colon \cY \to \cY_{\le \xi}$, so that we have $\evt \circ (\cdot)_{\le \xi} = (\cdot)_{\le \xi} \circ \evt$.
However, it restricts to the injective algebra homomorphisms $\cK_{t}(\Cc_{\le \xi}) \hookrightarrow \cY_{t, \le \xi}$ and $\cK_{t}(\Cc_{\cQ}) \hookrightarrow \cY_{t, \cQ}$ (see \cite[\S5.4]{FHOO}). 

For future use, we remark the following fact. 

\begin{Lem} \label{Lem:tr} 
Let $\cQ= (\Delta, \sigma, \xi)$ be a Q-datum for $\fg$. We have
\[ \cK_t(\Cc_{\le \xi}) = \cK_t(\Cc_\Z) \cap \cY_{t, \le \fD \xi}. \]
In other words, an element $x \in \cK_t(\Cc_\Z)$ belongs to $\cK_t(\Cc_{\le \xi})$ if and only if it satisfies that $x_{\le \fD\xi} = x$.  
\end{Lem}
\begin{proof}
We consider the subset $\hI_{\ge -\xi} \seq \{ (i,p) \in \hI \mid (i,-p) \in \hI_{\le \xi}\}$ of $\hI$ and the corresponding $\Z[t^{\pm 1/2}]$-subalgebra $\cY_{t, \ge -\xi}$ of $\cY_t$.
Then, for each $(i,p) \in \hI_{\ge -\xi}$, we have $F_t(Y_{i,p}) \in \cY_{\ge -\xi}$ (see \cite[Proof of Corollary 5.12]{FHOO}).
In particular, if we define $\cK_t(\Cc_{\ge -\xi})$ to be the $\Z[t^{\pm 1/2}]$-subalgebra generated by $\{ F_t(Y_{i,p}) \mid (i,p) \in \hI_{\ge -\xi}\}$, we have
\[ \cK_t(\Cc_{\ge -\xi}) = \cK_{t}(\Cc_\Z) \cap \cY_{t, \ge -\xi}. \]
Let $\omega_t$ be an algebra involution of $\cY_t$ given by $\omega_t(t^{\pm 1/2}) = t^{\mp 1/2}$ and $\omega_t(\ul{Y_{i,p}}) = \ul{Y_{i,-p}^{-1}}$ for all $(i,p) \in \hI$.
By definition, we have $\omega_t(\cY_{t, \le \xi}) = \cY_{t, \ge -\xi}$.
Moreover, it satisfies $\omega_t(F_t(Y_{i,p})) = F_t(Y_{i^*, -p-rh^\vee})$ for any $(i,p) \in \hI$ (see \cite[\S3.4]{FHOO}).
Therefore, we have 
\[ \omega_t(\cK_t(\Cc_{\le \xi})) = \cK_t(\Cc_{\ge -\fD \xi}) = \cK_{t}(\Cc_\Z) \cap \cY_{t, \ge -\fD\xi} = \omega_t(\cK_{t}(\Cc_\Z) \cap \cY_{t, \le \fD\xi}), \]
which yields the assertion.
\end{proof}

\subsection{Kirillov--Reshetikhin modules}

Recall that a \emph{Kirillov--Reshetikhin (KR) module} is a simple object of $\Cc$ whose Drinfeld polynomials are of the form 
\[\left(\prod_{k=0}^l(1-aq^{2d_ik}z)^{\delta_{i,j}}\right)_{j \in I}\] 
for some $i \in I$, $l \in \N$ and $a \in \kk^\times$. 
\begin{Prop}
\label{Prop:KRF}
For a KR module $L(m)$, we have $\evt (F_t(m)) = \chi_q(L(m))$.
\end{Prop}
\begin{proof} The $q$-character of a KR-module has a unique 
dominant monomial as proved in \cite{Nakajima03II, Her06}.
Moreover, the image of the quantum Grothendieck ring by $\evt$ is the ring of $q$-characters \cite{Nak04, Her04}. 
But a $q$-character is characterized by the multiplicity of its dominant monomial \cite{FM01}. This implies the result.
\end{proof}

Having Proposition~\ref{Prop:KRF}, in view of Conjecture~\ref{Conj:KL}, one may expect the following.

\begin{Conj} \label{Conj:KRF}
For a KR module $L(m)$, we have $F_t(m) = L_t(m)$.
\end{Conj}

At this moment, we know that Conjecture~\ref{Conj:KRF} is true when $\fg$ is of type $\mathrm{ABDE}$ by \cite{Nakajima03II} and \cite[Theorem 11.6]{FHOO}. We will prove Conjecture \ref{Conj:KRF} for the remaining case as one of the results of this paper, that is, when $\fg$ is of type $\mathrm{CFG}$. See Corollary \ref{Cor:KRF}.  

For our purpose, it is useful to introduce the following notation for KR modules.
Let $\cQ = (\Delta, \sigma, \xi)$ be a Q-datum for $\fg$.
For $\im \in \Delta_0$ and integers $a,b \in \Z$ with $a \le b$, we define a dominant monomial $m^{(\im)}[a,b] \in \cM$ by
\[ m^{(\im)}[a,b] \seq \prod_{(\im, p) \in \hD_{[\xi]} \colon p \in [a,b]} Y_{\bar{\im},p}. \]
The corresponding simple module $L^{(\im)}[a,b] \seq L(m^{(\im)}[a,b])$ is a KR module. 
Any KR module in the category $\Cc_\Z$ is written in this form.
%In addition, any KR module in the category $\Cc_\Z$ belongs to $\Cc_{\le \xi}$ for some height function $\xi$. 
In what follows, we use simplify the notation by setting $F_t^{(\im)}[a,b] \seq F_t(m^{(\im)}[a,b])$ and $L_t^{(\im)}[a,b]\seq L_t(m^{(\im)}[a,b])$. 
Letting $(a,b] \seq [a,b]\setminus \{a\}$, $[a,b) \seq [a,b]\setminus\{b\}$, $(a,b) \seq [a,b] \setminus \{a,b\}$, we define $F_t^{(\im)}(a,b]$, $F_t^{(\im)}[a,b)$, $F_t^{(\im)}(a,b)$ in the same way as $F^{(\im)}_t[a,b]$.

The following states that certain truncations of these elements are just a single monomial.

\begin{Lem}[{\cite[Lemma 6.7]{FHOO}}]
\label{Lem:qKR}
For any $(\im, p) \in \hD_{\le \xi}$, we have
\[ F_{t}^{(\im)}[p,\xi_\im]_{\le \xi}= \ul{m^{(\im)}[p,\xi_{\im}]}.\]
\end{Lem}

The following generalizes the quantum $T$-system in simply-laced types
\cite{Nakajima03II} and deforms the general $T$-systems \cite{Her06}. 

\begin{Thm}[Quantum $T$-system {\cite[Theorem 6.8]{FHOO}}] \label{Thm:qTsys}
Let $\cQ = (\Delta, \sigma, \xi)$ be a Q-datum for $\fg$ and let $(\im, p), (\im, s) \in \hD_{[\xi]}$ satisfy $p < s$.
Then the elements in $\{ F_t^{(\jm)}(p, s) \}_{\jm \sim \im}$ are mutually commutative up to powers of $t^{\pm 1/2}$, and we have
\begin{equation} \label{eq:qTsys}
F_{t}^{(\im)}[p,s) F_{t}^{(\im)}(p,s]=
t^{a} F_{t}^{(\im)}(p,s) F_{t}^{(\im)}[p,s] + 
t^{b} \prod_{\jm \sim \im}^{\to} F_{t}^{(\jm)}(p,s)
\end{equation}
in $\cK_t(\Cc_\Z)$ for some $a,b \in \frac{1}{2}\Z$.
Here the ordered product is taken with respect to any total ordering of the set $\{ \jm \in \Delta_0 \mid \jm \sim \im\}$.
\end{Thm}

Note that the factors of the last term do not commute in general but only $t$-commute, that is why the product has to be ordered.

\section{Cluster structure on $\cK_t(\Cc_{\le \xi})$}\label{ssecfive}

We show that the quantum Grothendieck ring of the category $\Cc_{\le \xi}$ has a structure of a quantum cluster algebras 
isomorphic to an algebra $\cA_\bfi$ introduced above (Theorem \ref{Thm:qcl}). This generalizes previous results \cite{HL15, Qin17, Bit}. 
For our purposes, we use the quantum affine quiver introduced in \cite{HL16} and for which we discuss several technical results. 
Then our proof is based first on an isomorphism of quantum tori that we establish (Corollary \ref{Cor:teta}). The isomorphism between the quantum Grothendieck ring and the quantum cluster algebra is then obtained by identifying their respective image by the truncation of $(q,t)$-characters and the natural inclusion, inside the quantum tori.

\subsection{Adapted sequences}
\label{Ssec:adpt}
Throughout this section, we fix a Q-datum $\cQ = (\Delta, \sigma, \xi)$ for $\fg$.
A vertex $\im \in \Delta_{0}$ is called a \emph{source} of $\cQ$ if we have $\xi_{\im} > \xi_{\jm}$
for any $\jm \in \Delta_{0}$ with $\jm \sim \im$.
In this case, we define a new Q-datum $s_{\im}\cQ = (\Delta, \sigma, s_{\im}\xi)$
by setting 
$$
(s_{\im} \xi)_{\jm} \seq \xi_{\jm} - 2 d_{\bar{\im}}\delta_{\im,\jm} \quad \text{for $\jm \in \Delta_{0}$}. 
$$
Note that, if $\cQ$ satisfies \eqref{eq:xiep}, so does $s_{\im}\cQ$ for any source $\im \in \Delta_0$ of $\cQ$. 
We say that a sequence $\bfi = (\im_1, \im_2, \ldots)$ in $\Delta_0$ 
is \emph{adapted to $\cQ$} if $\im_k$ is a source of 
the Q-datum $s_{\im_{k-1}} \cdots s_{\im_2} s_{\im_1} \cQ$ 
for all $k \in \{ 1, 2, \ldots\}$.  
We have the following combinatorial result, which is crucial for our purposes. 
Recall the notion of commutation equivalence from Definition~\ref{Def:comm-eq}. 

\begin{Prop}[{\cite[\S3.5]{FO21}}] \label{Prop:DQ}
For each Q-datum $\cQ$ for $\fg$, there is a reduced word $(\im_1, \ldots, \im_\ell)$ for the longest element $w_\circ$ adapted to $\cQ$, uniquely up to commutation-equivalence.
Moreover, we have $s_{\im_\ell} \cdots s_{\im_1}\cQ = \fD^{-1}\cQ$.
\end{Prop}

Let $\bfi = (\im_k)_{k \in \N} \in \Delta_0^{(\infty)}$ be a sequence satisfying the condition~\eqref{eq:cond1}.
In what follows, we use the notation
\[ n_\bfi(u) \seq |\{ v \in \N \mid v < u, \im_v = \im_u \}|.\]
Then we define the map $\rho_\bfi \colon \N \to \hD_{\le \xi}$ by
\[ \rho_\bfi(u) \seq (\im_u, \xi_{\im_u} - 2d_{\bar{\im}_u} n_\bfi(u)).\] 
By the condition~\eqref{eq:cond1}, $\rho_\bfi$ is a bijection.
We set $\bar{\rho}_\bfi \seq f \circ \rho_\bfi \colon \N \to \hI_{\le \xi}$.

For $(\im,p), (\jm, s) \in \hD_{[\xi]}$, write $(\im, p) \prec (\jm,s)$ if $\im \sim \jm$ and $p = s+ \min(d_{\bar{\im}}, d_{\bar{\jm}})$. 
Taking the transitive closure of this relation, we obtain a partial ordering $\preceq$ on the set $\hD_{[\xi]}$. 

\begin{Rem} \label{Rem:Hasse}
The Hasse diagram of the poset $(\hD_{\le \xi}, \preceq)$ is identical to the repetition quiver (restricted to $\hD_{\le \xi}$) in the sense of \cite[\S3.4]{FO21}. 
\end{Rem}

\begin{Lem}
\label{Lem:ad}
An infinite sequence $\bfi \in \Delta_0^\N$ is adapted to $\cQ$ if and only if the condition \eqref{eq:cond1} is satisfied and the bijection $\rho_\bfi^{-1} \colon (\hD_{\le \xi}, \preceq) \to (\N, \le)$ is a morphism of posets.
\end{Lem}
\begin{proof}
It is easy to see that the condition \eqref{eq:cond1} is satisfied if $\bfi \in \Delta_0^\N$ is adapted to $\cQ$. 
Then the assertion follows from the following observation:
 $(\im, p) \in \hD_{\le \xi}$ is minimal if and only if $p=\xi_{\im}$ and $\im$ is a source of $\cQ$. Moreover, if this is the case, we have $\hD_{\le \xi} \setminus \{(\im, \xi_\im)\} = \hD_{\le s_{\im}\xi}$. 
\end{proof}

\begin{Lem} \label{Lem:adce}
Assume that two sequences $\bfi = (\im_{u})_{u \in \N}$ and $\bfi' = (\im'_{u})_{u \in \N}$  are both adapted to $\cQ$.
Then $\bfi$ and $\bfi'$ are commutation-equivalent by the transformation $\rho_{\bfi'}^{-1} \circ \rho_{\bfi}$.  
\end{Lem}
\begin{proof}
Letting $\pi \seq \rho_{\bfi'}^{-1} \circ \rho_{\bfi} \in \SG_\N$, we have $\im_{u} = \im'_{\pi(u)}$ for all $u \in \N$. 
Assume that two positive integers $u, v \in \Z$ satisfy $u < v$ and $\pi(u) > \pi(v)$. 
By Lemma~\ref{Lem:ad}, $\rho_{\bfi}(u)$ and $\rho_{\bfi}(v)$ are not comparable in $(\hD_{\le \xi}, \preceq)$.
Then, we have $\im_u \neq \im_v$ and $\im_u \not\sim \im_v$ (see \cite[Remark 3.17]{FO21} together with Remark~\ref{Rem:Hasse}).
Therefore, we obtain the conclusion. 
\end{proof}

\begin{Ex} \label{Ex:periodic}
Let $\bfi = (\im_u)_{u \in \N} \in \Delta_0^\N$ be a sequence satisfying the condition
\begin{equation} \label{eq:cond2}
\begin{cases}
(1) & \text{$(\im_1, \ldots, \im_{\ell})$ is a reduced word for $w_\circ$ adapted to $\cQ$, and} \\
(2) & \text{we have $\im_{u+\ell} = \im_u^*$ for all $u \in \N$.}
\end{cases}
\end{equation}
Then, the sequence $\bfi$ is adapted to $\cQ$ by Proposition~\ref{Prop:DQ}.
\end{Ex}

\begin{Ex} \label{Ex:dimin}
Let $\{ (\im_u,p_u) \}_{u \in \N}$ be an arbitrary total ordering of the set $\hD_{\le \xi}$ satisfying $p_1 \ge p_2 \ge \cdots$.
Then, the sequence $\bfi \seq (\im_u)_{u \in \N}$ is adapted to $\cQ$ by Lemma~\ref{Lem:ad} and we have $\rho_{\bfi}(u) = (\im_u,p_u)$ for all $u \in \N$.
\end{Ex}

{\color{green}\subsection{Application to the quantum affine quiver}}
Following \cite{HL16}, we define the quiver $G$ as follows. 
The set of vertices of $G$ is $\hI$.  
For $(i,p), (j,s) \in \hI$, we assign an arrow $(i,p) \to (j,s)$ if
\[ c_{ij} \neq 0 \quad \text{and} \quad  s-d_j = p-d_i + d_i c_{ij}. \] 
Let $G_{\le \xi}$ denote the full subquiver of $G$ supported on the set $\hI_{\le \xi}$.

\begin{Rem} \label{Rem:HL}
We can always find a Q-datum $\cQ$ for $\fg$ whose height function $\xi$ satisfies $-2d_{\bar{\im}} < \xi_\im \le 0$ for all $\im \in \Delta_0$. 
In this case, the quiver $G_{\le \xi}$ is identical to the quiver $G^-$ defined in \cite{HL16}.
In this sense, the quiver $G_{\le \xi}$ is a generalization of the quiver $G^-$.
\end{Rem}

\begin{Lem}[{\cite[Proposition 7.27]{KKOP2}}] \label{Lem:quiverisom}
For any sequence $\bfi$ adapted to $\cQ$, the bijection $\bar{\rho}_\bfi \colon \N \to \hI_{\le \xi}$ induces the quiver isomorphism
\[ \Gamma_\bfi \simeq G_{\le \xi}.  \]
\end{Lem} 
\begin{proof}
By Proposition~\ref{Prop:comm-eq} and Lemma~\ref{Lem:adce}, we may assume that $\bfi$ satisfies \eqref{eq:cond2} in Example~\ref{Ex:periodic}.
Then, the assertion is identical to \cite[Proposition 7.27]{KKOP2}. 
\end{proof}

\begin{Ex}
We exhibit some examples of the quiver $G_{\le \xi}$. Here the labeling of $I$ is as in Figure~\ref{Fig:unf}. 
The symbols $\star$ indicate the vertices $(\bar{\im}, \xi_{\im})$ for $\im \in \Delta_0$.
\begin{itemize}
\item 
Type $\mathrm{A}_{5}$:
$$
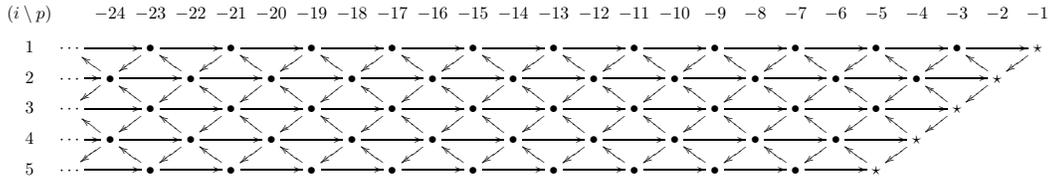
\raisebox{3mm}{
\scalebox{0.60}{\xymatrix@!C=0.5mm@R=2mm{
(i\setminus p) & & -24 & -23 &-22&-21 &-20& -19 &-18& -17 & -16& -15& -14& -13&  -12
& -11 & -10 & -9 & -8 & -7 & -6 & -5 & -4 & -3 & -2 & -1  \\
1&\cdots \ar@{->}[rr] \ar@{<-}[dr]&& \bullet \ar@{->}[rr] \ar@{<-}[dr] &&\bullet \ar@{->}[rr]\ar@{<-}[dr]
&&\bullet \ar@{->}[rr] \ar@{<-}[dr] && \bullet \ar@{->}[rr] \ar@{<-}[dr] &&\bullet \ar@{->}[rr] \ar@{<-}[dr] &&  \bullet \ar@{->}[rr] \ar@{<-}[dr] 
&&\bullet \ar@{->}[rr] \ar@{<-}[dr] && \bullet \ar@{->}[rr] \ar@{<-}[dr] &&\bullet \ar@{->}[rr] \ar@{<-}[dr]  && \bullet \ar@{->}[rr]\ar@{<-}[dr] && \bullet \ar@{->}[rr] \ar@{<-}[dr] && \star  \\
2&\cdots \ar@{->}[r] &\bullet \ar@{->}[rr] \ar@{<-}[dr]\ar@{<-}[ur]&& \bullet \ar@{->}[rr] \ar@{<-}[dr]\ar@{<-}[ur] &&\bullet \ar@{->}[rr] \ar@{<-}[dr]\ar@{<-}[ur]
&& \bullet \ar@{->}[rr] \ar@{<-}[dr]\ar@{<-}[ur]&& \bullet \ar@{->}[rr]\ar@{<-}[dr] \ar@{<-}[ur]&& \bullet \ar@{->}[rr] \ar@{<-}[dr]\ar@{<-}[ur]&&\bullet \ar@{->}[rr] \ar@{<-}[dr]\ar@{<-}[ur]&
&\bullet \ar@{->}[rr] \ar@{<-}[dr]\ar@{<-}[ur]&&\bullet \ar@{->}[rr]\ar@{<-}[dr] \ar@{<-}[ur]&& \bullet \ar@{->}[rr] \ar@{<-}[dr]\ar@{<-}[ur]
&&\bullet \ar@{->}[rr] \ar@{<-}[dr]\ar@{<-}[ur]&& \star \ar@{<-}[ur]& \\ 
3&\cdots \ar@{->}[rr] \ar@{<-}[dr] \ar@{<-}[ur]&& \bullet \ar@{->}[rr] \ar@{<-}[dr] \ar@{<-}[ur] &&\bullet \ar@{->}[rr]\ar@{<-}[dr] \ar@{<-}[ur]
&&\bullet \ar@{->}[rr] \ar@{<-}[dr] \ar@{<-}[ur] && \bullet \ar@{->}[rr] \ar@{<-}[dr]\ar@{<-}[ur] &&\bullet \ar@{->}[rr] \ar@{<-}[dr] \ar@{<-}[ur]&&  \bullet \ar@{->}[rr] \ar@{<-}[dr] \ar@{<-}[ur]
&&\bullet \ar@{->}[rr] \ar@{<-}[dr] \ar@{<-}[ur] && \bullet \ar@{->}[rr] \ar@{<-}[dr] \ar@{<-}[ur]&&\bullet \ar@{->}[rr] \ar@{<-}[dr] \ar@{<-}[ur] && \bullet \ar@{->}[rr]\ar@{<-}[dr] \ar@{<-}[ur]&&
\star \ar@{<-}[ur] && \\
4&\cdots \ar@{->}[r]& \bullet \ar@{->}[rr] \ar@{<-}[ur]\ar@{<-}[dr]&&\bullet \ar@{->}[rr] \ar@{<-}[ur]\ar@{<-}[dr]&&\bullet \ar@{->}[rr] \ar@{<-}[ur]\ar@{<-}[dr] &&\bullet \ar@{->}[rr] \ar@{<-}[ur]\ar@{<-}[dr]&& \bullet \ar@{->}[rr] \ar@{<-}[ur]\ar@{<-}[dr]
&&\bullet \ar@{->}[rr] \ar@{<-}[ur]\ar@{<-}[dr]&& \bullet \ar@{->}[rr] \ar@{<-}[ur]\ar@{<-}[dr] &&\bullet \ar@{->}[rr] \ar@{<-}[ur]\ar@{<-}[dr]&&\bullet \ar@{->}[rr] \ar@{<-}[ur]\ar@{<-}[dr]&&
\bullet \ar@{->}[rr] \ar@{<-}[ur]\ar@{<-}[dr]&&\star \ar@{<-}[ur]&&& \\
5&\cdots \ar@{->}[rr]  \ar@{<-}[ur]&& \bullet \ar@{->}[rr]  \ar@{<-}[ur] &&\bullet \ar@{->}[rr] \ar@{<-}[ur]
&&\bullet \ar@{->}[rr]  \ar@{<-}[ur] && \bullet \ar@{->}[rr] \ar@{<-}[ur] &&\bullet \ar@{->}[rr]  \ar@{<-}[ur]&&  \bullet \ar@{->}[rr]  \ar@{<-}[ur]
&&\bullet \ar@{->}[rr]  \ar@{<-}[ur] && \bullet \ar@{->}[rr]  \ar@{<-}[ur]&&\bullet \ar@{->}[rr]  \ar@{<-}[ur] && \star \ar@{<-}[ur]&&&& }}}
$$
\item  
Type $\mathrm{D}_{5}$:
$$
\raisebox{3mm}{
\scalebox{0.60}
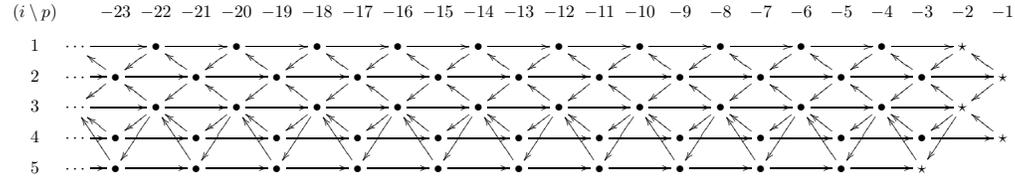
{\xymatrix@!C=0.5mm@R=2mm{
(i\setminus p) & & -23 & -22 &-21&-20 &-19& -18 &-17& -16 & -15 & -14 & -13& -12&  -11
& -10 & -9 & -8 & -7 & -6 & -5 & -4 & -3 & -2 & -1  \\
1&\cdots \ar@{->}[rr] \ar@{<-}[dr]&& \bullet \ar@{->}[rr] \ar@{<-}[dr] &&\bullet \ar@{->}[rr]\ar@{<-}[dr]
&&\bullet \ar@{->}[rr] \ar@{<-}[dr] && \bullet \ar@{->}[rr] \ar@{<-}[dr] &&\bullet \ar@{->}[rr] \ar@{<-}[dr] &&  \bullet \ar@{->}[rr] \ar@{<-}[dr] 
&&\bullet \ar@{->}[rr] \ar@{<-}[dr] && \bullet \ar@{->}[rr] \ar@{<-}[dr] &&\bullet \ar@{->}[rr] \ar@{<-}[dr]  && \bullet \ar@{->}[rr]\ar@{<-}[dr] &&
\star \ar@{<-}[dr] & \\
2&\cdots \ar@{->}[r]&\bullet \ar@{->}[rr] \ar@{<-}[dr]\ar@{<-}[ur]&& \bullet \ar@{->}[rr] \ar@{<-}[dr]\ar@{<-}[ur] &&\bullet \ar@{->}[rr] \ar@{<-}[dr]\ar@{<-}[ur]
&& \bullet \ar@{->}[rr] \ar@{<-}[dr]\ar@{<-}[ur]&& \bullet \ar@{->}[rr]\ar@{<-}[dr] \ar@{<-}[ur]&& \bullet \ar@{->}[rr] \ar@{<-}[dr]\ar@{<-}[ur]&&\bullet \ar@{->}[rr] \ar@{<-}[dr]\ar@{<-}[ur]&
&\bullet \ar@{->}[rr] \ar@{<-}[dr]\ar@{<-}[ur]&&\bullet \ar@{->}[rr]\ar@{<-}[dr] \ar@{<-}[ur]&& \bullet \ar@{->}[rr] \ar@{<-}[dr]\ar@{<-}[ur]
&&\bullet \ar@{->}[rr] \ar@{<-}[dr]\ar@{<-}[ur]&& \star & \\ 
3&\cdots \ar@{->}[rr] \ar@{<-}[dr] \ar@{<-}[ur]&& \bullet \ar@{->}[rr] \ar@{<-}[dr] \ar@{<-}[ur] &&\bullet \ar@{->}[rr]\ar@{<-}[dr] \ar@{<-}[ur]
&&\bullet \ar@{->}[rr] \ar@{<-}[dr] \ar@{<-}[ur] && \bullet \ar@{->}[rr] \ar@{<-}[dr]\ar@{<-}[ur] &&\bullet \ar@{->}[rr] \ar@{<-}[dr] \ar@{<-}[ur]&&  \bullet \ar@{->}[rr] \ar@{<-}[dr] \ar@{<-}[ur]
&&\bullet \ar@{->}[rr] \ar@{<-}[dr] \ar@{<-}[ur] && \bullet \ar@{->}[rr] \ar@{<-}[dr] \ar@{<-}[ur]&&\bullet \ar@{->}[rr] \ar@{<-}[dr] \ar@{<-}[ur] && \bullet \ar@{->}[rr]\ar@{<-}[dr] \ar@{<-}[ur]&&
\star \ar@{<-}[dr] \ar@{<-}[ur] &   \\
4&\cdots \ar@{->}[r]& \bullet \ar@{->}[rr] \ar@{<-}[ur]&&\bullet \ar@{->}[rr] \ar@{<-}[ur]&&\bullet \ar@{->}[rr] \ar@{<-}[ur] &&\bullet \ar@{->}[rr] \ar@{<-}[ur]&& \bullet \ar@{->}[rr] \ar@{<-}[ur]
&&\bullet \ar@{->}[rr] \ar@{<-}[ur]&& \bullet \ar@{->}[rr] \ar@{<-}[ur] &&\bullet \ar@{->}[rr] \ar@{<-}[ur]&&\bullet \ar@{->}[rr] \ar@{<-}[ur]&&
\bullet \ar@{->}[rr] \ar@{<-}[ur]&&\bullet \ar@{->}[rr] \ar@{<-}[ur]&&\star  \\
5&\cdots \ar@{->}[r]& \bullet \ar@{->}[rr] \ar@{->}[uul]\ar@{<-}[uur]&&\bullet \ar@{->}[rr] \ar@{->}[uul]\ar@{<-}[uur]&&\bullet \ar@{->}[rr] \ar@{->}[uul]\ar@{<-}[uur] &&\bullet \ar@{->}[rr] \ar@{->}[uul]\ar@{<-}[uur]&& \bullet \ar@{->}[rr] \ar@{->}[uul]\ar@{<-}[uur]
&&\bullet \ar@{->}[rr] \ar@{->}[uul]\ar@{<-}[uur]&& \bullet \ar@{->}[rr] \ar@{->}[uul]\ar@{<-}[uur] &&\bullet \ar@{->}[rr] \ar@{->}[uul]\ar@{<-}[uur]&&\bullet \ar@{->}[rr] \ar@{->}[uul]\ar@{<-}[uur]&&
\bullet \ar@{->}[rr] \ar@{->}[uul]\ar@{<-}[uur]&&\star \ar@{->}[uul]\ar@{<-}[uur]&&
}}}
$$
\item
Type $\mathrm{B}_{3}$:
$$
\raisebox{3mm}{
\scalebox{0.60}
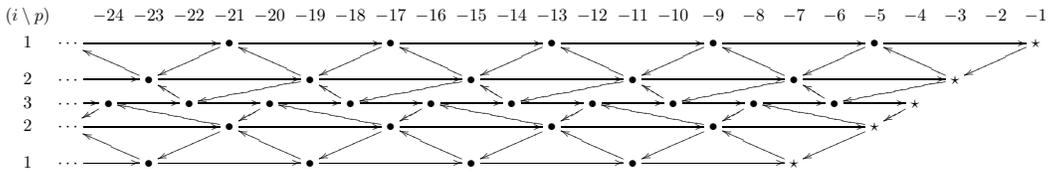
{\xymatrix@!C=0.5mm@R=0.5mm{
(i\setminus p)& &-24 & -23 & -22 & -21 &-20&-19 &-18& -17 &-16& -15 & -14& -13 & -12& -11&  -10
& -9 & -8 & -7 & -6 & -5 & -4 & -3 & -2 & -1 &  \\
1&\cdots \ar@{->}[rrrr] \ar@{<-}[ddrr] &&&&\bullet \ar@{->}[rrrr] \ar@{<-}[ddrr]&&&& \bullet \ar@{->}[rrrr] \ar@{<-}[ddrr]&&&&  \bullet \ar@{->}[rrrr] \ar@{<-}[ddrr] &&&& \bullet \ar@{->}[rrrr]\ar@{<-}[ddrr]
&&&& \bullet \ar@{->}[rrrr] \ar@{<-}[ddrr]&&&& \star & \\ \\
2&\cdots \ar@{->}[rr]&&\bullet \ar@{->}[rrrr] \ar@{<-}[dr] \ar@{<-}[uurr]&&&&\bullet \ar@{->}[rrrr]\ar@{<-}[dr] \ar@{<-}[uurr]
&&&& \bullet \ar@{->}[rrrr] \ar@{<-}[dr]\ar@{<-}[uurr] &&&&  \bullet \ar@{->}[rrrr] \ar@{<-}[dr] \ar@{<-}[uurr]
&&&& \bullet \ar@{->}[rrrr] \ar@{<-}[dr]\ar@{<-}[uurr]&&&& \star \ar@{<-}[uurr]&&& \\
3&\cdots \ar@{->}[r]&\bullet \ar@{->}[rr] \ar@{<-}[drrr]&& \bullet \ar@{->}[rr] \ar@{<-}[urrr] && \bullet \ar@{->}[rr] \ar@{<-}[drrr]  &&\bullet \ar@{->}[rr] \ar@{<-}[urrr]  && \bullet \ar@{->}[rr] \ar@{<-}[drrr]  && \bullet \ar@{->}[rr] \ar@{<-}[urrr] 
&& \bullet \ar@{->}[rr] \ar@{<-}[drrr]  && \bullet \ar@{->}[rr] \ar@{<-}[urrr]  &&
\bullet \ar@{->}[rr] \ar@{<-}[drrr] &&\bullet \ar@{->}[rr] \ar@{<-}[urrr]  &&\star &&  && \\
2&\cdots \ar@{->}[rrrr] \ar@{<-}[ddrr] \ar@{<-}[ur]&&&&\bullet \ar@{->}[rrrr]\ar@{<-}[ddrr]\ar@{<-}[ur]&&&&\bullet \ar@{->}[rrrr]\ar@{<-}[ddrr]\ar@{<-}[ur] &&&&  \bullet \ar@{->}[rrrr] \ar@{<-}[ddrr]\ar@{<-}[ur] &&&&
\bullet \ar@{->}[rrrr] \ar@{<-}[ddrr] \ar@{<-}[ur]
&&&& \star\ar@{<-}[ur]&&&&& \\ \\
1& \cdots \ar@{->}[rr] && \bullet \ar@{->}[rrrr]  \ar@{<-}[uurr]&&&&\bullet \ar@{->}[rrrr] \ar@{<-}[uurr] &&&& \bullet \ar@{->}[rrrr] \ar@{<-}[uurr]
&&&& \bullet \ar@{->}[rrrr]   \ar@{<-}[uurr] &&&& \star \ar@{<-}[uurr]
&&&& &&&  }}}
$$
\item 
Type $\mathrm{C}_{4}$:
$$
\raisebox{3mm}{
\scalebox{0.60}
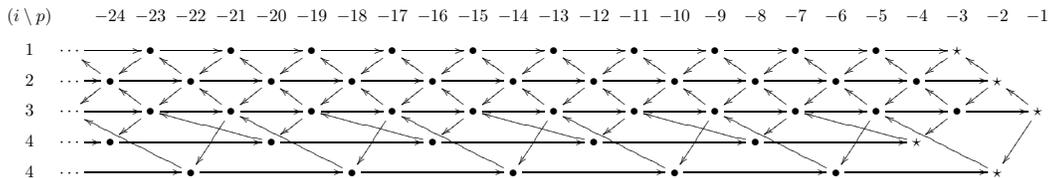
{\xymatrix@!C=0.5mm@R=2mm{
(i\setminus p) & & -24 & -23 &-22&-21 &-20& -19 &-18& -17 & -16& -15 & -14& -13&  -12
& -11 & -10 & -9 & -8 & -7 & -6 & -5 & -4 & -3 & -2 & -1 \\
1&\cdots \ar@{->}[rr] \ar@{<-}[dr]&& \bullet \ar@{->}[rr] \ar@{<-}[dr] &&\bullet \ar@{->}[rr]\ar@{<-}[dr]
&&\bullet \ar@{->}[rr] \ar@{<-}[dr] && \bullet \ar@{->}[rr] \ar@{<-}[dr] &&\bullet \ar@{->}[rr] \ar@{<-}[dr] &&  \bullet \ar@{->}[rr] \ar@{<-}[dr] 
&&\bullet \ar@{->}[rr] \ar@{<-}[dr] && \bullet \ar@{->}[rr] \ar@{<-}[dr] &&\bullet \ar@{->}[rr] \ar@{<-}[dr]  && \bullet \ar@{->}[rr]\ar@{<-}[dr] &&
\star \ar@{<-}[dr] && \\
2&\cdots \ar@{->}[r]&\bullet \ar@{->}[rr] \ar@{<-}[dr]\ar@{<-}[ur]&& \bullet \ar@{->}[rr] \ar@{<-}[dr]\ar@{<-}[ur] &&\bullet \ar@{->}[rr] \ar@{<-}[dr]\ar@{<-}[ur]
&& \bullet \ar@{->}[rr] \ar@{<-}[dr]\ar@{<-}[ur]&& \bullet \ar@{->}[rr]\ar@{<-}[dr] \ar@{<-}[ur]&& \bullet \ar@{->}[rr] \ar@{<-}[dr]\ar@{<-}[ur]&&\bullet \ar@{->}[rr] \ar@{<-}[dr]\ar@{<-}[ur]&
&\bullet \ar@{->}[rr] \ar@{<-}[dr]\ar@{<-}[ur]&&\bullet \ar@{->}[rr]\ar@{<-}[dr] \ar@{<-}[ur]&& \bullet \ar@{->}[rr] \ar@{<-}[dr]\ar@{<-}[ur]
&&\bullet \ar@{->}[rr] \ar@{<-}[dr]\ar@{<-}[ur]&& \star \ar@{<-}[dr] & \\ 
3&\cdots \ar@{->}[rr] \ar@{<-}[ddrrr] \ar@{<-}[ur]&& \bullet \ar@{->}[rr] \ar@{<-}[drrr] \ar@{<-}[ur] &&\bullet \ar@{->}[rr]\ar@{<-}[ddrrr] \ar@{<-}[ur]
&&\bullet \ar@{->}[rr] \ar@{<-}[drrr] \ar@{<-}[ur] && \bullet \ar@{->}[rr] \ar@{<-}[ddrrr]\ar@{<-}[ur] &&\bullet \ar@{->}[rr] \ar@{<-}[drrr] \ar@{<-}[ur]&&  \bullet \ar@{->}[rr] \ar@{<-}[ddrrr] \ar@{<-}[ur]
&&\bullet \ar@{->}[rr] \ar@{<-}[drrr] \ar@{<-}[ur] && \bullet \ar@{->}[rr] \ar@{<-}[ddrrr] \ar@{<-}[ur]&&\bullet \ar@{->}[rr] \ar@{<-}[drrr] \ar@{<-}[ur] && \bullet \ar@{->}[rr] \ar@{<-}[ddrrr] \ar@{<-}[ur]&&
\bullet \ar@{->}[rr] \ar@{<-}[ur] && \star  \\
4&\cdots \ar@{->}[r]& \bullet \ar@{->}[rrrr] \ar@{<-}[ur]&&&&\bullet \ar@{->}[rrrr] \ar@{<-}[ur] &&&& \bullet \ar@{->}[rrrr] \ar@{<-}[ur]
&&&& \bullet \ar@{->}[rrrr] \ar@{<-}[ur] &&&&\bullet \ar@{->}[rrrr] \ar@{<-}[ur]&&&&\star \ar@{<-}[ur]&&& \\
4&\cdots \ar@{->}[rrr]&&&\bullet \ar@{->}[rrrr] \ar@{<-}[uur]&&&&\bullet \ar@{->}[rrrr] \ar@{<-}[uur]&&&&  \bullet \ar@{->}[rrrr] \ar@{<-}[uur]&&&&
\bullet \ar@{->}[rrrr] \ar@{<-}[uur]&&&& \bullet \ar@{->}[rrrr] \ar@{<-}[uur]&&&& \star \ar@{<-}[uur]& }}}
$$
\item
Type $\mathrm{G}_{2}$:
$$
\raisebox{3mm}{
\scalebox{0.60}
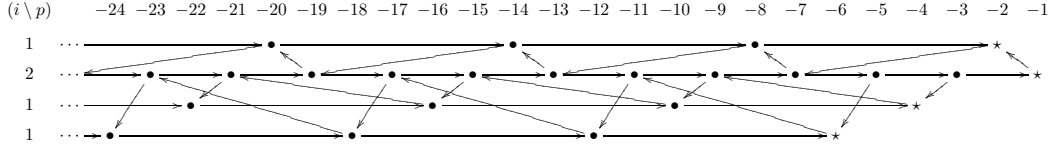
{\xymatrix@!C=0.5mm@R=2mm{
(i\setminus p) & & -24 & -23 &-22&-21 &-20& -19 &-18& -17 & -16& -15 & -14& -13& -12
& -11 & -10 & -9 & -8 & -7 & -6 & -5 & -4 & -3 & -2 & -1 \\
1&\cdots \ar@{->}[rrrrr]&&&&&\bullet \ar@{->}[rrrrrr] \ar@{<-}[dr]&&&&&&\bullet \ar@{->}[rrrrrr] \ar@{<-}[dr]&&& 
&&&\bullet \ar@{->}[rrrrrr] \ar@{<-}[dr]&&&&&&\star  \ar@{<-}[dr]&\\
2&\cdots \ar@{->}[rr] \ar@{<-}[urrrrr]&&\bullet \ar@{->}[rr] \ar@{<-}[ddrrrrr]&&\bullet \ar@{->}[rr]\ar@{<-}[drrrrr] &&
\bullet \ar@{->}[rr] \ar@{<-}[urrrrr]&&\bullet \ar@{->}[rr] \ar@{<-}[ddrrrrr]&&\bullet \ar@{->}[rr]\ar@{<-}[drrrrr] &&
\bullet \ar@{->}[rr] \ar@{<-}[urrrrr]&&\bullet \ar@{->}[rr] \ar@{<-}[ddrrrrr]&&\bullet \ar@{->}[rr]\ar@{<-}[drrrrr] &&
\bullet \ar@{->}[rr] \ar@{<-}[urrrrr]&&\bullet \ar@{->}[rr] &&\bullet \ar@{->}[rr] &&
\star \\
1&\cdots \ar@{->}[rrr]&&&\bullet \ar@{->}[rrrrrr] \ar@{<-}[ur] &&
&&&&\bullet \ar@{->}[rrrrrr] \ar@{<-}[ur] &&
&&&&\bullet \ar@{->}[rrrrrr] \ar@{<-}[ur] &&
&&&&\star \ar@{<-}[ur] &&&\\
1& \cdots \ar@{->}[r] & \bullet \ar@{->}[rrrrrr] \ar@{<-}[uur]&&&&&& 
\bullet \ar@{->}[rrrrrr] \ar@{<-}[uur]&&&&&& 
\bullet \ar@{->}[rrrrrr] \ar@{<-}[uur]&&&&&& 
\star \ar@{<-}[uur]&&&&& \\
}}}
$$
\end{itemize}
\end{Ex}

\subsection{Technical complement} \label{Ssec:Cox}
In this subsection, we often identify the automorphism $\sigma$ of $\Delta$ with the linear operator  on $\sP$ given by $\varpi_{\im} \mapsto \varpi_{\sigma(\im)}$ for any $\im \in \Delta_0$. 
Given a Q-datum $\cQ = (\Delta, \sigma, \xi)$,
for each $i \in I$, let $i^{\circ}$ denote the unique vertex in the $\sigma$-orbit $i$
satisfying $\xi_{i^{\circ}} = \max \{ \xi_{\im} \mid \im \in i\}$.
We consider the following condition on $\cQ$:
\begin{equation} \label{eq:condQ}
\text{For each $i \in I$ and $k \in [1, d_{i}-1]$, we have 
$\xi_{\sigma^{k}(i^{\circ})} = \xi_{i^{\circ}} - 2k$}.
\end{equation}
Note that this condition is always satisfied unless $\fg$ is of type $\mathrm{B}_n$ or $\mathrm{F}_4$ (we underline however
that our results and proofs below will be uniform for all types). 

\begin{Prop}[{\cite[Section~3.6]{FO21}}] \label{Prop:Cox}
There is a unique collection $\{ \tau_{\cQ}\}_{\cQ} \subset \sW \rtimes \langle \sigma \rangle$ 
labelled by Q-data for $\fg$ and characterized by the following conditions$\colon$
\begin{enumerate}
\item \label{Prop:Cox:1} If $\cQ = (\Delta, \sigma, \xi)$ satisfies \eqref{eq:condQ}, we have
$\tau_{\cQ} = s_{i_{1}^{\circ}} \cdots s_{i_{n}^{\circ}}\sigma$, where 
$(i_{1}, \ldots, i_{n})$ is any total ordering of $I$ such that $\xi_{i_{1}^{\circ}} \ge \cdots \ge \xi_{i_{n}^{\circ}}$.
\item \label{Prop:Cox:2} If $\im \in \Delta_{0}$ is a source of $\cQ$, we have
$\tau_{s_{\im}\cQ} = s_{\im} \tau_{\cQ} s_{\im}$.
\end{enumerate}  
\end{Prop}  

The element $\tau_{\cQ}$ is called the \emph{generalized Coxeter element} associated with $\cQ$.

\begin{Lem} \label{Lem:Cox}
Let $\cQ = (\Delta, \sigma, \xi)$ be a Q-datum for $\fg$, and $\bfi = (\im_u)_{u \in \N} \in \Delta_0^\N$ a sequence adapted to $\cQ$. 
Then, we have
\begin{equation} \label{eq:Cox}
w^{\bfi}_{u}\varpi_{\im_u} = \tau_{\cQ}^{d_{\bar{\im}_u}(n_\bfi(u)+1)}\varpi_{\im_u} \quad \text{for any $u \in \N$}. 
\end{equation}
\end{Lem}
\begin{proof}
Let $\bfi' = (\im'_u)_{u \in \N} \in \Delta_0^\N$ be another sequence adapted to $\cQ$. 
Thanks to Lemma~\ref{Lem:adce}, it follows that $w^\bfi_u \varpi_{\im_u} = w^{\bfi'}_{u'}\varpi_{\im'_u}$ when $\im_u = \im'_{u'}$ and $n_{\bfi}(u) = n_{\bfi'}(u')$.
Therefore, it is enough to prove the assertion for a specific choice of $\bfi$ adapted to $\cQ$.  

First we assume that $\cQ$ satisfies the condition \eqref{eq:condQ}. 
Let $I = \{i_1,\ldots, i_n \}$ be a total ordering such that $\xi_{i^\circ_1} \ge \cdots \ge \xi_{i^\circ_n}$ holds as in Proposition~\ref{Prop:Cox}~\eqref{Prop:Cox:1}. 
Then, we define a sequence $\bfi = (\im_u)_{u \in \N}$ by setting $\im_{u} \seq i_u^\circ$ for $u \in [1, n]$ and $\im_{u + n} \seq \sigma(\im_u)$ for all $u \in \N$.
It is easy to see that the sequence $\bfi$ is adapted to $\cQ$. 
For this $\bfi$, \eqref{eq:Cox} holds by a direct computation using Proposition~\ref{Prop:Cox}~\eqref{Prop:Cox:1}.

Next, we shall prove that \eqref{eq:Cox} holds for Q-datum $s_{\im}\cQ$ assuming that it holds for $\cQ$, where $\im \in \Delta_0$ be a source of $\cQ$.
Let $\bfi = (\im_u)_{u \in \N}$ be a sequence adapted to $\cQ$ such that $\im_1 = \im$.
Then the sequence $\bfi' = (\im'_u)_{u \in \N} \seq \partial_+ \bfi$ is adapted to $s_\im\cQ$.
Take any $u \in \N$ and set $\jm \seq \im'_u = \im_{u+1}$.
Under these assumptions, we have
\begin{align*} 
\tau_{s_\im \cQ}^{d_{\bar{\jm}}(n_{\bfi'}(u)+1)} \varpi_{\jm} &= s_\im \tau_{\cQ}^{d_{\bar{\jm}}(n_{\bfi}(u+1)-\delta_{\im,\jm}+1)}s_{\im}\varpi_{\jm} && \text{by Proposition~\ref{Prop:Cox}~\eqref{Prop:Cox:2}} \allowdisplaybreaks \\ 
&= s_{\im}\tau_{\cQ}^{d_{\bar{\jm}}(n_{\bfi}(u+1)+1)}\varpi_{\jm} && \text{by $s_{\im} \varpi_{\jm} = w_{1}^{\bfi}\varpi_{\jm} = \tau_\cQ^{d_{\bar{\jm}} \delta_{\im,\jm}}\varpi_{\jm}$} \allowdisplaybreaks \\ 
& = s_\im w^{\bfi}_{u+1} \varpi_{\jm}&& \text{by \eqref{eq:Cox} for $\bfi$} \allowdisplaybreaks \\ 
& = w^{\bfi'}_u \varpi_{\jm}.
\end{align*}
Thus we get \eqref{eq:Cox} for $s_\im \cQ$.

Finally, recall that every Q-datum can be obtained from one satisfying \eqref{eq:condQ} by applying source reflections (see \cite[Equation (3.11)]{FO21}).  
Therefore, we conclude that \eqref{eq:Cox} holds for any Q-datum $\cQ$ for $\fg$.
\end{proof}

\subsection{Isomorphism of quantum tori}

We establish first isomorphisms of quantum tori. 
Again, we fix a Q-datum $\cQ = (\Delta, \sigma, \xi)$ for $\fg$.

\begin{Prop} \label{Prop:L=N}
Let $\bfi \in \Delta_0^\N$ be a sequence adapted to $\cQ$ and $\Lambda_{\bfi} = (\Lambda_{u,v})_{u,v \in \N}$ the skew-symmetric matrix defined by \eqref{eq:Lambda}.
For any $u, v \in \N$, we have
\begin{equation} \label{eq:L=N}
 \Lambda_{u,v} = \Nn(m^{(\im)}[p,\xi_{\im}], m^{(\jm)}[s,\xi_{\jm}]),
\end{equation}
where $(\im,p) = \rho_{\bfi}(u)$ and $(\jm,s) = \rho_{\bfi}(v)$. 
\end{Prop}
\begin{proof}
Thanks to Proposition~\ref{Prop:comm-eq} and Lemma~\ref{Lem:adce}, it is enough to prove the assertion for a specific sequence $\bfi$ adapted to $\cQ$. 
Therefore, we may assume that our $\bfi$ is as in Example~\ref{Ex:dimin}.
Moreover, since the both sides of \eqref{eq:L=N} are skew-symmetric, we may assume that $u < v$.
Under these assumptions, we have $s \le p$. 
Applying \cite[Proposition 8.4]{FHOO}, we find that
\[ \Nn(m^{(\im)}[p,\xi_{\im}], m^{(\jm)}[s,\xi_{\jm}]) = (\varpi_{\im} - \tau_\cQ^{d_{\bar{\im}} (n_\bfi(u)+1)}\varpi_{\im}, \varpi_{\jm} + \tau_\cQ^{d_{\bar{\jm}}(n_\bfi(v)+1)}\varpi_{\jm}),\]
where $\tau_{\cQ}$ denotes the generalized Coxeter element associated with $\cQ$ (see \S\ref{Ssec:Cox} above). 
Here we remind that $p = \xi_\im - 2d_{\bar{\im}}n_{\bfi}(u)$ and $s = \xi_{\jm} - 2d_{\bar{\jm}}n_{\bfi}(v)$.  
Now, we obtain the desired equality~\eqref{eq:L=N} by Lemma~\ref{Lem:Cox}.  
\end{proof}

\begin{Cor} \label{Cor:teta}
Let $\bfi \in \Delta_0^\N$ be a sequence adapted to $\cQ$.
We have an isomorphism of $\Z[t^{\pm 1/2}]$-algebras 
\[ \teta_{\bfi} \colon \cT(\Lambda_{\bfi}) \simeq \cY_{t, \le \xi} \quad \text{given by $X_u \mapsto \ul{m^{(\im)}[p,\xi_{\im}]}$ for $u \in \N$},\]
where $(\im,p) = \rho_{\bfi}(u)$. 
In addition, we have $\teta_\bfi \circ \ol{(\cdot)} = \ol{(\cdot)} \circ \teta_\bfi$.
\end{Cor}
\begin{proof}
For each $(\im,p) \in \hD_{\le \xi}$, we have $Y_{\bar{\im},p} = m^{(\im)}[p,\xi_{\im}]/m^{(\im)}[p+2d_{\bar{\im}}, \xi_{\im}]$. 
This implies that the $\Z[t^{\pm 1/2}]$-algebra $\cY_{t, \le \xi}$ is generated by $\{\ul{m^{(\im)}[p,\xi_{\im}]}\}_{(\im,p) \in \hD_{\le \xi}}$.
Then, Proposition~\ref{Prop:L=N} asserts that the presentation of $\cY_{t,\le \xi}$ in terms these generators is identical to that of $\cT(\Lambda_\bfi)$ in \S\ref{Ssec:Ai} under the correspondence $X_u \mapsto \ul{m^{(\im)}[p,\xi_\im]}$ as in the statement. 
Thus, we have the isomorphism $\teta_\bfi$.
The compatibility with the bar involutions is obvious from the definition.
\end{proof}
We conclude this subsection by the following important observation. 
\begin{Prop}[{\cite[Lemma 4.15]{HL16}}]\label{Prop:y-hat}
Let $\bfi \in \Delta_0^\N$ be a sequence adapted to $\cQ$. For $u\in \N$, we set $\bfe_u=(\delta_{ku})_{k\in \N}\in \Z^{\oplus \N}$, and 
\[
\bfb_u=\tB_{\bfi}\bfe_u=\sum_{k \in \N} b_{k,u}\bfe_{k}. 
\]
Then we have 
\[
\teta_{\bfi}(X^{\bfb_u})=\ul{A_{i, p-d_i}^{-1}} 
\]
where $(i,p) = \bar{\rho}_\bfi(u)$. 
\end{Prop}
\begin{proof}
    We know that $\teta_{\bfi}(X^{\bfb_u})=t^{a/2}\ul{A_{i, p-d_i}^{-1}}$ for some $a\in \Z$ by Lemma \ref{Lem:quiverisom} and the observation similar to the proof of \cite[Lemma 4.15]{HL16}. We can conclude that $a=0$ since $\teta_{\bfi}$ is compatible with the bar involutions on $\cT(\Lambda_{\bfi})$ and $\cY_{t, \le \xi}$. 
\end{proof}
\subsection{Cluster structure on $\cK_t(\Cc_{\le \xi})$}

We prove the main result of this section: $\cK_t(\Cc_{\le \xi})$ has a quantum cluster algebra structure. Using the isomorphism $\teta_\bfi$ of quantum tori in Corollary \ref{Cor:teta}, we use the approach in 
\cite{HL16, Bit}: to identify the respective images in the quantum tori, for the first inclusion we prove that the generators  $F_t(Y_{i,p})$ correspond to some quantum cluster variables 
(by using quantum $T$-systems) and then we show a stability 
property of the quantum Grothendieck ring by quantum cluster 
mutation for the other inclusion.

 Recall one of the main results of \cite{HL16}.

\begin{Thm}[{\cite[Theorem 5.1]{HL16}}] The Grothendieck ring $K(\Cc^-)$ has a cluster algebra structure
so that the classes of Kirillov-Reshetikhin modules in $K(\Cc^-)$ are cluster variables.
\end{Thm}

The proof of the following Theorem is a generalization of the proof of \cite[Theorem 5.1]{HL16} 
above, and that of \cite[Theorem 5.2.4]{Bit} for the quantum Grothendieck ring $\cK_t(\Cc^-)$ 
of type $\mathrm{ADE}$.
Here $\Cc^-$ is a special example of the category $\Cc_{\le \xi}$ with $\xi$ satisfying the condition 
$-2d_{\bar{\im}} < \xi_\im \le 0$ for all $\im \in \Delta_0$ (cf.~Remark~\ref{Rem:HL}).

\begin{Thm} \label{Thm:qcl}
Let $\cQ = (\Delta, \sigma, \xi)$ be a Q-datum for $\fg$ and $\bfi = (\im_u)_{u \in \N} \in \Delta_0^\N$ a sequence adapted to $\cQ$.  
There is a unique isomorphism of $\Z[t^{\pm 1/2}]$-algebras
$ \eta_{\bfi} \colon \cA_\bfi \simeq \cK_t(\Cc_{\le \xi})$
which makes the following diagram commute:
\begin{equation} \label{eq:diagT0}
\vcenter{
\xymatrix{
\cA_{\bfi} \ar[r]^-{\eta_{\bfi}} \ar@{^{(}->}[d] & \cK_t(\Cc_{\le \xi}) \ar@{^{(}->}[d]^-{(\cdot)_{\le \xi}}\\
\cT(\Lambda_\bfi) \ar[r]^-{\teta_\bfi}& \cY_{t, \le \xi}.
}}
\end{equation} 
Moreover, when $m$ is a dominant monomial of a KR module in $\Cc_{\le \xi}$, the element $F_t(m)$ corresponds to a cluster variable of $\cA_\bfi$ under the isomorphism $\eta_\bfi$. 
\end{Thm}
\begin{proof}
%Our proof is a generalization of that of \cite[Theorem 5.1]{HL16} for the usual Grothendieck ring $K(\Cc^-)$, and that of \cite[Theorem 5.2.4]{Bit} for the quantum Grothendieck ring $\cK_t(\Cc^-)$ of type $\mathrm{ADE}$.
%Here $\Cc^-$ is a special example of the category $\Cc_{\le \xi}$ with $\xi$ satisfying the condition $-2d_{\bar{\im}} < \xi_\im \le 0$ for all $\im \in \Delta_0$ (cf.~Remark~\ref{Rem:HL}).

Let $\cK_{t, \le \xi} \subset \cY_{t, \le \xi}$ be the image of $\cK_t(\Cc_{\le \xi})$ under the truncation map $(\cdot)_{\le \xi}$.
To establish the isomorphism $\eta_\bfi$, it is enough to show that $\teta_{\bfi}(\cA_{\bfi}) = \cK_{t, \le \xi}$.

First, we shall show $\teta_{\bfi}(\cA_{\bfi}) \supset \cK_{t, \le \xi}$.
Since $\cK_t(\Cc_{\le \xi})$ is generated by the elements $F_t(Y_{i,p})$, it is enough to prove that $F_t(Y_{i,p})_{\le \xi} \in \teta_\bfi(\cA_\bfi)$ for all $(i,p) \in \hI_{\le \xi}$.
It follows from the following more general assertion: for any $k \in \N_0$, and $u \in \N$, we have 
\begin{equation} \label{eq:claim}  
\teta_\bfi (\partial_+^*)^k X_u = F_t^{(\im)}[p, (s_{\im_k}\cdots s_{\im_1}\xi)_\im]_{\le \xi} \quad \text{if $\rho_{\partial_+^k \bfi}(u) = (\im,p)$}, 
\end{equation}
where $(\partial_+^*)^k$ denotes the composition of the homomorphisms 
\[ \cA_{\partial_+^k \bfi} \xrightarrow{\partial_+^*} \cA_{\partial_+^{k-1} \bfi} \xrightarrow{\partial_+^*} \cdots \xrightarrow{\partial_+^*} \cA_{\partial_+ \bfi} \xrightarrow{\partial_+^*} \cA_{\bfi}  \]
and $X_u$ denotes the $u$-th initial cluster variable of $\cA_{\partial_+^k \bfi}$.
Once we have the assertion~\ref{eq:claim}, it follows that $\teta_\bfi^{-1}F_t(m)_{\le \xi}$ is a cluster variable for all dominant monomial $m$ corresponding to a KR module in $\Cc_{\le \xi}$, since every $F_t^{(\im)}[p,s]_{\le \xi}$ (with $(\im,p), (\im, s) \in \hI_{\le \xi}$, $p \le s$) appears in the RHS of the equality \eqref{eq:claim} when $k$ and $u$ vary.
This proves the last assertion in the statement.

Let us outline the proof of \eqref{eq:claim}, which is the same as those of previous works \cite{HL16, Bit}.
It proceeds by transfinite induction on $(k,u) \in \N_0 \times \N$ along the lexicographic order.  
When $k=0$, \eqref{eq:claim} holds for all $u \in \N$ by Lemma~\ref{Lem:qKR} and the definition of $\teta_\bfi$ in Corollary~\ref{Cor:teta}.
Let us discuss the case when $k =1$.  
We set $\im \seq \im_1$ and recall the notation $\mu_+$ and $\sigma_+$ which appeared in the definition of $\partial_+^*$ in \S\ref{fshif}.
For each $u \in \N$, let $X'_u \seq \mu_+^* X_u$ be the mutated cluster variable so that we have $\partial_+^* X_{\sigma_+(u)} = X'_u$.
Note that $X'_u = X_u$ if $\im_u \neq \im$. 
For any $v \in \N$ with $\im_v = \im$, we have the exchange relation 
\begin{equation} \label{eq:exXu}
X'_v X_v = t^a X'_{v_\bfi^-} X_{v_\bfi^+} + t^b \prod^{\to}_{\jm \sim \im}X_{v_\bfi^+(\jm)}
\end{equation}
for some $a, b \in \frac{1}{2}\Z$ (see the proof of Lemma~\ref{Lem:shiftB} and Lemma~\ref{Lem:mu}).
Now, we want to show that $\teta_\bfi(X'_v) = F_t^{(\im)}[p, \xi_\im)_{\le \xi}$ if $\rho_{\bfi}(v) = (\im, p+2d_{\bar{\im}})$, which is equivalent to \eqref{eq:claim} for $(k,u) = (1,\sigma_+(v))$ (note that $\rho_{\partial_+\bfi}(\sigma_+(v)) = (\im, p)$ in this case).
By induction, we assume that $\teta_{\bfi}(X'_{v_\bfi^-}) = F_t^{(\im)}(p,\xi_{\im})_{\le \xi}$.
Then, from \eqref{eq:exXu}, we deduce that $\teta_\bfi(X'_v)$ is the unique bar invariant element of the form
\[ \left(t^{a}F_{t}^{(\im)}[p,\xi_\im]_{\le \xi} F_{t}^{(\im)}(p,\xi_\im)_{\le \xi} + 
t^{b} \prod_{\jm \sim \im}^{\to} F_{t}^{(\jm)}(p,\xi_\im)_{\le \xi}\right) \left(F_{t}^{(\im)}(p,\xi_\im]_{\le \xi}\right)^{-1}.\]
On the other hand, the truncation of the quantum $T$-system equation~\eqref{eq:qTsys} tells us that the same property also characterizes the element $F_t^{(\im)}[p,\xi_\im)_{\le \xi}$.
Therefore, we obtain the equality $\teta_\bfi(X'_v) = F_t^{(\im)}[p, \xi_\im)_{\le \xi}$, which completes the proof of \eqref{eq:claim} for $k=1$.
The proof for the case $k > 1$ is similar and hence we omit it.

Next, we shall prove the opposite inclusion $\teta_{\bfi}(\cA_{\bfi}) \subset \cK_{t, \le \xi}$.
Take a reduced word $(\jm_1, \ldots, \jm_\ell)$ of the longest element $w_\circ$ adapted to the Q-datum $\fD \cQ$.
Then we define a sequence $\bfi' = (\im'_u)_{u \in \N}$ so that $\im'_u = \jm_u$ for $1 \le u \le \ell$ and $\im'_{u+\ell} = \im_u$ for any $u \in \N$.
By Proposition~\ref{Prop:DQ}, $\bfi'$ is adapted to $\fD \cQ$, and obviously we have $\partial_+^\ell \bfi' = \bfi$. 
Consider the following diagram:
\begin{equation} \label{eq:diagD}
\vcenter{
\xymatrix{
\cA_{\bfi'} \ar[r]^-{\teta_{\bfi'}} & \cY_{t, \le \fD\xi} \ar@{->}[d]^-{(\cdot)_{\le \xi}}\\
\cA_\bfi \ar[r]^-{\teta_\bfi} \ar[u]^{(\partial_+^*)^\ell}& \cY_{t, \le \xi}.
}}
\end{equation} 
A priori, it is not clear whether \eqref{eq:diagD} commutes (but we will see that actually it does).
Letting $\zeta \seq \teta_{\bfi'} \circ (\partial_+^*)^\ell$, we shall prove that $\zeta(x)\in \cK_t(\Cc_{\le \xi})$ and $\zeta(x)_{
\le \xi} = \teta_{\bfi}(x)$ for any quantum cluster variable $x \in \cA_\bfi$ by induction on the distance from the initial cluster.
When $x$ is an initial quantum cluster variable, it follows from~\eqref{eq:claim} (with $k=\ell$), Proposition~\ref{Prop:DQ} and Lemma~\ref{Lem:tr}.
Let $x$ be a quantum cluster variable obtained by the exchange relation $x x' = y_1 + y_2$, where $x'$ is another quantum cluster variable and $y_1, y_2$ are quantum cluster monomials multiplied by some powers of $t^{\pm 1/2}$. 
By induction, we assume that $\zeta(z) \in \cK_{t}(\Cc_{\le \xi})$ and $\zeta(z)_{\le \xi} =\teta_\bfi(z)$ for $z \in \{x', y_1, y_2\}$. 
Since $\zeta$ is an algebra homomorphism, we have 
\begin{equation} \label{eq:zeta} 
\zeta(x)\zeta(x') = \zeta(y_1) + \zeta(y_2) \in \cK_t(\Cc_{\le \xi}).
\end{equation}
Then, Lemmas~\ref{Lem:scr} \& \ref{Lem:tr} imply that $\zeta(x) \in \cK_t(\Cc_{\le \xi})$.
Recall that the restriction of $(\cdot)_{\le \xi}$ to $\cK_t(\Cc_{\le \xi})$ is an algebra homomorphism. 
We apply it to the equation \eqref{eq:zeta} to find
\[ \zeta(x)_{\le \xi} \teta_\bfi(x') = \teta_\bfi(y_1) + \teta_\bfi(y_2) = \teta_\bfi(x) \teta_\bfi(x'),\]
which implies that $\zeta(x)_{\le \xi} = \teta_\bfi(x)$.
Thus, we have proved that $\teta_\bfi(x) \in \cK_{t, \le \xi}$ for any quantum cluster variable $x \in \cA_\bfi$.  
\end{proof}

For future use, we remark the following. 

\begin{Lem} \label{Lem:degf}
Let $\cQ = (\Delta, \sigma, \xi)$ be a Q-datum for $\fg$ and $\bfi \in \Delta_0^\N$ a sequence adapted to $\cQ$.
For each $(i,p) \in \hI_{\le \xi}$, the element $\eta_\bfi^{-1}F_t(Y_{i,p})$ is a quantum cluster variable whose degree is $\bfe_u - \bfe_{u^-_\bfi}$, where $\bar{\rho}_\bfi(u) = (i,p)$.  
\end{Lem}
\begin{proof}
By~\eqref{eq:claim}, we have $\eta_\bfi^{-1}F_t(Y_{i,p}) = (\partial_+^*)^{u-1}X_1$. 
Therefore, the assertion follows from Lemma~\ref{Lem:dg}. 
\end{proof}

\subsection{Relation to the HLO isomorphisms}
\label{HLO}
Let $\bfi = (\im_1, \ldots, \im_\ell)$ be a reduced word for the longest element adapted to $\cQ$ (recall Proposition~\ref{Prop:DQ}).
We extend it to be an infinite sequence $\tilde{\bfi} = (\im_{u})_{u \in \N}$ adapted to $\cQ$. 
Such an extension always exists, see Example~\ref{Ex:periodic}.
From the discussion in \cite[\S8.2]{FHOO}, we can deduce that the isomorphism  $\eta_{\tilde{\bfi}} \colon \cA_{\tilde{\bfi}} \simeq \cK_t(\Cc_{\le \xi})$ in Theorem~\ref{Thm:qcl} restricted to $\cA_t(\tB_\bfi, \Lambda_\bfi)$ yields the isomorphism of $\Z[t^{\pm 1/2}]$-subalgebras 
\[ \eta_{\bfi} \colon \cA_{\tilde{\bfi}}^\ell = \cA_t(\tB_\bfi, \Lambda_\bfi) \simeq \cK_{t}(\Cc_\cQ),\]
which does not depend on the choice of extension $\tilde{\bfi}$ of $\bfi$.
Moreover, the composition 
\[ \Phi_\cQ = \eta_\bfi \circ \varphi_\bfi^{-1} \colon \cA_t[N_-] \simeq \cA_t(\tB_\bfi, \Lambda_\bfi) \simeq \cK_t(\Cc_\cQ) \]
only depends on the Q-datum $\cQ$ (that is, it does not depend on the reduced word $\bfi$ adapted to $\cQ$, see Propositions~\ref{Prop:comm-eq}, \ref{Prop:tau}, \& \ref{Prop:DQ}) and it is identical to the isomorphism in \cite[\S6]{HL15}, \cite[\S10]{HO19}, \cite[\S8.2]{FHOO}, which is called the \emph{HLO isomorphism} in \cite{FHOO}.
%Here, we remind the following property of $\Phi_\cQ$.

%\begin{Thm}[{\cite[Corollary 8.7]{FHOO}}]
%The HLO isomorphism $\Phi_\cQ$ induces a bijection between the canonical bases $\tbfB$ and $\bfL_{t, \cQ}$.
%\end{Thm}

\section{Isomorphisms among quantum Grothendieck rings and applications}\label{secsix}

We give in Theorem \ref{Thm:Psi=tau} a cluster theoreticalal interpretation of the isomorphisms among the quantum Grothendieck rings constructed in \cite{FHOO}, together with their canonical basis (Corollary \ref{Cor:Psi}). This leads to a quantum version of the monoidal categorification Theorem of \cite{KKOP2} for the categories $\Cc_{\le \xi}$ (Theorem  \ref{Thm:main}): the quantum cluster monomials belong to the canonical basis. We obtain several applications, 
including the proof of the positivity conjecture of $(q,t)$-characters
 (Corollary \ref{Cor:pos}) and the proof of the Kazhdan--Lusztig conjecture for reachable modules (Corollary \ref{Cor:KLr}). 
% and a $t$-analog of Hernandez--Leclerc's geometric character formula.

\subsection{Isomorphisms among quantum Grothendieck rings}\label{Ssec:FHOO}

We recall the isomorphisms among the quantum Grothendieck rings constructed in \cite{FHOO}.

Let $\cQ = (\Delta, \sigma, \xi)$ be a Q-datum for $\fg$.
Take another complex simple Lie algebra $\fg'$ and a Q-datum $\cQ' = (\Delta', \sigma', \xi')$ for $\fg'$.
We assume $\Delta \simeq \Delta'$, that is, we assume either that $\fg$ and $\fg'$ are related by (un)folding, or that $\fg \simeq \fg'$ holds.   
In what follows, we identify $\Delta$ with $\Delta'$, and in particular we have $\Delta_0 = \Delta'_0$.
To avoid a possible confusion, we often denote a mathematical object $X$ by $X'$ when it is associated with $\fg'$. 
For example, we denote by $\Cc'$ the category of finite-dimensional $U_q(L\fg')$-modules of type $1$, and by $\fD^{\prime \pm1}$ the duality functors on $\Cc'$.
In \cite{FHOO}, we proved the following. 

\begin{Thm}[{\cite[\S10.3]{FHOO}}] \label{Thm:FHOO}
With the above assumption, there exists a unique isomorphism of $\Z[t^{\pm 1/2}]$-algebras 
\[ \Psi \equiv \Psi(\cQ, \cQ') \colon \cK_t(\Cc'_{\Z}) \simeq \cK_t(\Cc_{\Z})\]
satisfying the following properties:
\begin{enumerate} 
\item \label{FHOO:rest} 
restricted to the subalgebra $\cK_t(\Cc'_{\cQ'})$, it coincides with the isomorphism
\[ \Phi_{\cQ} \circ \Phi_{\cQ'}^{-1} \colon \cK_t(\Cc'_{\cQ'}) \simeq \cA_t[N_-] \simeq \cK_t(\Cc_{\cQ}),\]
\item \label{FHOO:D}
we have $\Psi \circ \fD_t^{\prime \pm 1} = \fD_t^{\pm 1}\circ\Psi$.
\end{enumerate}
Moreover, the isomorphism $\Psi$ induces a bijection between the $(q, t)$-characters of simple modules.
\end{Thm}

Note that the isomorphism $\Psi$ restricts to the isomorphism
\[  \cK_t(\Cc'_{\le \fD^{\prime k}\xi'}) \simeq \cK_t(\Cc_{\le \fD^{k}\xi}) \]
which induces a bijection between the canonical bases $\bfL'_{t, \le \fD^{\prime k}\xi'}$ and $\bfL_{t, \le \fD^{k}\xi}$ for all $k\in \Z$.

We conclude this section by mentioning a related work by Kashiwara, Kim, Park and the third named author. Though this observation will not be used in this paper, it is of interest itself. 
%\begin{Rem} \label{Rem:braid}
%When $\fg$ is of type $\mathrm{ADE}$,
In \cite{KKOP21}, a collection of automorphisms $\{ \sigma_\im \}_{\im \in \Delta_0}$ is constructed on the localized quantum Grothendieck ring $\cK_t(\Cc_\Z)\otimes_{\Z[t^{\pm 1/2}]} \Q(t^{1/2})$, which satisfies the braid group relations.     
When $\fg = \fg'$ and $\cQ = s_{\im}\cQ'$ with $\im \in \Delta_0$ being a source of $\cQ'$, it is easy to see that the isomorphism $\Psi(\cQ, \cQ')$ in Theorem~\ref{Thm:FHOO} is identical to the automorphism $\sigma_{\im}$ after the localization.
In particular, we obtain the following. 

\begin{Prop}\label{Prop:braid}
The braid group symmetry given by the automorphisms $\{ \sigma_\im \}_{\im \in \Delta_0}$ in \cite{KKOP21} respects the canonical basis of $\cK_t(\Cc_\Z)$.
\end{Prop}
%\end{Rem}

\subsection{Cluster theoretical interpretation of the isomorphisms}\label{Ssec:CTinter}

Let us choose an infinite sequence $\bfi = (\im_u)_{u \in \N}$ (resp.~$\bfi' = (\im'_u)_{u \in \N}$) satisfying the condition~\eqref{eq:cond2} in Example~\ref{Ex:periodic} with respect to $\cQ$ (resp.~$\cQ'$).  
Recall the notation $\gamma_k$ and $\beta_k$ from \S\ref{cmov} and \S\ref{bmov} respectively.
We choose and fix a finite sequence $\bftau = (\tau_1, \ldots, \tau_l)$ in $\{ \gamma_{1}, \ldots, \gamma_{\ell -1} \}\cup \{ \beta_1, \ldots, \beta_{\ell-2}\}$ such that 
\[(\im'_1, \ldots. \im'_\ell) = \tau_1 \cdots \tau_l (\im_1, \ldots, \im_\ell).\] 
%By the skew-periodicity, we have $\bfi' = \hat{\tau}_1 \cdots \hat{\tau}_l \bfi$, where we put 
%\[ \hat{\gamma}_k \seq \prod_{n=0}^{\infty} \gamma_{k+n\ell}, \qquad  \hat{\beta}_k \seq \prod_{n=0}^{\infty} \beta_{k+n\ell}. \]
For each $n \in \N_0$, we have the $\Z[t^{\pm1/2}]
$-algebra isomorphism 
\[ \hat{\bftau}^{(n)*} \seq \hat{\tau_l}^{(n)*} \cdots \hat{\tau}_1^{(n)*} \colon \cA_{\bfi'}^{n\ell} \to \cA_{\bfi}^{n\ell},\]
where we set  
\[
\hat{\gamma}_k^{(n)*} \seq \gamma_{k}^* \gamma_{k +\ell}^* \cdots \gamma_{k+(n-1)\ell}^*,
\qquad 
\hat{\beta}_k^{(n)*} \seq \beta_{k}^* \beta_{k +\ell}^* \cdots \beta_{k+(n-1)\ell}^*.
\] 
Note that $\varepsilon_1^* \varepsilon_2^* = \varepsilon_2^*\varepsilon_1^*$ whenever $\varepsilon_i \in \{\gamma_{k+n_i\ell}\}_{k \in [1,\ell-1]} \cup \{\beta_{k+n_i\ell}\}_{k \in [1,\ell-2]}$ and $n_1 \neq n_2$.
Recall that we have $\cA_\bfi = \bigcup_n \cA_\bfi^{n\ell}$ and $\cA_{\bfi'} = \bigcup_n\cA_{\bfi'}^{n\ell}$.
Taking the inductive limit, we obtain the $\Z[t^{\pm1/2}]
$-algebra isomorphism
\[ \hat{\bftau}^* \seq \lim_{n \to \infty}\hat{\bftau}^{(n)*} \colon \cA_{\bfi'}\to \cA_\bfi. \] 
By construction, the isomorphism $\hat{\bftau}^*$ induces a bijection between the sets of quantum cluster monomials.

The following is one of the main theorems of this section and gives a cluster theoretical interpretation of the isomorphism $\Psi$. Our proof has two steps: we study first the result for subcategories $\Cc_\cQ$, and then we extend it to the categories $\Cc_{\le \xi}$ using the dualities $\fD_t$.

\begin{Thm} \label{Thm:Psi=tau}
In the situation described above, the following diagram commutes:
\begin{equation} \label{eq:diagT}
\vcenter{
\xymatrix{
\cA_{\bfi'} \ar[r]^-{\eta_{\bfi'}} \ar[d]_-{\hat{\bftau}^*}& \cK_t(\Cc'_{\le \xi'}) \ar[d]^-{\Psi}\\
\cA_{\bfi} \ar[r]^-{\eta_{\bfi}}& \cK_t(\Cc_{\le \xi}).
}}
\end{equation}
\end{Thm}
\begin{proof}
Let $\bT \seq \eta_{\bfi} \circ \hat{\bftau}^* \circ \eta_{\bfi'}^{-1}$. 
We have to show $\bT = \Psi$.
By the characterization of the isomorphism $\Psi$ in Theorem~\ref{Thm:FHOO}, it suffices to prove 
\begin{enumerate}
\item \label{claim1} $\bT = \Psi$ after restricted to $\cK_t(\Cc'_{\cQ'})$, and 
\item \label{claim2} $\bT \circ \fD_t^{\prime-1} = \fD_t^{-1} \circ \bT$.
\end{enumerate}  
Restricting \eqref{eq:diagT} to $\cA_{\bfi'}^{\ell} \simeq \cA_t(\tB_{\bfi'}, \Lambda_{\bfi'})$, we obtain
\begin{equation} \label{eq:diagtau} 
\vcenter{
\xymatrix{
\cA_t(\tB_{\bfi'}, \Lambda_{\bfi'}) \ar[rr]^-{\eta_{\bfi'}} \ar[dd]_-{{\bftau}^*} \ar[dr]^-{\varphi_{\bfi'}} & & \cK_t(\Cc'_{\cQ'})  \ar[dd]^-{\Psi} \\
& \cA_t[N_-] \ar[ru]^-{\Phi_{\cQ'}} \ar[rd]_-{\Phi_{\cQ}} & \\
\cA_t(\tB_{\bfi}, \Lambda_{\bfi}) \ar[ru]_-{\varphi_{\bfi}} \ar[rr]_-{\eta_{\bfi}} & & \cK_t(\Cc_\cQ). 
}}
\end{equation}
We know that all of the 4 small triangles in the diagram \eqref{eq:diagtau} commute thanks to Corollary~\ref{Cor:bftau}, the discussion in \S\ref{HLO}, and Theorem~\ref{Thm:FHOO}~\eqref{FHOO:rest}.
Thus, the large square in \eqref{eq:diagtau} also commutes, which proves \eqref{claim1}.

Next, we shall prove \eqref{claim2}.
Note that $\bfi = \partial_+^\ell \bfi^*$ with $\bfi^* = (\im_u^*)_{u \in \N}$.
Since $\tB_\bfi = \tB_{\bfi^*}$ and $\Lambda_\bfi = \Lambda_{\bfi^*}$, there is the obvious isomorphism $\nu \colon \cA_{\bfi^*} \simeq \cA_{\bfi}$ identifying $X_u$ for all $u \in \N$.
Then, we have the identity
\begin{equation} \label{eq:D=dl}
\eta_\bfi \circ \nu \circ (\partial^*_+)^\ell = \fD_t^{-1} \circ \eta_\bfi
\end{equation} 
which relates the homomorphisms $(\partial^*_+)^\ell \colon \cA_{\bfi} \to \cA_{\bfi^*}$ and $\fD_t^{-1} \colon \cK_t(\Cc_{\le \xi}) \to \cK_t(\Cc_{\le \xi})$.
Indeed, by Lemmas~\ref{Lem:dg} \& \ref{Lem:degf} and Proposition~\ref{Prop:DQ}, we see that the cluster variables
\[\eta_{\bfi}^{-1} \fD_t^{-1} F_t(Y_{i,p})\quad \text{and} \quad \nu(\partial_+^*)^\ell \eta_{\bfi}^{-1} F_t(Y_{i,p})\] 
share the same degree, and hence they coincide for all $(i,p) \in \hI_{\le \xi}$ (see Theorem~\ref{Thm:gvsx}). 
Since the elements $\{F_t(Y_{i,p})\}_{(i,p) \in \hI_{\le \xi}}$ generate the $\Z[t^{\pm 1/2}]$-algebra $\cK_t(\Cc_{\le \xi})$, we get \eqref{eq:D=dl}.
In the same way, we have the analogous identity for $\bfi'$. 
In addition, for any cluster monomial $x \in \cA_{\bfi'}$ whose degree belongs to the cone $C_{\bfi'}$ (see~\eqref{eq:cone} for its definition), we have  
\begin{equation} \label{eq:taudl}
\hat{\bftau}^{*} \nu(\partial_+^*)^\ell x = \nu(\partial_+^*)^\ell \hat{\bftau}^* x. 
\end{equation}
Indeed, the degrees of both sides of \eqref{eq:taudl} coincide by Lemmas~\ref{Lem:dg} \& \ref{Lem:cone} and the definition of $\hat{\bftau}^*$. 
Now, we have
\begin{align*}
\bT \fD_t^{-1} F_t(Y_{i,p}) &=\eta_{\bfi}\hat{\bftau}^* \eta_{\bfi'}^{-1}\fD_t^{\prime -1} F_t(Y_{i,p}) && \allowdisplaybreaks \\
&=  \eta_{\bfi}\hat{\bftau}^* \nu(\partial_+^*)^\ell \eta_{\bfi'}^{-1} F_t(Y_{i,p}) && \text{by \eqref{eq:D=dl} for $\bfi'$} \allowdisplaybreaks \\
&= \eta_{\bfi}\nu(\partial_+^*)^\ell \hat{\bftau}^* \eta_{\bfi'}^{-1} F_t(Y_{i,p}) && \text{by \eqref{eq:taudl}} \allowdisplaybreaks \\
&= 
\fD_t^{-1} \eta_{\bfi} \hat{\bftau}^* \eta_{\bfi'}^{-1} F_t(Y_{i,p}) && \text{by \eqref{eq:D=dl} for $\bfi$} \allowdisplaybreaks \\
&= \fD_t^{-1} \bT F_t(Y_{i,p})
\end{align*}
for any $(i,p) \in \hI_{\le \xi'}$. 
Since the elements $\{F_t(Y_{i,p})\}_{(i,p) \in \hI_{\le \xi}}$ generate the $\Z[t^{\pm 1/2}]$-algebra $\cK_t(\Cc_{\le \xi})$, we obtain the claim \eqref{claim2}. 
\end{proof}

\begin{Cor} \label{Cor:Psi}
Let $\cQ = (\Delta, \sigma, \xi)$ and $\cQ' = (\Delta', \sigma', \xi')$ be two Q-data such that $\Delta = \Delta'$ as above.
For any sequences $\bfi, \bfi' \in \Delta_0^\N$ adapted to $\cQ$ and $\cQ'$ respectively, the $\Z[t^{\pm 1/2}]$-algebra isomorphism 
$\eta_\bfi^{-1} \circ \Psi(\cQ, \cQ') \circ \eta_{\bfi'}$ induces a bijection between the set of quantum cluster monomials in $\cA_{\bfi'}$ and the set of quantum cluster monomials in $\cA_\bfi$.
\end{Cor}
\begin{proof}
In view of Proposition~\ref{Prop:comm-eq}, Lemma~\ref{Lem:adce} and Example~\ref{Ex:periodic}, it is enough to prove the assertion when both $\bfi$ and $\bfi'$ are satisfying the condition~\eqref{eq:cond2}.
In this case, it is immediate from Theorem~\ref{Thm:Psi=tau}.
\end{proof}

\subsection{Quantization of monoidal categorification theorem}

Let $\cQ = (\Delta, \sigma, \xi)$ be a Q-datum for $\fg$, and $\bfi \in \Delta_0^\N$ a sequence adapted to $\cQ$.
Let $\bar{\cA}_\bfi \seq\bar{\cA}(\tB_\bfi)$ denote the (classical) cluster algebra  associated with $\tB_\bfi$, which comes with the evaluation map $\evt \colon \cA_\bfi \to \bar{\cA}_\bfi$.   
Specializing $t$ to $1$, in Theorem~\ref{Thm:qcl}, we obtain the isomorphism of commutative algebras $\bar{\eta}_{\bfi} \colon \bar{\cA}_\bfi \to K(\Cc_{\le \xi})$. 
For example, this is the isomorphism of \cite{HL16} for the category $\Cc^-$. 
Precisely, we define $\bar{\eta}_\bfi$ to be the unique isomorphism which makes the following diagram commute:
\begin{equation} \label{eq:comm-ev}
\vcenter{
\xymatrix{
\cA_\bfi \ar[r]^-{\eta_{\bfi}} \ar[d]_-{\evt}& \cK_t(\Cc_{\le \xi})  \ar@{^{(}->}[r]& \cY_t \ar[d]^-{\evt}\\
\bar{\cA}_\bfi \ar[r]^-{\bar{\eta}_\bfi}& K(\Cc_{\le \xi}) \ar@{^{(}->}[r]^-{\chi_q}& \cY.
}}
\end{equation}

Here, we remind the important monoidal categorification theorem established by Kashiwara, Kim, Park, and the third named author.

\begin{Thm}[{\cite[\S8]{KKOP2}}]
The isomorphism $\bar{\eta}_\bfi$ sends each cluster monomial of the cluster algebra $\bar{\cA}_\bfi$ to the class of a simple object of the category $\Cc_{\le \xi}$.
\end{Thm}

Now, we shall prove a quantum analog of this result.  

\begin{Thm} \label{Thm:main}
The isomorphism $\eta_\bfi$ sends each quantum cluster monomial of the quantum cluster algebra $\cA_\bfi$ to an element of the canonical basis $\bfL_{t, \le \xi}$ of the quantum Grothendieck ring $\cK_t(\Cc_{\le \xi})$. 
More precisely, we have $\eta_\bfi(x) = L_t(m)$ for any quantum cluster monomial $x \in \cA_\bfi$, where $m \in \cM_{\le \xi}$ is the unique dominant monomial such that $\bar{\eta}_\bfi(\evt(x)) = [L(m)]$.
\end{Thm}

\begin{proof} Our proof has two main steps: first we prove the result for simply-laced types, and then we use a (un)folding argument for general types which is based on the new ingredients discussed above. As the first step also works directly for type $\mathrm{B}$, it is also included.

So first, we discuss the case when $\fg$ is of type $\mathrm{ABDE}$.
In this case, we know that Conjectures~\ref{Conj:KL} \& \ref{Conj:KRF} are true.
In particular, it follows that, for any $m_1,m_2 \in \cM$, the tensor product module $L(m_1)\otimes L(m_2)$ is simple if and only if we have 
\[L_t(m_1) L_t(m_2) = t^a L_t(m_1m_2)\] 
in $\cK_t(\Cc_\Z)$ for some $a \in \frac{1}{2}\Z$ (see \cite[Corollary 5.5]{HL15}, \cite[Lemma 11.5]{FHOO}).
Thanks to this fact, it is enough to show that $\eta_\bfi(x) \in \bfL_{t, \le\xi}$ for every quantum cluster \emph{variable} $x \in \cA_\bfi$ (rather than \emph{monomial}).
We prove this by induction on the distance from the initial cluster.   
When $x$ is an initial quantum cluster variable, it follows because Conjecture~\ref{Conj:KRF} is true for type $\mathrm{ABDE}$.
Assume that $x$ is a quantum cluster variable obtained by the exchange relation \begin{equation} \label{eq:exx}
x x' = t^{a_1}y_1 + t^{a_2}y_2,
\end{equation} where $x'$ is another quantum cluster variable, $y_1, y_2$ are quantum cluster monomials, and $a_1, a_2 \in \frac{1}{2}\Z$. 
Let $m,m',m_1,m_2 \in \cM_{\le \xi}$ be the unique dominant monomials satisfying
\[ \bar{\eta}_\bfi(\evt(x)) = [L(m)], \quad \bar{\eta}_\bfi(\evt(x')) = [L(m')], \quad \bar{\eta}_\bfi(\evt(y_k)) = [L(m_k)] \]
for $k = 1,2$.
We note that the relation~\eqref{eq:exx} goes to 
\begin{equation} \label{eq:exx1}
\chi_q(L(m)) \chi_q(L(m')) = \chi_q(L(m_1)) + \chi_q(L(m_2))
\end{equation}
under the homomorphism $\chi_q \circ \bar{\eta}_\bfi\circ \evt$.
By induction, we assume that
\[ \eta_\bfi(x') = L_t(m'), \quad \eta_\bfi(y_k) = L_t(m_k)\]  
for $k = 1,2$. Then, by the relation \eqref{eq:exx}, we see that the element $\eta_\bfi(x)$ is characterized as the unique bar invariant element of the form
\begin{equation} \label{eq:charL}
\left(t^{a_1}L_t(m_1) + t^{a_2}L_t(m_2)\right) L_t(m')^{-1} \quad \text{for some $a_1, a_2 \in \textstyle \frac{1}{2}\Z$}
\end{equation}
in the fraction field $\F(\cY_t)$.  
On the other hand, in $\cK_t(\Cc_{\le \xi})$, we have
\begin{equation} \label{eq:LL}
L_t(m) L_t(m') = \sum_{m'' \in \cM_{\le \xi}} c_{m''}(t) L_t(m'') 
\end{equation}
with some $c_{m''}(t) \in \N_0[t^{\pm 1/2}]$ for $m'' \in \cM_{\le \xi}$ by Theorem~\ref{Thm:pos}.
Since Conjecture~\ref{Conj:KL} is verified for type $\mathrm{ABDE}$, we can compare the specialization of \eqref{eq:LL} at $t=1$ with the relation~\eqref{eq:exx1}.
Then, the positivity forces that
\[ c_{m''}(t) = \begin{cases}
t^{b_k} \text{ for some $b_k \in \frac{1}{2}\Z$} & \text{if $m'' = m_k$, $k \in \{1,2\}$}, \\ 
0 & \text{otherwise}.
\end{cases}\] 
This implies that the element $L_t(m)$ is also the bar invariant element of the form \eqref{eq:charL} in $\F(\cY_t)$. 
Therefore, we obtain $\eta_\bfi(x) = L_t(m)$, which completes the proof for type $\mathrm{ABDE}$.

Next, we consider the other case, that is, when $\fg$ of type $\mathrm{CFG}$.
Let $x \in \cA_\bfi$ be an arbitrary quantum cluster monomial, and $m \in \cM_{\le \xi}$ the dominant monomial satisfying $\bar{\eta}_\bfi(\evt(x)) = [L(m)]$.
Let $\cQ' = (\Delta, \id, \xi')$ be a Q-datum for the simply-laced Lie algebra $\sg$ whose Dynkin diagram is $\Delta$.  
Take a sequence $\bfi' \in \Delta_0^\N$ adapted to $\cQ'$ and set $x' \seq \eta_{\bfi'}^{-1}\Psi(\cQ',\cQ)\eta_\bfi(x) \in \cA_{\bfi'}$.
By Corollary~\ref{Cor:Psi}, this $x'$ is a quantum cluster monomial.
Since we already know that the assertion of the theorem is true for simply-laced type, we have $\eta_{\bfi'}(x') \in \bfL_{t,\le \xi'}$.
Since $\Psi(\cQ, \cQ')$ respects the canonical bases (Theorem~\ref{Thm:FHOO}), we have 
$\eta_\bfi(x) = \Psi(\cQ,\cQ')(\eta_{\bfi'}(x')) \in \bfL_{t, \le \xi}$.
In other words, there is a dominant monomial $m' \in \cM_{\le \xi}$ such that $\eta_{\bfi}(x) = L_t(m')$.
By the commutativity of \eqref{eq:comm-ev}, we have $\evt L_t(m') = \chi_q(L(m))$.
By comparing \eqref{eq:L} and \eqref{eq:Lt}, it implies that $m' = m$. 
\end{proof}

Let us consider the following notion of reachable simple modules, that is of simple modules corresponding to cluster monomials.

\begin{Def} \label{Def:reachable}
We say that a simple module $L(m)$ is \emph{reachable} if there is a $\cQ$-datum $\cQ = (\Delta, \sigma, \xi)$ for $\fg$ such that $m \in \cM_{\le \xi}$ and $\bar{\eta}_\bfi^{-1}[L(m)]$ is a cluster monomial in $\bar{\cA}_\bfi$ for some (or any) sequence  $\bfi \in \Delta_0^\N$ adapted to $\cQ$.
\end{Def}

\begin{Rem}
Thanks to Theorem~\ref{Thm:main} above, $L(m)$ is reachable if and only if there is a $\cQ$-datum $\cQ = (\Delta, \sigma, \xi)$ for $\fg$ such that $m \in \cM_{\le \xi}$ and $\eta_\bfi^{-1}L_t(m)$ is a quantum cluster monomial in $\cA_\bfi$ for some (or any) sequence $\bfi \in \Delta_0^\N$  adapted to $\cQ$.
\end{Rem}

\begin{Cor}
The isomorphism $\Psi \colon \cK_t(\Cc'_\Z) \to \cK_t(\Cc_{\Z})$ in Theorem~\ref{Thm:FHOO} induces a bijection between the sets of $(q,t)$-characters of reachable modules.
\end{Cor}
\begin{proof}
This is immediate from Corollary~\ref{Cor:Psi}.
\end{proof}

\subsection{Corollaries}
We explain several applications of our main results. 
The statements of Corollaries \ref{Cor:KLr}, \ref{Cor:KRF}, \ref{Cor:pos} are known in types $\mathrm{ADE}$ \cite{Nak04} and type $\mathrm{B}$ \cite{FHOO} but are new for general types.

We first state the Kazhdan--Lusztig type conjecture (= Conjecture~\ref{Conj:KL}~\eqref{Conj:KL:KL}) for reachable modules. Note that reachable modules include KR-modules, and when $\fg$ is of types $\mathrm{CFG}$, the following result is new even for KR-modules. 
%Since the Conjecture was formulated in \cite{Her04} almost 20 years ago, this is the first result in this direction for types $\mathrm{CFG}$ (for these types, the result is new even in the case of KR-modules for example).

\begin{Cor}\label{Cor:KLr}
If $L(m)$ is reachable, we have $\evt(L_t(m)) = \chi_q(L(m))$.
\end{Cor}

\begin{proof}
Noting the commutativity of the diagram~\eqref{eq:comm-ev}, this is an immediate consequence of Theorem~\ref{Thm:main}.
\end{proof}

Then we establish Conjecture~\ref{Conj:KRF} for any $\fg$.

\begin{Cor} \label{Cor:KRF}
If $L(m)$ is a KR module, we have $F_t(m) = L_t(m)$.
\end{Cor}
\begin{proof}
Assume that $L(m)$ is a KR module.
Theorems~\ref{Thm:qcl} \& \ref{Thm:main} tell us that $F_t(m) \in \bfL_t$.
Thus, there exists $m' \in \cM$ such that $F_t(m) = L_t(m')$.
The characterization of $F_t(m)$ in Theorem~\ref{Thm:Ft} and the equation \eqref{eq:Lt} force that $m' = m$, which proves the assertion.
\end{proof}

We also obtain the positivity of the coefficients of $(q,t)$-characters of simples modules, that is, we prove Conjecture~\ref{Conj:KL}~\eqref{Conj:KL:pos} for any $\fg$, 
which was formulated for non-simply-laced types almost 20 years ago in \cite{Her04}. For types $\mathrm{CFG}$, the statement was only known for fundamental representations and was derived from a computer calculation in \cite{Her05}. Hence we obtain an explanation of these computational results and a vast generalization of the statement.

\begin{Cor}\label{Cor:pos} The coefficients of $(q,t)$-characters of simple modules are positive. 
Precisely, for any $m \in \cM$ and $m' \in \cM^*$ $($with $m' < m)$, we have $a_t[m;m'] \in \N_0[t^{\pm 1/2}]$ in the notation of \eqref{eq:Lt}. 
\end{Cor}

\begin{proof} The idea of the proof is to identify the coefficients $a_t[m;m']$ with structure constants of the quantum Grothendieck ring with respect to the canonical basis $\bfL_t$ by using Corollary~\ref{Cor:KRF} and Lemma~\ref{Lem:qKR}.
Then the desired positivity of $a_t[m;m']$ follows from the positivity of structure constants (Theorem~\ref{Thm:pos}).
Since it is the same argument as for the proof of \cite[Theorem 11.7]{FHOO}, we omit the details. 
\end{proof}

\begin{comment}
\begin{Rem}
Combined with the result by Qin~\cite{Qin12}, we also have the $t$-analog of Hernandez-Leclerc's geometric character formula~\cite[Conjecture 5.3]{HL16} (which is now a theorem thanks to \cite{KKOP2}) for any reachable module. 
This gives a geometric explanations for the positivity in Corollary \ref{Cor:pos} for reachable modules. Our general result suggests the existence of a geometric interpretation for all simple modules.
\end{Rem}
\end{comment}

\subsection{Cluster structure on $\cK_t(\Cc_\Z)$} 
In this subsection, as an application of the results from the previous subsections, we briefly discuss how one can lift the cluster algebra structure on the Grothendieck ring $K(\Cc_\Z)$ of the whole category $\Cc_\Z$ fully investigated by \cite{KKOP2}, to the quantum Grothendieck ring $\cK_t(\Cc_\Z)$.

First, we recall the cluster algebra structure on $K(\Cc_\Z)$.
In the original paper~\cite{KKOP2}, it is described by a combinatorial gadget called the \emph{admissible chains of $i$-boxes}.
Here, let us explain it with an equivalent but a little bit different terminology to make the things more suitable with our notations.   
Let $\cQ = (\Delta, \sigma, \xi)$ be a Q-datum for $\fg$.
Recall the sets $\hD_{\le \xi} \subset \hD_{[\xi]} \subset \Delta_0 \times \Z$ from \S\ref{Ssec:Qdata} and the partial ordering $\preceq$ on $\hD_{[\xi]}$ from \S\ref{Ssec:adpt}.
Put $\hD_{> \xi} \seq \hD_{[\xi]}\setminus \hD_{\le \xi}$.
We say that a bijection $\fe \colon \N \to \Delta_{[\xi]}$ is a \emph{$\cQ$-adapted enumeration} if the restrictions of its inverse $\fe^{-1}$ to the subsets $\hD_{\le \xi}$ and $\hD_{> \xi}$ give the morphisms of posets $(\hD_{\le \xi}, \preceq) \to (\N, \le)$ and $(\hD_{> \xi}, \preceq^{\mathrm{op}}) \to (\N, \le)$ respectively (cf.~Lemma~\ref{Lem:ad}).   
Given such a $\cQ$-adapted enumeration $\fe$, for each $u \in \N$, we write $\fe(u) = (\im_u, p_u) \in \Delta_0 \times \Z$ and define
\[
u^\star \seq \max(\{ v \in [1, u] \mid \im_u = \im_v, \sgn(u) \neq \sgn(v) \}\cup\{0\}),
\]
where 
\[\sgn(u) \seq \begin{cases}
1 & \text{if $\fe(u) \in \hD_{>\xi}$}, \\
-1 & \text{if $\fe(u) \in \hD_{\le \xi}$}.
\end{cases}\]
With this notation, we have a collection $\{ L(m^\fe_u)\}_{u \in \N}$ of KR modules in $\Cc_\Z$ given by 
\begin{align}\label{eq: m_u} 
m_u^\fe \seq \begin{cases}
m^{(\im_u)}[p_u, p_{u^\star}] & \text{if $\sgn(u) = -1$ and $u^\star \neq 0$}, \\
m^{(\im_u)}[p_u, \xi_{\im_u}] & \text{if $\sgn(u) = -1$ and $u^\star = 0$}, \\
m^{(\im_u)}[p_{u^\star},p_u] & \text{if $\sgn(u) = 1$ and $u^\star \neq 0$}, \\
m^{(\im_u)}(\xi_{\im_u},p_u] & \text{if $\sgn(u) = 1$ and $u^\star = 0$}.
\end{cases}
\end{align}
By \cite[Theorem 5.5]{KKOP2}, the collection $\{L(m^\fe_u)\}_{u \in \N}$ forms a commuting family of real simple modules, i.e., for any $u,v \in \N$, we have 
\[L(m^\fe_u)\otimes L(m^\fe_v) \simeq L(m^\fe_u \cdot m^\fe_v). \]
Now, we define a skew-symmetric matrix $\Lambda^\fe = (\Lambda^\fe_{u,v})_{u,v \in \N}$ by
\[ \Lambda^\fe_{u,v} \seq \Nn(m^\fe_u, m^\fe_v). \]

\begin{Thm}[{\cite[Theorem 8.1]{KKOP2}}]
\label{Thm:KKOP_CZ}
Let $\cQ = (\Delta, \sigma, \xi)$ be a Q-datum for $\fg$.
For any $\cQ$-adapted enumeration $\fe \colon \N \to \hD$, there is a unique exchange matrix $\tB^\fe = (b^\fe_{u,v})_{u,v \in \N}$ such that $\sum_{k \in \N} b^\fe_{ku}\Lambda^\fe_{kv} = 2\delta_{u,v}$ for all $u,v \in \N$.
Moreover, we have a ring isomorphism
\[ \bar{\eta}_\fe \colon \cA(\tB^\fe) \simeq K(\Cc_\Z) \]
under which the initial cluster variable $X_u$ corresponds to $[L(m^\fe_u)]$ for any $u \in \N$, and every cluster monomial corresponds to the class of a simple module. 
\end{Thm}

Examples are given in Appendix \ref{sec:QCB}.

\begin{Prop}[{\cite[\S8]{KKOP2}}] \label{Prop:reachable}
The isomorphism $\bar{\eta}_\fe$ in Theorem~\ref{Thm:KKOP_CZ} induces a bijection between the cluster monomials in $\cA(\tB^\fe)$ and the classes of reachable modules in the sense of Definition~\ref{Def:reachable}.
\end{Prop}
\begin{proof}
Let $x$ be an arbitrary cluster monomial in $\cA(\tB^\fe)$.
Then, the discussion in the proof of \cite[Theorem 8.1]{KKOP2} implicitly tells us that, for any Q-datum $\cQ = (\Delta, \sigma, \xi)$ such that $\bar{\eta}_\fe(x) \in K(\Cc_{\le \xi})$, the element $\bar{\eta}_\bfi^{-1}\bar{\eta}_{\fe}(x)$ is a cluster monomial in $\bar{\cA}_{\bfi}$ for any sequence $\bfi \in \Delta_0^\N$ adapted to $\cQ$.   
Therefore, $\bar{\eta}_\fe(x)$ is reachable.
Similarly, for any reachable $L(m)$, we see that the element $\bar{\eta}_\fe^{-1}([L(m)])$ is a cluster monomial in $\cA(\tB^\fe)$.
\end{proof}

Now, we shall prove a quantum analog of Theorem~\ref{Thm:KKOP_CZ}.

\begin{Thm} \label{Thm:qmc_CZ}
Let $\cQ = (\Delta, \sigma, \xi)$ be a Q-datum for $\fg$.
For any $\cQ$-adapted enumeration $\fe \colon \N \to \hD_{[\xi]}$, we have an isomorphism of $\Z[t^{\pm 1/2}]$-algebras
\[ \eta_\fe \colon \cA_t(\Lambda^\fe, \tB^\fe) \simeq \cK_t(\Cc_\Z) \]
under which the initial quantum cluster variable $X_u$ corresponds to $L_t(m^\fe_u)$ for any $u \in \N$, and every cluster monomial corresponds to the $(q,t)$-character of a simple module. 
Here $\tB^\fe$ is the unique exchange matrix in Theorem~\ref{Thm:KKOP_CZ} above.
\end{Thm}
\begin{proof}
The proof is similar to the former part of the proof of Theorem~\ref{Thm:main}.
Thanks to Theorem~\ref{Thm:pos} and Corollary~\ref{Cor:KLr}, for any mutually commuting pair of reachable modules $L(m)$ and $L(m')$, we have
\[ L_t(m)L_t(m') = t^{\Nn(m,m')/2} L_t(mm').\]    
In particular, we have an embedding $\eta_\fe \colon \F(\cT(\Lambda^\fe)) \hookrightarrow \F(\cY_t)$ of skew fields over $\Q(t^{1/2})$ given by $\eta_\fe(X_u) = L_t(m_u^\fe)$ for all $u \in \N$. 
By a similar argument as in the proof of Theorem~\ref{Thm:main}, using Theorem~\ref{Thm:pos}, Corollary~\ref{Cor:KLr} and Proposition~\ref{Prop:reachable}, we can show that $\eta_\fe(X) = L_t(m)$ holds for any quantum cluster variable $X$ in $\cA_t(\Lambda^\fe, \tB^\fe)$, where $m \in \cM$ the dominant monomial such that $\bar{\eta}_\fe(\evt(X)) = [L(m)]$ in the notation of Theorem~\ref{Thm:KKOP_CZ} above.
Therefore, the image of $\cA_t(\Lambda^\fe, \tB^\fe)$ under the embedding $\eta_\fe$ is included in $\cK_t(\Cc_\Z)$.
Since every fundamental module is reachable, the above discussion also tells us that $L_t(Y_{i,p}) = F_t(Y_{i,p}) \in \eta_\fe(\cA_t(\Lambda^\fe, \tB^\fe))$ for any $(i,p) \in \hI$. 
Since the set $\{F_t(Y_{i,p})\}_{(i,p) \in \hI}$ is generating the $\Z[t^{\pm 1/2}]$-algebra $\cK_t(\Cc_\Z)$, we conclude that $\eta_\fe( \cA_t(\Lambda^\fe, \tB^\fe)) = \cK_t(\Cc_\Z)$. 
\end{proof}

\section{Substitution formulas}\label{sec:Substitution}

In Theorem \ref{Thm:Psi=tau}, we gave a cluster theoretical interpretation of $\Psi|_{\cK_t(\Cc'_{\le \xi'})}$. As an application of it, we show in this section that $\Psi$ comes from an explicit birational transformation among the variables in $\cY_t$ and $\cY'_t$, which we call \emph{substitution formulas}. It reveals a non-trivial relation among the $(q, t)$-characters of simple modules which are mutually related under $\Psi$. 

Let us state the main result of this section more precisely. We return to the assumption and the notation in Sections \ref{Ssec:FHOO} and \ref{Ssec:CTinter}. Recall that we have 
\[
\cK_t(\Cc_{\Z})\subset \cY_{t}\subset \F(\cY_t),\qquad 
\cK_t(\Cc'_{\Z})\subset \cY'_{t}\subset \F(\cY'_t),
\]
where $\F(\cY_t)$ and $\F(\cY'_t)$ denote the skew field of fractions of $\cY_t$ and $\cY'_t$, respectively. The main theorem in this section is the following: 

\begin{Thm}[Substitution formulas]\label{Thm:Substitution}
With the above assumption, there exists an isomorphism of skew fields 
\[
\tPsi\equiv \tPsi(\cQ, \cQ')\colon \F(\cY'_t)\xrightarrow{\sim} \F(\cY_t), 
\]
such that 
\begin{itemize}
    \item[(1)] $\fD_t\circ \tPsi=\tPsi\circ \fD_t^{\prime}$, 
    \item[(2)] $\tPsi(\ul{Y_{i, p}})\in \F(\cY_{t, \fD^k\cQ})$ for $k\in \Z$ and $(i, p)\in \hI'_{\fD^{\prime k}\cQ'}$,
    \item[(3)] the following diagram commutes. 
\begin{equation*} %\label{eq:diagT}
\vcenter{
\xymatrix{
\F(\cY'_t) \ar[r]^-{\tPsi}  & \F(\cY_t) \\
\cK_t(\Cc'_{\Z}) \ar@{^{(}->}[u]\ar[r]^-{\Psi}& \cK_t(\Cc_{\Z})\ar@{^{(}->}[u]
}}
\end{equation*} 
\end{itemize}
In particular, there exists a birational transformation between the variables in $\cY'_t$ and those in $\cY_t$ which makes 
the $(q, t)$-characters of simple modules in $\Cc'_{\Z}$ into those in $\Cc_{\Z}$. 
\end{Thm}
Our proof of Theorem \ref{Thm:Substitution} is valid under the setting when $t=1$. Hence we obtain the parallel result at $t=1$ by the same proof. Let $\F(\cY)$ (rep.~$\F(\cY')$) be the field of fractions of $\cY$ (resp.~$\cY'$). Denote by $\Psi_{t=1}$ the $\Z$-algebra isomorphism which makes the following diagram commutative:
  \[ \vcenter{
\xymatrix{
\cK_t(\Cc'_{\Z})\ar[r]^-{\Psi}_-{\sim}\ar@{->>}^-{\evt}[d]& \cK_t(\Cc_{\Z})\ar@{->>}^-{\evt}[d]\\
\chi_q(K(\Cc'_{\Z}))\ar[r]^-{\Psi_{t=1}}_-{\sim}& \chi_q(K(\Cc_{\Z}))
}}
\]
We call the image of the $(q, t)$-characters of simple modules under $\evt$ the \emph{$(q, 1)$-characters} of simple modules. 

\begin{Thm}[Substitution formulas at $t=1$]\label{Thm:Substitution_cl}
With the above assumption, there exists an isomorphism of fields 
\[
\tPsi_{t=1}\equiv \tPsi(\cQ, \cQ')_{t=1}\colon \F(\cY')\xrightarrow{\sim} \F(\cY), 
\]
such that 
\begin{itemize}
    \item[(1)] $\fD\circ \tPsi_{t=1}=\tPsi_{t=1}\circ \fD^{\prime}$, 
    \item[(2)] $\tPsi_{t=1}(Y_{i, p})\in \F(\cY_{\fD^k\cQ})$ for $k\in \Z$ and $(i, p)\in \hI'_{\fD^{\prime k}\cQ'}$,
    \item[(3)] the following diagram commutes. 
\begin{equation*}  
\vcenter{
\xymatrix{
\F(\cY') \ar[r]^-{\tPsi_{t=1}}  & \F(\cY) \\
\chi_q(K(\Cc'_{\Z}))\ar@{^{(}->}[u]\ar[r]^-{\Psi_{t=1}}& \chi_q(K(\Cc_{\Z})).\ar@{^{(}->}[u]
}}
\end{equation*} 
\end{itemize}
In particular, there exists a birational transformation between the variables in $\cY'$ and those in $\cY$ which makes the $(q, 1)$-characters of simple modules in $\Cc'_{\Z}$ into those in $\Cc_{\Z}$. 
\end{Thm}
\begin{Rem}
        We can calculate $\tPsi$ and $\tPsi_{t=1}$ explicitly by tracing a specific mutation sequence. See Appendix \ref{sec:QCC}.
\end{Rem}
    The coincidence of the $(q, 1)$-characters and the $q$-characters are now known in many cases. See Conjecture \ref{Conj:KL}, Corollary \ref{Cor:KLr}, and expositions around them. Hence the substitution formulas at $t=1$ reveal several non-trivial relations among the $q$-characters of simple modules which are mutually related under $\Psi_{t=1}$.

The rest of this section is devoted to the proof of Theorem \ref{Thm:Substitution}. We consider the following commutative diagram: 
\begin{equation}{
\scalebox{0.87}{$
\xymatrix{
&&&&&&&\\
\F(\cY'_{t, \le \xi'})\ar@/^35pt/[rrrrrr]^-{\tPsi_{\le \xi, \le \xi'}}&&
\F(\cT(\Lambda_{\bfi'}))\ar[ll]^-{\sim}_-{\teta_{\bfi', \F}}\ar[rr]_-{\sim}^-{\hat{\bftau}_{\F}^*}&&
\F(\cT(\Lambda_{\bfi}))\ar[rr]^-{\teta_{\bfi, \F}}_-{\sim}&&
\F(\cY_{t, \le \xi})\\
%&&&&&&\\
\cY'_{t, \le \xi'}\ar@{^{(}->}[u]&&
\cT(\Lambda_{\bfi'})\ar[ll]^-{\sim}_-{\teta_{\bfi'}}\ar@{^{(}->}[u]&&
\cT(\Lambda_{\bfi})\ar[rr]_{\sim}^{\teta_{\bfi}}\ar@{^{(}->}[u]&&
\cY_{t, \le \xi}\ar@{^{(}->}[u]\\
\cK_t(\Cc'_{\le \xi'})_{\le \xi'}\ar@{^{(}->}[u]&
\cK_t(\Cc'_{\le \xi'})\ar[lu]_-{(\cdot)_{\le \xi'}}\ar[l]_-{\sim}\ar@/_35pt/[rrrr]^-{\Psi|_{\cK_t(\Cc'_{\le \xi'})}}_-{\sim}&
\cA_{\bfi'}\ar[l]^-{\eta_{\bfi'}}_-{\sim}\ar[rr]^-{\hat{\bftau}^*}_-{\sim}\ar@{^{(}->}[u]&&
\cA_{\bfi}\ar[r]_-{\eta_{\bfi}}^-{\sim}\ar@{^{(}->}[u]&
\cK_t(\Cc_{\le \xi})\ar[ru]^-{(\cdot)_{\le \xi}}\ar[r]^-{\sim}&
\cK_t(\Cc_{\le \xi})_{\le \xi}\ar@{^{(}->}[u]\\
\cK_t(\Cc'_{\le \fD^{\prime -1}\xi'})\ar@{^{(}->}[u]\ar@/_30pt/[rrrrrr]^-{\Psi|_{\cK_t(\Cc'_{\le \fD^{\prime -1}\xi'})}}_-{\sim}\ar@{=}[r]^-{\mathrm{id}}&
\cK_t(\Cc'_{\le \fD^{\prime -1}\xi'})\ar@{^{(}->}[u]&&&&
\cK_t(\Cc_{\le \fD^{-1}\xi})\ar@{^{(}->}[u]\ar@{=}[r]^-{\mathrm{id}}&
\cK_t(\Cc_{\le \fD^{-1}\xi})\ar@{^{(}->}[u]\\
&&&&&
}$}}\label{eq:diagram}
\end{equation}

Here the morphisms $\teta_{\bfi, \F}, \teta_{\bfi', \F}, \hat{\bftau}_{\F}^*$, and $\tPsi_{\le \xi, \le \xi'}$ are uniquely defined so that the diagram becomes commutative. The hooked arrows denote the inclusion maps. Note that the restriction of $(\cdot)_{\le \xi}$ (resp.~$(\cdot)_{\le \xi'}$) to $\cK_t(\Cc_{\le \fD^{-1}\xi})$ (resp.~$\cK_t(\Cc'_{\le \fD^{\prime -1}\xi'})$) is the identity map by Lemma \ref{Lem:tr}.

The key step for the proof of Theorem \ref{Thm:Substitution} is to show the compatibility of $\tPsi_{\le \xi, \le \xi'}$ with the dual operations, that is, $\fD_t^{-1}\circ \tPsi_{\le \xi, \le \xi'}=\tPsi_{\le \xi, \le \xi'}\circ \fD_t^{\prime -1}$ on $\F(\cY'_{t, \le \xi'})$ (Lemma \ref{Lem:compatD} (1)). Once we established it, we can extend $\tPsi_{\le \xi, \le \xi'}$ to $\tPsi$ by using the dual operations. Note that we can not deduce the compatibility of $\tPsi_{\le \xi, \le \xi'}$ with the dual operations directly from Theorem~\ref{Thm:Psi=tau}  because the truncation map $(\cdot)_{\le \xi} \colon \cK_t(\Cc_{\le \xi}) \hookrightarrow \cY_{t, \le \xi}$ does not commute with $\fD_t^{- 1}$.  

\begin{Lem}\label{Lem:compatD}\ 
\begin{itemize}
    \item[(1)] $\left.\fD_t^{-1}\right|_{\F(\cY_{t, \le \xi})}\circ \tPsi_{\le \xi, \le \xi'}=\tPsi_{\le \xi, \le \xi'}\circ \left.\fD_t^{\prime -1}\right|_{\F(\cY'_{t, \le \xi'})}$. 
    \item[(2)] For $k\in \N_0$ and $(i, p)\in \hI'_{\fD^{\prime -k}\cQ'}$, we have $\tPsi_{\le \xi, \le \xi'}(\ul{Y_{i, p}})\in \F(\cY_{t, \fD^{-k}\cQ})$. 
\end{itemize}
\end{Lem}
\begin{proof}
The statement (2) immediately follows from (1) since we have $\tPsi_{\le \xi, \le \xi'}(\F(\cY'_{t, \cQ'}))\subset \F(\cY_{t, \cQ})$ by the construction of $\tPsi_{\le \xi, \le \xi'}$. Hence we shall prove (1). 

Let us define the homomorphisms of skew fields as 
\begin{align*}
&\tSigma\seq \teta_{\bfi, \F}^{-1}\circ \fD_t^{-1}\circ\teta_{\bfi, \F} \colon \F(\cT(\Lambda_{\bfi}))\to \F(\cT(\Lambda_{\bfi})),\\
&\tSigma'\seq \teta_{\bfi', \F}^{-1}\circ \fD_t^{\prime -1}\circ\teta_{\bfi', \F} \colon \F(\cT(\Lambda_{\bfi'}))\to \F(\cT(\Lambda_{\bfi'})).
\end{align*}
Let $\sfSg, \tsfSg\colon \Z^{\oplus \N}\to \Z^{\oplus \N}$ be $\Z$-module homomorphisms determined by 
\[
\sfSg(\bfe_u)=\bfe_{u+\ell},\qquad 
\tsfSg(\bfe_u)=\bfe_{u+\ell}-\bfe_{(\ell+1)_{\bfi}^{-}(\im_u^{\ast})}\qquad \text{for all }u\in \N.
\]
Note that we have 
\begin{equation}
\tsfSg(\bfa)=\sfSg(\bfa)-\sum_{\im\in \Delta_0}\sfp_{\bfi}(\bfa; \im)\bfe_{(\ell+1)_{\bfi}^{-}(\im^{\ast})}\label{eq:tsfSg}
\end{equation}
for $\bfa\in \Z^{\oplus \N}$ by the notation defined in \eqref{eq:p-sum}. 
Recall $\rho_{\bfi}\colon \N\to \hD_{\le \xi}$ defined in Section \ref{Ssec:adpt}. Then, for $u \in \N$ with $\rho_{\bfi}(u)=(\im, p)$, we have 
\begin{align}
	\tSigma(X_u)&=\teta_{\bfi, \F}^{-1}(\fD_t^{-1}(\teta_{\bfi, \F}(X_u)))\label{eq:tSigmacalc}\\
	&=\teta_{\bfi, \F}^{-1}(\fD_t^{-1}(\ul{m^{(\im)}[p,\xi_{\im}]}))\notag\\
	&=\teta_{\bfi, \F}^{-1}(\ul{m^{(\im^{\ast})}[p-rh^{\vee},\xi_{\im}-rh^{\vee}]})\notag\\
	&=\teta_{\bfi, \F}^{-1}(\ul{m^{(\im^{\ast})}[p-rh^{\vee},\xi_{\im^{\ast}}](m^{(\im^{\ast})}((\fD^{-1}\xi)_{\im^{\ast}},\xi_{\im^{\ast}}])^{-1}})\notag\\	
	&=X^{\tsfSg(\bfe_u)}\notag
\end{align}
by the condition~\eqref{eq:cond2}. Therefore, we have $\tSigma(X^{\bfa})=X^{\tsfSg(\bfa)}$ for $\bfa\in \Z^{\oplus \N}$. 
Moreover, in the notation of Proposition \ref{Prop:y-hat}, we have 
\begin{align}
    \tSigma(X^{\bfb_u})=\teta_{\bfi, \F}^{-1}(\fD_t^{-1}(\teta_{\bfi, \F}(X^{\bfb_u})))
    =\teta_{\bfi, \F}^{-1}(\fD_t^{-1}(\ul{A_{i, p-d_i}^{-1}}))
    =\teta_{\bfi, \F}^{-1}(\ul{A_{i^{\ast}, p-rh^{\vee}-d_{i^{\ast}}}^{-1}})
    =X^{\bfb_{u+\ell}}\label{eq:a-vect}
\end{align}
for $u\in \N$ with $\bar{\rho}_{\bfi}(u)=(i, p)$ by Proposition \ref{Prop:y-hat} and the condition~\eqref{eq:cond2}. 

For $u\in \N$, set $Z_{u}\seq\hat{\bftau}^*(X'_{u})$. Then, by the definition of $\hat{\bftau}^*$, $Z_{u}$ is a quantum cluster variable in $\cA_{\bfi}$. Therefore, by \cite[Theorem 5.3]{Tran}, there exist $c_{u, \bfn}\in \Z[t^{\pm 1/2}]$ for $\bfn\in \N_0^{\oplus \N} \setminus\{0\}$ and $\bfg_u\in \Z^{\oplus \N}$ such that  
\[
Z_{u}=X^{\bfg_u}\left(1+\sum_{\bfn=(n_k)_k \in \N_0^{\oplus \N} \setminus\{0\}} c_{u, \bfn} X^{\sum_{k\in \N}n_k\bfb_k}\right).%ok 
\]
For $k\in \N_{0}$, set 
\begin{align*}
&\hD_{\fD^{ -k}\cQ}^{\mathrm{fr}}\coloneqq \{(\im, (\fD^{-(k+1)}\xi)_{\im}+2d_{\bar{\im}})\mid \im \in \Delta_0\},\\
&\hD_{\fD^{ -k}\cQ}^{\mathrm{uf}}\coloneqq \hD_{\fD^{-k}\cQ}\setminus \hD_{\fD^{ -k}\cQ}^{\mathrm{fr}}.
\end{align*}
For $u\in \N$, there exists $k\in \N_{0}$ such that the quantum cluster variable $Z_{u}$ is obtained from the initial cluster $\{X_s\mid s\in \N\}$ by the mutations only at the vertices labelled by $\hD_{\fD^{ -k}\cQ}^{\mathrm{uf}}$. See the construction of $\hat{\bftau}^*$ in Section \ref{Ssec:CTinter} and recall the fact that $\varepsilon_1^* \varepsilon_2^* = \varepsilon_2^*\varepsilon_1^*$ whenever $\varepsilon_i \in \{\gamma_{k+n_i\ell}\}_{k \in [1,\ell-1]} \cup \{\beta_{k+n_i\ell}\}_{k \in [1,\ell-2]}$ and $n_1 \neq n_2$. (If $Z_u$ belongs to the initial cluster, then we may take arbitrary $k\in \N_{0}$.) 

Moreover, if $Z_{u}$ is obtained from the initial cluster by the mutation sequence $\mu_{k_{s_u}^{(u)}}\cdots \mu_{k_1^{(u)}}$ for $k_1^{(u)},\dots, k_{s_u}^{(u)}\in \rho_{\bfi}^{-1}(\hD_{\fD^{ -k}\cQ}^{\mathrm{uf}})$, then $Z_{u+\ell}$ is obtained from the initial cluster by the mutation sequence $\mu_{k_{s_u}^{(u)}+\ell}\cdots \mu_{k_1^{(u)}+\ell}$, by the construction of $\hat{\bftau}^*$. 
Since $c_{u, \bfn}\in \Z[t^{\pm 1/2}]$ and $\bfg_u$ are determined only from $\tB_{\bfi}$ and the sequence $k_1^{(u)},\dots, k_{s_u}^{(u)}$ \cite[Theorem 5.3]{Tran}, this periodicity together with the periodicity of $\tB_{\bfi}$ coming from the condition~\eqref{eq:cond2} implies that 
\begin{align}
&\bfg_{u+\ell}
\begin{cases}
=\sfSg(\bfg_{u})&\text{if}\ u>\ell\ \text{or}\ u=(\ell+1)_{\bfi'}^{-}(\im)\ \text{for}\ \im\in \Delta_0,\\
\in \sfSg(\bfg_{u})+\sum_{\im\in \Delta_0}\Z\bfe_{(\ell+1)_{\bfi}^{-}(\im)}&\text{otherwise},
\end{cases}
\quad\text{and}\label{eq:periodicity-g}\\
&c_{u+\ell, \bfn}=
\begin{cases}
	c_{u, \sfSg^{-1}(\bfn)}&\text{if}\ \bfn\in \sfSg(\N_0^{\oplus \N}),\\
	0&\text{otherwise},
\end{cases}\label{eq:periodicity-c}
\end{align}
for $u\in \N$. (Here the statement for $\bfg_{u+\ell}$ is divided by cases since the vertices in $\hD_{\cQ}^{\mathrm{uf}}$ does not have a neighborhood corresponding to $\hD_{\fD\cQ}^{\mathrm{fr}}$ in $\Gamma_{\bfi}$.) Note that $Z_{(\ell+1)_{\bfi'}^{-}(\im)}=X_{(\ell+1)_{\bfi}^{-}(\im)}$ and $Z_{(\ell+1)_{\bfi'}^{-}(\im)+\ell}=X_{(\ell+1)_{\bfi}^{-}(\im)+\ell}$.

Therefore, for $u\in \N$, we have 
\begin{align}
	\tSigma(Z_{u})
	&=X^{\tsfSg(\bfg_u)}\left(1+\sum_{\bfn=(n_k)_k \in \N_0^{\oplus \N} \setminus\{0\}} c_{u, \bfn} X^{\sum_{k\in \N}n_k\bfb_{k+\ell}}\right)\qquad \text{by }\eqref{eq:a-vect} \allowdisplaybreaks\label{eq:dualZ}\\
	&=X^{\tsfSg(\bfg_u)}\left(1+\sum_{\bfn=(n_k)_k \in \N_0^{\oplus \N} \setminus\{0\}} c_{u+\ell, \sfSg(\bfn)} X^{\sum_{k\in \N}n_k\bfb_{k+\ell}}\right)\qquad \text{by }\eqref{eq:periodicity-c} \allowdisplaybreaks\notag\\
	&=X^{\tsfSg(\bfg_u)}\left(1+\sum_{\bfn=(n_k)_k \in \N_0^{\oplus \N} \setminus\{0\}} c_{u+\ell, \bfn} X^{\sum_{k\in \N}n_k\bfb_{k}}\right)\qquad \text{by }\eqref{eq:periodicity-c} \allowdisplaybreaks\notag\\
	&\simeq X^{\tsfSg(\bfg_u)-\bfg_{u+\ell}}Z_{u+\ell}\notag
\end{align}
Here $\simeq$ stands for the equality up to multiplication by $t^{a/2}$ for some $a\in \Z$. (We use this notation throughout this proof.) Note that, by \eqref{eq:tsfSg} and \eqref{eq:periodicity-g}, $\tsfSg(\bfg_u)-\bfg_{u+\ell}\in \sum_{\im\in \Delta_0}\Z\bfe_{(\ell+1)_{\bfi}^{-}(\im)}$, and every $Z_s$ ($s\in \N$) $t$-commutes with $X^{\tsfSg(\bfg_u)-\bfg_{u+\ell}}$ since $\rho_{\bfi}((\ell+1)_{\bfi}^{-}(\im))\in \hD_{\cQ}^{\mathrm{fr}}$ for $\im\in \Delta_0$. 

To show the statement (1), it suffices to prove 
\begin{align}
\fD_t^{-1}(\tPsi_{\le \xi, \le \xi'}(\ul{Y_{i,p}}))=\tPsi_{\le \xi, \le \xi'}(\fD_t^{\prime -1}(\ul{Y_{i,p}}))\label{eq:statement}
\end{align}
for all $(i, p)\in \hI'_{\le \xi'}$. First, we show the equality \eqref{eq:statement} for $(i, p)\in \hI'_{\le \fD^{\prime -1}\xi'}$.
Let $(i, p)\in \hI'_{\le \fD^{\prime -1}\xi'}$ and $u\coloneqq \bar{\rho}_{\bfi'}^{-1}((i, p))\in \N$. Then $u^-\seq u^-_{\bfi'}\in \N$, and
\begin{align*}
	\fD_t^{-1}(\tPsi_{\le \xi, \le \xi'}(\ul{Y_{i,p}}))
	&=\fD_t^{-1}(\teta_{\bfi, \F}(\hat{\bftau}_{\F}^*(\teta_{\bfi', \F}^{-1}(\ul{Y_{i,p}}))))\allowdisplaybreaks\\
	&=\fD_t^{-1}(\teta_{\bfi, \F}(\hat{\bftau}_{\F}^*(X^{\prime \bfe_{u}-\bfe_{u^-}})))\allowdisplaybreaks\\
	&\simeq \fD_t^{-1}(\teta_{\bfi, \F}(Z_{u^-}^{-1}Z_u))\allowdisplaybreaks\\
	&=\teta_{\bfi, \F}(\tSigma(Z_{u^-}^{-1}Z_u))\allowdisplaybreaks\\
	&\simeq \teta_{\bfi, \F}(X^{\tsfSg(\bfg_u)-\bfg_{u+\ell}-\tsfSg(\bfg_{u^-})+\bfg_{u^-+\ell}} Z_{u^-+\ell}^{-1}Z_{u+\ell}), \allowdisplaybreaks\\
	\tPsi_{\le \xi, \le \xi'}(\fD_t^{\prime -1}(\ul{Y_{i,p}}))
	&=\teta_{\bfi, \F}(\hat{\bftau}_{\F}^*(\teta_{\bfi', \F}^{-1}(\fD_t^{\prime -1}(\ul{Y_{i,p}}))))\allowdisplaybreaks\\
	&=\teta_{\bfi, \F}(\hat{\bftau}_{\F}^*(\tSigma'(\teta_{\bfi', \F}^{-1}(\ul{Y_{i,p}}))))\allowdisplaybreaks\\
	&=\teta_{\bfi, \F}(\hat{\bftau}_{\F}^*(\tSigma'(X^{\prime \bfe_{u}-\bfe_{u^-}}))) \allowdisplaybreaks\\
	&=\teta_{\bfi, \F}(\hat{\bftau}_{\F}^*(X^{\prime \bfe_{u+\ell}-\bfe_{u^-+\ell}}))\quad \text{by the argument parallel to \eqref{eq:tSigmacalc}}, \allowdisplaybreaks\\
	&\simeq \teta_{\bfi, \F}(Z_{u^-+\ell}^{-1}Z_{u+\ell}).
\end{align*}
Let us show that 
\begin{align*}
    \tsfSg(\bfg_u)-\bfg_{u+\ell}-\tsfSg(\bfg_{u^-})+\bfg_{u^-+\ell}=0
\end{align*}
for all $u>\ell$. By the iterated application of Lemmas \ref{Lem:cg}, \ref{Lem:bg}, and \ref{Lem:p-inv}, we have 
\[
\sfp_{\bfi}(\bfg_u; \im)=\sfp_{\bfi'}(\bfe_u; \im)=\delta_{\im'_u, \im}\quad\text{and}\quad \sfp_{\bfi}(\bfg_{u^-}; \im)=\sfp_{\bfi'}(\bfe_{u^-}; \im)=\delta_{\im'_{u^-}, \im}=\delta_{\im'_u, \im}.
\]
Note that, when applying Lemma \ref{Lem:p-inv}, we used the fact that the $s$-th components of $\bfg_u$ and $\bfg_{u^-}$ are equal to zero for all $s\in \rho_{\bfi}^{-1}(\hD_{\cQ}^{\mathrm{uf}})$ by the definition of $\hat{\bftau}^*$ and our assumption on $u$. Therefore, by \eqref{eq:tsfSg} and \eqref{eq:periodicity-g}, we have 
\begin{align*}
    &\tsfSg(\bfg_u)-\bfg_{u+\ell}-\tsfSg(\bfg_{u^-})+\bfg_{u^-+\ell}\\
    &=\sfSg(\bfg_u)-\bfe_{(\ell+1)_{\bfi}^{-}((\im'_u)^{\ast})}-\bfg_{u+\ell}-\sfSg(\bfg_{u^-})+\bfe_{(\ell+1)_{\bfi}^{-}((\im'_u)^{\ast})}+\bfg_{u^-+\ell}\\
    &=0.
\end{align*}
Then we have 
\[
\fD_t^{-1}(\tPsi_{\le \xi, \le \xi'}(\ul{Y_{i,p}}))
\simeq \tPsi_{\le \xi, \le \xi'}(\fD_t^{\prime -1}(\ul{Y_{i,p}})).
\]
Since $\fD_t, \fD_t^{\prime}$, and $\tPsi_{\le \xi, \le \xi'}$ are compatible with the bar involution, $\fD_t^{-1}(\tPsi_{\le \xi, \le \xi'}(\ul{Y_{i,p}}))$ and $\tPsi_{\le \xi, \le \xi'}(\fD_t^{\prime -1}(\ul{Y_{i,p}}))$ are bar-invariant. Therefore, we have 
\[
\fD_t^{-1}(\tPsi_{\le \xi, \le \xi'}(\ul{Y_{i,p}}))
= \tPsi_{\le \xi, \le \xi'}(\fD_t^{\prime -1}(\ul{Y_{i,p}})), 
\]
which completes the proof of \eqref{eq:statement} for $(i, p)\in \hI'_{\le \fD^{\prime -1}\xi'}$. 

Next, we show the equality \eqref{eq:statement} for $(i, p)\in \hI'_{\cQ'}$. 
For $u\in \N$, set 
\[
F_u\coloneqq\left(1+\sum_{\bfn=(n_k)_k \in \N_0^{\oplus \N} \setminus\{0\}} c_{u, \bfn} X^{\sum_{k\in \N}n_k\bfb_k}\right).
\]
Then the calculation in \eqref{eq:dualZ} shows that 
\[
\tSigma(F_{u})=F_{u+\ell}
\]
for $u\in \N$. We prepare the following claim. 
\begin{Claim}\label{Claim:sub}
For $(i, p)\in \hI'_{\le \fD^{\prime -2}\xi'}$, $\tPsi_{\le \xi, \le \xi'}(\ul{Y_{i, p}})\in \F(\cY_{t, \le \fD^{-2}\xi})$. 
\end{Claim}
\begin{proof}[Proof of Claim]
Write $u\coloneqq \bar{\rho}_{\bfi'}^{-1}((i, p))$ and $u^-\seq u^-_{\bfi'}$. Then, 
\begin{align*}
    \tPsi_{\le \xi, \le \xi'}(\ul{Y_{i, p}})&=\teta_{\bfi, \F}(\hat{\bftau}_{\F}^*(X^{\prime \bfe_{u}-\bfe_{u^-}}))
    \simeq \teta_{\bfi, \F}(Z_{u^-}^{-1}Z_u)
    \simeq \teta_{\bfi, \F}(F_{u^-}^{-1}X^{\bfg_u-\bfg_{u^-}}F_u). 
\end{align*}
By Proposition \ref{Prop:y-hat} and \eqref{eq:periodicity-c}, we have
\[
\teta_{\bfi, \F}(F_{u^-}^{-1}), \teta_{\bfi, \F}(F_{u})\in \F(\cY_{t, \le \fD^{-2}\xi}). 
\]
Hence it suffices to show that $\teta_{\bfi, \F}(X^{\bfg_u-\bfg_{u^-}})\in \F(\cY_{t, \le \fD^{-2}\xi})$. 
Since we already proved \eqref{eq:statement} for $(i, p)\in \hI'_{\le \fD^{\prime -1}\xi'}$, we have 
\begin{align*}
\tPsi_{\le \xi, \le \xi'}(\F(\cY'_{t, \le \fD^{\prime -2}\xi'}))
&=\tPsi_{\le \xi, \le \xi'}(\fD_t^{\prime -1}(\F(\cY'_{t, \le \fD^{\prime -1}\xi'})))\\
&=\fD_t^{-1}(\tPsi_{\le \xi, \le \xi'}(\F(\cY'_{t, \le \fD^{\prime -1}\xi'})))\subset \F(\cY_{t, \le \fD^{-1}\xi}). 
\end{align*}
Therefore, $\teta_{\bfi, \F}(F_{u^-}^{-1}X^{\bfg_u-\bfg_{u^-}}F_u)=\tPsi_{\le \xi, \le \xi'}(\ul{Y_{i, p}})\in  \F(\cY_{t, \le \fD^{-1}\xi})$, hence we have 
\begin{equation}
\teta_{\bfi, \F}(X^{\bfg_u-\bfg_{u^-}})\in \F(\cY_{t, \le \fD^{-1}\xi}).\label{eq:locality}
\end{equation}

The property \eqref{eq:periodicity-g} implies that $\teta_{\bfi, \F}(X^{\bfg_u-\bfg_{u^-}})$ is a monomial in the variables $\{\ul{m^{(\im)}[p, \xi_\im]^{\pm 1}}\mid p\leq (\fD^{-2}\xi)_{\im}+2d_{\bar{\im}}, \im\in \Delta_0\}$. Hence, we can write the monomial $\teta_{\bfi, \F}(X^{\bfg_u-\bfg_{u^-}})$ as 
\[
\teta_{\bfi, \F}(X^{\bfg_u-\bfg_{u^-}})=\ul{M\cdot \prod\nolimits_{\im\in \Delta_0}m^{(\im)}[(\fD^{-2}\xi)_{\im}+2d_{\bar{\im}}, \xi_\im]^{k_{\im}}}
\]
for some $k_\im\in \Z$ and a monomial $M$ in $\cY_{\le \fD^{-2}\xi}$. 
Here the property \eqref{eq:locality} implies $k_{\im}=0$ for all $\im\in \Delta_0$. Therefore, $\teta_{\bfi, \F}(X^{\bfg_u-\bfg_{u^-}})\in \F(\cY_{t, \le \fD^{-2}\xi})$.
\end{proof}

Assume that $1\leq u\leq \ell$. Since we already proved \eqref{eq:statement} for $u>\ell$, it suffices to show that 
\[
\fD_t^{-2}(\tPsi_{\le \xi, \le \xi'}(\teta_{\bfi', \F}(X'_u)))=\tPsi_{\le \xi, \le \xi'}(\fD_t^{\prime -2}(\teta_{\bfi', \F}(X'_u))).
\]
Indeed, since $\fD_t^{\prime -1}(\teta_{\bfi', \F}(X'_u))\in \cY'_{t, \le \fD^{\prime -1}\xi'}$, this equality implies
\begin{align*}
\fD_t^{-2}(\tPsi_{\le \xi, \le \xi'}(\teta_{\bfi', \F}(X'_u)))&=\tPsi_{\le \xi, \le \xi'}(\fD_t^{\prime -2}(\teta_{\bfi', \F}(X'_u)))=\fD_t^{-1}(\tPsi_{\le \xi, \le \xi'}(\fD_t^{\prime -1}(\teta_{\bfi', \F}(X'_u)))),
\end{align*}
which is equivalent to $\fD_t^{-1}(\tPsi_{\le \xi, \le \xi'}(\teta_{\bfi', \F}(X'_u)))=\tPsi_{\le \xi, \le \xi'}(\fD_t^{\prime -1}(\teta_{\bfi', \F}(X'_u)))$. 

Write $f\coloneqq (2\ell+1)^-_{\bfi}(\im_u)$ and $f'\coloneqq (2\ell+1)^-_{\bfi'}(\im_u)$. Note that $X_f=Z_{f'}$. Then 
\begin{align*}
	\fD_t^{-2}(\tPsi_{\le \xi, \le \xi'}(\teta_{\bfi', \F}(X'_u)))
 &=\fD_t^{-2}(\teta_{\bfi, \F}(Z_u))\\
&=\fD_t^{-2}(\teta_{\bfi, \F}(X^{\bfg_{u}}F_u))\\
&=\teta_{\bfi, \F}(\tSigma^2(X^{\bfg_{u}}F_u))\\
&=\teta_{\bfi, \F}(X^{\tsfSg^2(\bfg_{u})}F_{u+2\ell}). \\
\tPsi_{\le \xi, \le \xi'}(\fD_t^{\prime -2}(\teta_{\bfi', \F}(X'_u)))
	&=\tPsi_{\le \xi, \le \xi'}(\teta_{\bfi', \F}(\tSigma^{\prime 2}(X'_u)))\\
        &=\tPsi_{\le \xi, \le \xi'}(\teta_{\bfi', \F}(X^{\prime \bfe_{u+2\ell}-\bfe_{f'}}))\\
        &\simeq \teta_{\bfi, \F}(X_{f}^{-1}Z_{u+2\ell})\\
        &\simeq \teta_{\bfi, \F}(X^{\bfg_{u+2\ell}-\bfe_{f}}F_{u+2\ell}).
\end{align*}
By Claim, we have 
\begin{align*}
&\teta_{\bfi, \F}(X^{\tsfSg^2(\bfg_{u})-\bfg_{u+2\ell}+\bfe_{f}})\\
&\simeq \fD_t^{-2}(\tPsi_{\le \xi, \le \xi'}(\teta_{\bfi', \F}(X'_u)))\left(	\tPsi_{\le \xi, \le \xi'}(\fD_t^{\prime -2}(\teta_{\bfi', \F}(X'_u)))\right)^{-1}\in \F(\cY_{t, \le \fD^{-2}\xi}).
\end{align*}
On the other hand, by \eqref{eq:periodicity-g}, 
\[
\tsfSg^2(\bfg_{u})-\bfg_{u+2\ell}+\bfe_{f}\in \sum_{\im\in \Delta_0}\Z\bfe_{(2\ell+1)_{\bfi}^{-}(\im)}+\sum_{\im\in \Delta_0}\Z\bfe_{(\ell+1)_{\bfi}^{-}(\im)}.
\]
Hence $
\teta_{\bfi, \F}(X^{\tsfSg^2(\bfg_{u})-\bfg_{u+2\ell}+\bfe_{f}})$ is a monomial in the variables $\ul{Y_{j,s}^{\pm1}}$ for $(j,s)\in \hI_{\cQ}\cup \hI_{\fD^{-1}\cQ}$. Therefore, 
\[
\teta_{\bfi, \F}(X^{\tsfSg^2(\bfg_{u})-\bfg_{u+2\ell}+\bfe_{f}})=1,
\]
which implies $\tsfSg^2(\bfg_{u})=\bfg_{u+2\ell}-\bfe_{f}$. Therefore, 
\[
\fD_t^{-2}(\tPsi_{\le \xi, \le \xi'}(\teta_{\bfi', \F}(X'_u)))\simeq \tPsi_{\le \xi, \le \xi'}(\fD_t^{-2}(\teta_{\bfi', \F}(X'_u))).
\]
Then, by taking the bar-invariance into account again, we obtain  
\[
\fD_t^{-2}(\tPsi_{\le \xi, \le \xi'}(\teta_{\bfi', \F}(X'_u)))= \tPsi_{\le \xi, \le \xi'}(\fD_t^{-2}(\teta_{\bfi', \F}(X'_u))),
\]
which completes the proof. 
\end{proof}

We are now in the position to define $\tPsi$ in Theorem \ref{Thm:Substitution}. Define the map $\tPsi\colon \F(\cY'_t)\to \F(\cY_t)$ as follows.  
For $y\in \F(\cY'_t)$, there exists $k\in\N_0$ such that $\fD_t^{\prime -k}(y)\in \F(\cY'_{t, \le \xi'})$. Then we set 
\begin{align*}
    \tPsi(y)\coloneqq \fD_t^{k}\left(\tPsi_{\le \xi, \le \xi'}\left(\fD_t^{\prime -k}(y)\right)\right).
\end{align*}
Let us check the well-definedness of $\tPsi$. If $k, k'\in \Z_{\geq 0}$ with $k<k'$ satisfy $\fD_t^{\prime -k}(y), \fD_t^{\prime -k'}(y)\in \F(\cY'_{t, \le \xi'})$, then by Lemma \ref{Lem:compatD} (1), 
\begin{align*}
\fD_t^{k'}\left(\tPsi_{\le \xi, \le \xi'}\left(\fD_t^{\prime -k'}(y)\right)\right)
&=\fD_t^{k'}\left(\tPsi_{\le \xi, \le \xi'}\left(\fD_t^{\prime -(k'-k)}\left(\fD_t^{\prime -k}(y)\right)\right)\right) \allowdisplaybreaks\\
&=\fD_t^{k'}\left(\fD_t^{-(k'-k)}\left(\tPsi_{\le \xi, \le \xi'}\left(\fD_t^{\prime -k}(y)\right)\right)\right) \allowdisplaybreaks\\
&=\fD_t^{k}\left(\tPsi_{\le \xi, \le \xi'}\left(\fD_t^{\prime -k}(y)\right)\right),
\end{align*}
which shows the well-definedness of $\tPsi$. %Note that, in this case, $\fD_t^{\prime -2(k+s)}(y)\in \F(\cY'_{t, \le \fD^{\prime -2}\xi'})$ for all $s\in \Z_{\ge 0}$.  

The property (1) in Theorem \ref{Thm:Substitution} immediately follows from the definition of $\tPsi$. The property (2) in Theorem \ref{Thm:Substitution} follows from Lemma \ref{Lem:compatD} (2). 

Let $y_1, y_2\in \F(\cY'_t)$, and choose $k\in \Z_{\geq 0}$ satisfying $\fD_t^{\prime -k}(y_1), \fD_t^{\prime -k}(y_2)\in \F(\cY'_{t, \le \xi'})$. Then $\fD_t^{\prime -k}(y_1y_2)\in \F(\cY'_{t, \le \xi'})$, and we have 
\begin{align*}
    \tPsi(y_1y_2)&=\fD_t^{k}\left(\tPsi_{\le \xi, \le \xi'}\left(\fD_t^{\prime -k}(y_1y_2)\right)\right)\\
    &=\fD_t^{k}\left(\tPsi_{\le \xi, \le \xi'}\left(\fD_t^{\prime -k}(y_1)\right)\right)\fD_t^{ k}\left(\tPsi_{\le \xi, \le \xi'}\left(\fD_t^{\prime -k}(y_2)\right)\right)\\
    &=\tPsi(y_1)\tPsi(y_2). 
\end{align*}
We can show the linearity of $\tPsi$ in the same way. Therefore, $\tPsi$ is a homomorphism of skew fields. 

Moreover, for $x\in \cK_t(\Cc'_{\Z})$, we can take $k\in \Z_{\geq 0}$ satisfying $\fD_t^{-k}(x)\in \cK_t(\Cc'_{\le \fD^{\prime -1}\xi'})(\subset \F(\cY'_{t, \le \xi'}))$. Then, by the commutativity of the diagram \eqref{eq:diagram} and Theorem \ref{Thm:FHOO} (2), 
\begin{align*}
    \tPsi(x)&=\fD_t^{k}\left(\tPsi_{\le \xi, \le \xi'}\left(\fD_t^{\prime -k}(x)\right)\right)\\
    &=\fD_t^{k}\left(\Psi\left(\fD_t^{\prime -k}(x)\right)\right)=\fD_t^{k}\left(\fD_t^{-k}\left(\Psi(x)\right)\right)=\Psi(x).
\end{align*}
Hence $\tPsi$ satisfies the property in Theorem \ref{Thm:Substitution} (3). 

Finally, by reversing the role of $\fg$ and $\fg'$, we can construct a homomorphism of skew fields $\tPsi'\colon \F(\cY_t)\xrightarrow{\sim} \F(\cY'_t)$ which gives an inverse of $\tPsi$. Therefore $\tPsi$ is an isomorphism, which completes the proof of Theorem \ref{Thm:Substitution}. 

We provide some explicit examples of substitution formulas in Appendix \ref{sec:QCC}.
%%%%%%%%%%%%%%%%%%%%%%%%%%%%%%%%%%%%%%%%%%%%%%%%%%%%%%%%%%%%%%
%%%%%%%%%%%%%%%%%%%%%%%%%%%%%%%%%%%%%%%%%%%%%%%%%%%%%%%%%%%%%%
\appendix

\section{Quantum cluster algebras}
\label{sec:QCA}

In this appendix, we fix our notation around the quantum cluster algebras. 

%%%%%%%%%%%%%%%%%%%%%%%%%%%%%%%%%%%%%%%%%%%%%%%%%%%%%%%%%%%%%%
\subsection{Quantum torus}\label{ssec:Qtorus}

Let $t$ be an invertible indeterminate with a formal square root $t^{1/2}$.
Let $J$ be a (possibly countably infinite) set. 
For a $\Z$-valued skew-symmetric $J\times J$-matrix 
$\Lambda = (\Lambda_{ij})_{i,j \in J}$, we define \emph{the quantum torus} 
$\cT(\Lambda)$ to be the $\Z[t^{\pm 1/2}]$-algebra presented by the set of generators 
$\{ X_{j}^{\pm 1} \mid j \in J\}$ and the relations:
\begin{itemize}
\item $X_j X_j^{-1} = X_j^{-1} X_j  =1$ for $j \in J$, 
\item $X_{i} X_{j} = t^{\Lambda_{ij}} X_j X_i$ for $i,j\in J$.
\end{itemize} 
We define the $\Z$-algebra anti-involution $\ol{(\cdot)}$ of $\cT(\Lambda)$ by
\[
\ol{t^{1/2}} \seq t^{-1/2}, \qquad \ol{X_j} \seq X_j
\]
for all $j \in J$.
This is called the {\em bar involution} of $\cT(\Lambda)$. 
For $\bfa = (a_j)_{j \in J} \in \Z^{\oplus J}$, we define the \emph{commutative monomial}
$$
X^{\bfa} \seq t^{-\frac{1}{2}\sum_{i < j}a_i a_j \Lambda_{ij}} \prod^{\to}_{j \in J} X_j^{a_j},
$$
where we fixed an arbitrary total ordering $<$ of the set $J$. Note that the resulting element $X^{\bfa} \in \cT(\Lambda)$
is independent from the choice of total ordering $<$, and invariant under the bar involution. 
The set $\{X^\bfa \mid \bfa \in \Z^{\oplus J}\}$ forms a free $\Z[t^{\pm 1/2}]$-basis of $\cT(\Lambda)$.
Since $\cT(\Lambda)$ is an Ore domain, it is embedded into the skew field of fractions $\F(\cT(\Lambda))$.

\subsection{Quantum cluster algebra}
\label{ssec:qca}

Let $J_f \subset J$ be a subset and put $J_e \seq J \setminus J_f$.
Let $\tB = (b_{ij})_{i \in J, j \in J_e}$ be a $\Z$-valued $J \times J_e$-matrix whose 
\emph{principal part} $B=(b_{ij})_{i,j \in J_e}$ is skew-symmetrizable, i.e., there is a diagonal matrix $D$ with positive integer entries such that the product $DB$ is skew-symmetric. 
We assume that the set $\{ i \in J \mid b_{ij} \neq 0\}$ is finite for all $j \in J$.
Such a matrix $\tB$ is called \emph{an exchange matrix}. 
We say that a pair $(\Lambda, \tB)$ is \emph{compatible} if we have 
$$
\sum_{k \in J} b_{ki}\Lambda_{kj} = \sfd_i  \delta_{i,j} \quad \text{$(i\in J_e, j \in J)$}
$$
for some positive integer $\sfd_i$. 
In this case, $B$ is skew-symmetrizable by the diagonal matrix $D = \mathrm{diag}(\sfd_i \mid i \in J_e)$.

Given a compatible pair $(\Lambda, \tB)$ and an element $k \in J_e$, we define a new pair 
\[\mu_k (\Lambda, \tB) = (\mu_k\Lambda, \mu_k\tB) \seq (E^{T} \Lambda E, E \tB F), \]
where the matrices $E=(e_{ij})_{i,j \in J}$
and $F=(f_{ij})_{i,j \in J_e}$ are given by 
\begin{equation} \label{eq:EF}
e_{ij} \seq \begin{cases}
\delta_{i,j} & \text{if $j \neq k$}, \\
-1 & \text{if $i=j=k$}, \\
\max(0, -b_{ik}) & \text{if $i\neq j = k$},
\end{cases} 
\qquad 
f_{ij}\seq \begin{cases}
\delta_{i,j} & \text{if $i \neq k$}, \\
-1 & \text{if $i=j=k$}, \\
\max(0, -b_{ik}) & \text{if $i = k \neq j$}.
\end{cases}
\end{equation}
The pair $\mu_k(\Lambda, \tB)$ again forms a compatible pair. 
The operation $\mu_k$ is called the mutation at $k$ and it is involutive, i.e., we have $\mu_k(\mu_k(\Lambda, \tB)) = (\Lambda, \tB)$.
In addition, we define an isomorphism of $\Q(t^{1/2})$-algebras $\mu_k^* \colon \F(\cT(\mu_k\Lambda)) \cong \F(\cT(\Lambda))$ by
\begin{equation}
\mu_k^*(X_{j}) \seq 
\begin{cases} \label{eq:exrel}
X^{\bfa^\prime} + X^{\bfa^{\prime\prime}} & \text{if $j =k$}, \\
X_{j} & \text{if $j \neq k$},
\end{cases}
\end{equation}
where $\bfa^\prime = (a^\prime_{j})_{j \in J}$
 and $\bfa^\prime = (a^{\prime\prime}_{j})_{j \in J}$ are given by
\[
a^\prime_{j} \seq \begin{cases}
-1 & \text{if $i=k$}, \\
\max(0, b_{ik}) & \text{if $i \neq k$},
\end{cases}
\qquad 
a^{\prime\prime}_{j} \seq \begin{cases}
-1 & \text{if $i=k$}, \\
\max(0, -b_{ik}) & \text{if $i \neq k$}.
\end{cases}
\]
Note that the definition of $\mu_k^*$ does depend on the exchange matrix $\tB$.
The isomorphism $\mu_k^*$ is often called the cluster transformation at $k$.
We have $\mu_k^* \circ \mu_k^* = \id_{\F(\cT(\Lambda))}$ and $\mu_k^* \circ \ol{(\cdot)} = \ol{(\cdot)} \circ \mu_k^*$.

\begin{Def} 
Let $(\Lambda, \tB)$ be a compatible pair.  
We say that an element of $\F(\cT(\Lambda))$ is a \emph{quantum cluster variable} (resp.~\emph{quantum cluster monomial}) if it is written as
\[ \mu_{k_1}^*\mu_{k_2}^*\cdots \mu_{k_n}^*(X_j) \quad (\text{resp. } \mu_{k_1}^*\mu_{k_2}^*\cdots \mu_{k_n}^*(X^\bfa)), \]
for some finite sequence $(k_1, k_2, \ldots, k_n)$ in $J_e$ and $j \in J$ (resp.~$\bfa \in \N_0^{\oplus J}$).
The \emph{quantum cluster algebra} $\cA_{t}(\Lambda, \tB)$
is defined to be the $\Z[t^{\pm 1/2}]$-subalgebra of $\F(\cT(\Lambda))$
generated by all the quantum cluster variables.
\end{Def}    

Note that each quantum cluster monomial is invariant under the bar involution.
 
\begin{Thm}[The quantum Laurent phenomenon {\cite[Corollary 5.2]{BZ05}}]
The quantum cluster algebra $\cA_{t}(\Lambda, \tB)$ is contained in the quantum torus $\cT(\Lambda)$.
\end{Thm}

By definition, the cluster transformation at $k \in J_e$ gives a $\Z[t^{\pm 1/2}]$-algebra isomorphism
\[ \mu_k^* \colon \cA_t(\mu_k(\Lambda, \tB)) \cong \cA_t(\Lambda, \tB), \]
which induces a bijection between the sets of quantum cluster monomials.

\subsection{Permutation} \label{Ssec:perm}
Let $\pi$ be a permutation of the index set $J$ satisfying $\pi(J_f) \subset J_f$.
Given a compatible pair $(\Lambda, \tB)$ as above, we can consider another compatible pair $\pi (\Lambda, \tB) = (\pi \Lambda, \pi \tB)$, where $\pi \Lambda \seq (\Lambda_{\pi^{-1}(i),\pi^{-1}(j)})_{i,j \in J}$ and $\pi \tB \seq (b_{\pi^{-1}(i),\pi^{-1}(j)})_{i \in J, j \in J_e}$.  
Then we have the isomorphism of $\Z[t^{\pm 1/2}]$-algebras
\[\pi^* \colon \cA_{t}(\pi(\Lambda, \tB)) \cong \cA_{t}(\Lambda, \tB) \qquad \text{given by $\pi^*(X_j) = X_{\pi^{-1}(j)}$ for all $j \in J$.} \]
Clearly, this isomorphism $\pi^*$ induces a bijection between the sets of quantum cluster monomials.

\subsection{Classical limit}
Let $\tB$ be an exchange matrix as above. 
Consider the (commutative) ring of Laurent polynomials $\Z[X_j^{\pm 1} \mid j \in J]$ and let $\Q(X_j \mid j \in J)$ be its fraction field. 
For each $k \in J_e$, we have an algebra involution $\mu_k^*$ of $\Q(X_j \mid j \in J)$ defined by the same formula as \eqref{eq:exrel}.

\begin{Def}
Let $\tB$ be an exchange matrix.
An element of $\Q(X_j \mid j \in J)$ is called  a \emph{cluster variable} (resp.~\emph{cluster monomial}) if it is written as
\[ \mu_{k_1}^*\mu_{k_2}^*\cdots \mu_{k_n}^*(X_j) \quad (\text{resp. } \mu_{k_1}^*\mu_{k_2}^*\cdots \mu_{k_n}^*(X^\bfa)), \]
for some finite sequence $(k_1, k_2, \ldots, k_n)$ in $J_e$ and $j \in J$ (resp.~$\bfa \in \N_0^{\oplus J}$).
The \emph{cluster algebra} $\cA(\tB)$
is defined to be the $\Z$-subalgebra of $\Q(X_j \mid j \in J)$ generated by all the cluster variables.
\end{Def}

Now, let $(\Lambda, \tB)$ be a compatible pair. 
By specializing $t^{1/2}$ to $1$, we obtain the surjective algebra homomorphism $\evt \colon \cT(\Lambda) \to \Z[X_j^{\pm 1}\mid j \in J]$.
By definition, it induces the surjection
\[ \evt \colon \cA_t(\Lambda, \tB) \to \cA(\tB).\]
under which a quantum cluster monomial goes to a cluster monomial.
See \cite[Lemma 3.3]{GLS20}. 

\begin{Lem}[cf.~{\cite{BZ05}}]
\label{Lem:evt}
The homomorphism $\evt$ gives a bijection between the set of quantum cluster monomials in $\cA_t(\Lambda, \tB)$ and the set of cluster monomials in $\cA(\tB)$.
\end{Lem}
\begin{proof}
The surjectivity is obvious from the definition.
The injectivity follows from an argument similar to the proof of \cite[Theorem 6.1]{BZ05}. 
\end{proof} 

\subsection{Degrees} 
Let $(\Lambda, \tB)$ be a compatible pair. We say that an element $x$ of the quantum torus $\cT(\Lambda)$ is \emph{pointed} if it is written in the form
\[x = X^{\bfg} + \sum_{\bfn \in \N_0^{\oplus J_e} \setminus\{0\}} c_{\bfn} X^{\bfg + \tB \bfn}\]
for some $\bfg \in \Z^{\oplus J}$ and $c_\bfn \in \Z[t^{\pm 1/2}]$.
In this case, we write $\deg x = \bfg$ and call it \emph{the degree of $x$}. 
Note that this notion of degree does depend on the exchange matrix $\tB$, not only on the quantum torus $\cT(\Lambda)$.

It is known that every quantum cluster monomial $x$ in $\cA_t(\Lambda, \tB)$ is pointed~\cite[Theorem 5.3]{Tran}. 
Its degree is often called the $g$-vector of $x$. 

\begin{Thm}[{\cite[(7.18)]{FZ07}, \cite{DWZ10}, \cite{GHKK}}] 
\label{Thm:trop}
Let $(\Lambda, \tB)$ be a compatible pair and $k \in J_e$. 
Let $x \in \cA_t(\Lambda, \tB)$ and $x' \in \cA_t(\mu_k(\Lambda, \tB))$ be quantum cluster monomials whose $g$-vectors are $\bfg = (g_j)_{j \in J}$ and $\bfg' = (g'_j)_{j \in J}$ respectively. 
If $\mu_k^* (x') = x$, we have
\[ 
g_j = \begin{cases}
- g'_k & \text{if $j=k$}, \\
g'_j + \max(b'_{jk},0)g'_k & \text{if $j \neq k$ and $g'_k \ge 0$}, \\
g'_j - \min(b'_{jk},0)g'_k & \text{if $j \neq k$ and $g'_k \le 0$}.  
\end{cases}
\]
Here $\mu_k\tB = (b'_{ij})_{i \in J, j \in J_e}$.
\end{Thm}

\begin{Rem} \label{Rem:trop}
In particular, under the assumption of Theorem~\ref{Thm:trop}, we have $\bfg = E \bfg'$ if $g'_k \ge 0$.
Here $E = (e_{ij})_{i,j \in J}$ is the matrix given by \eqref{eq:EF} with $\tB = (b_{ij})_{i \in J, j \in J_e}$.
\end{Rem}

\begin{Rem} \label{Rem:permg}
Let $\pi \in \SG_J$ be a permutation such that $\pi (J_f) \subset J_f$ as in Section~\ref{Ssec:perm}.
From the definition, it is clear that we have $\deg (\pi^*x) = (g_{\pi(j)})_{j \in J}$ if $\deg x = (g_j)_{j \in J}$.
\end{Rem}

\begin{Thm}[{\cite[Conjecture 7.10(1)]{FZ07}, \cite{DWZ10}, \cite{GHKK}}] \label{Thm:gvsx}
Let $(\Lambda, \tB)$ be a compatible pair and $x, x' \in \cA_t(\Lambda, \tB)$ be two cluster monomials.
If $\deg x = \deg x'$, we have $x=x'$.
\end{Thm}
\begin{proof}
Note that we have the similar notion of degrees for classical cluster algebra $\cA(\tB)$ \cite{FZ07} and we have $\deg x = \deg \evt(x)$ for any quantum cluster monomial $x$ thanks to~ \cite[Theorem 5.3(2)]{Tran}.
For the classical cluster algebra $\cA(\tB)$, the assertion is proved by \cite{DWZ10} when $B$ is skew-symmetric and by \cite{GHKK} in general.
Therefore, the assertion for the quantum cluster algebra $\cA_t(\Lambda, \tB)$ also follows in view of Lemma~\ref{Lem:evt}. 
\end{proof}

\begin{Rem}
In the main body of the present paper, we consider the degrees only for skew-symmetric quantum cluster algebras. 
Therefore, we do not really need the results of \cite{GHKK} for our purpose. 
\end{Rem}

\section{Cluster structure on $K(\Cc_\Z)$}
\label{sec:QCB}

We give examples for Theorem \ref{Thm:KKOP_CZ}.

\begin{Ex}[Type $\mathrm{A}$] \label{Ex: An quiver periodicity}
Let  $\cQ=(\Delta_{\mathrm{A}_n}, {\rm id},\xi)$ be a Q-datum of type $\mathrm{A}_n$ such that
$$  \xi_\im = \begin{cases} 
 0  & \text{ if } \im \equiv 1  \pmod 2, \\
 -1  & \text{ if } \im \equiv 0  \pmod 2,
\end{cases}
 $$
and take a
$\cQ$-adapted sequence
$$
\bfi = (\im_u)_{u \in \N}\seq \begin{cases}
(1,3,\ldots,n,2,4,\ldots,n-1, 1,3,\ldots,n,2,4,\ldots,n-1,\ldots) & \text{ if } n \equiv 1 \pmod 2, \\
(1,3,\ldots,n-1,2,4,\ldots,n, 1,3,\ldots,n-1,2,4,\ldots,n,\ldots) & \text{ if } n \equiv 0 \pmod 2.
\end{cases}
$$

The map $\fe$ defined below is a $\cQ$-adapted enumeration of $ \hD_{[\xi]}$: For $u=kn+s$ with $k \in \Z_{\ge0}$ and $0< s \le n$, set
$$
\fe(u) \seq 
\begin{cases}
 (\im_s,\xi_{\im_s}+k) & {\rm (i)} \    k\equiv  0   \pmod 2 \text{ and } 0< s \le \lceil n/2 \rceil, \\
 (\im_s,\xi_{\im_s}-k)      & {\rm (ii)} \    k\equiv  0    \pmod 2 \text{ and }  \lceil n/2 \rceil< s \le n, \\
 (\im_s,\xi_{\im_s}-  (k+1) ) & {\rm (iii)} \     k\equiv  1   \pmod 2 \text{ and } 0< s \le \lceil n/2 \rceil, \\
 (\im_s,\xi_{\im_s}+k+1)      & {\rm (iv)} \    k\equiv 1   \pmod 2  \text{ and }  \lceil n/2 \rceil< s \le n.
\end{cases}
$$
Then each $m_u^\fe$  in~\eqref{eq: m_u} becomes 
$$
m_u^\fe \seq
\begin{cases}
m^{(\im_s)}[\xi_{\im_s}- k,  \xi_{\im_s}+k]  &  {\rm (i)}, \\
m^{(\im_s)}[\xi_{\im_s}- k,  \xi_{\im_s}+k]  &  {\rm (ii)},  \\
m^{(\im_s)}[\xi_{\im_s}-  (k+1) ,  \xi_{\im_s}+k-1]  &  {\rm (iii)}, \\
m^{(\im_s)}[\xi_{\im_s}- k+1,  \xi_{\im_s}+k+1]  &  {\rm (iv)},
\end{cases}
$$
and 
the exchange matrix $\tB^\fe$ is given as follows (see \cite[Theorem 6.14]{KKOP22}):
$$
(\tB^\fe)_{u,v} 
\seq 
\begin{cases}
  1 & \text{ if }  {\rm (a)} \ u-v=\pm n  \text{ and $u$ satisfies {\rm (i)} or  {\rm (iv)}},\\
& \quad \text{ or }  {\rm (b)} \ |u-v| <n, \ \im_u \sim \im_v \text{ and    $u$ satisfies {\rm (iii)} or  {\rm (ii)}}, \\
  -1 & \text{ if }  {\rm (c)} \ u-v=\pm n  \text{ and $u$ satisfies {\rm (iii)} or  {\rm (ii)}}, \\
&\quad \text{ or }  {\rm (d)} \ |u-v| <n, \ \im_u \sim \im_v \text{ and   $u$ satisfies {\rm (i)} or  {\rm (iv)}}, \\
0 & \text{ otherwise.}
\end{cases}
$$
Using $\{ m_u^\fe \}_{u \in \N}$ as the set of vertices for the corresponding quiver $\Gamma_{\tB^\fe}$ of $\tB^\fe$, $\Gamma_{\tB^\fe}$ is isomorphic to the 
the product of sink-source Dynkin quivers of type $\mathrm{A}_\infty$ and $\mathrm{A}_n$,
usually denoted by $\overset{\to}{\mathrm{A}}_{\infty} \otimes \overset{\to}{\mathrm{A}_n}$ (see \cite[\S 3.3]{Ke13}). For instance,  $\Gamma_{\tB^\fe}$ for $n=3$ can be drawn as follows:
\begin{equation*}
\scalebox{0.8}{
\xymatrix@!C=14mm@R=6mm{
\cdots \ar[rr]&&   m^{(1)}_7[-2,2]   \ar[dl]    &&  m^{(1)}_4[-2,0]  \ar[rr]\ar[ll]    &&  m^{(1)}_1[0,0]   \ar[dl]    \\
\cdots &  m^{(2)}_9[-3,1] \ar[rr]  \ar[l]   && m^{(2)}_6[-1,1] \ar[ur]  \ar[dr] &&  m^{(2)}_3[-1,-1]   \ar[ll]  \\
\cdots \ar[rr] &&  m^{(3)}_8[-2,2] \ar[ul]    && m^{(3)}_5[-2,0]    \ar[rr]\ar[ll]    &&  m_2^{(3)}[0,0]      \ar[ul]
}}
\end{equation*}
Here $m_u^{(\im)}[p,s]$ means $m^\fe_u = m^{(\im)}[p,s]$.
%Note that the quiver above is known as the product of sink-source Dynkin quivers of type %$A_\infty$ and $A_3$, usually denoted by $\overset{\to}{A}_{\infty} \otimes \overset{\to}{A}_3$ 
%(see \cite[\S 3.3]{Ke13}). 
\end{Ex}

\begin{Ex}[Type $\mathrm{B}$] \label{Ex: Bn quiver periodicity}
Let $\cQ=(\Delta_{\mathrm{A}_{2n-1}}, \vee, \xi )$ be a Q-datum of type $\mathrm{B}_n$ such that
\begin{align*}
\xi_{n-s}   = \begin{cases}
\xi_{n} + 1 & \text{ if } n- s \equiv 0 \pmod 2, \\
\xi_{n} + 3 & \text{ if } n-s \equiv 1 \pmod 2,
\end{cases} \qquad  \xi_{2n-s}   = \xi_s-2 \ \text{ for } 1 \le s \le n-1,
\end{align*}
and take $\cQ$-adapted sequence
$$
\bfi \seq   (\im_1,\im_2,\ldots,\im_n,\im_{n+1},\im_{n+2},\im_{2n-1},\im_{2n},\im_{2n+1},\ldots)
$$
such that  
(1) $\im_k = \im_{k-2n}$ for any $k > 2n$, 
(2)  $\{ \im_1, \ldots,\im_n \} = \{ 1,\ldots, n\}$, (3) $\im_n=n$, (4) $\im_{n+k}= \vee(\im_k)$ for $1 \le k \le n-1$ and (5) $ \xi_{\im_1} \ge \xi_{\im_2} \ge \cdots \ge \xi_{\im_n}$ (recall Subsection~\ref{Ssec:Cox}).
Then the map $\fe$ defined below is a $\cQ$-adapted enumeration of $ \hD_{[\xi]}$:  For $u=2kn+s \in \N$ with $k \in \Z_{\ge0}$ and $0< s < n$, set
$$
\fe(u) \seq 
\begin{cases}
 (\im_s,\xi_{\im_s}-2k) & {\rm (i)} \    k\equiv  0   \pmod 2     \text{ and } \xi_n-\xi_{\im_s} \equiv 3 \pmod 4,   \\
 (\im_s,\xi_{\im_s}+2k) & {\rm (ii)} \    k\equiv  0   \pmod 2  \text{ and } \xi_n-\xi_{\im_s} \equiv 1 \pmod 4, \\

 (\im_s,\xi_{\im_s}+2(k+1)) & {\rm (iii)} \    k\equiv  1   \pmod 2 \text{ and } \xi_n - \xi_{\im_s} \equiv 3 \pmod 4, \\
 (\im_s,\xi_{\im_s}-2(k+1)) & {\rm (iv)} \    k\equiv  1   \pmod 2 \text{ and } \xi_n - \xi_{\im_s} \equiv 1 \pmod 4,
\end{cases}
$$
and, for $u=tn \in \N$ with $t \in \N$, set
$$
\fe(u) \seq 
\begin{cases}
(n, \xi_n + t)  & {\rm (v)} \  t \equiv 0 \pmod 2, \\ 
(n, \xi_n - (t-1))  & {\rm (vi)} \  t \equiv 1 \pmod 2.
\end{cases}
$$
Then each $m_u^\fe$  in~\eqref{eq: m_u} becomes as follows:
$$
m_u^\fe \seq
\begin{cases}
m^{(\im_s)}[\xi_{\im_s}- 2k,  \xi_{\im_s}+2k]  &  {\rm (i)} \text{ or } {\rm (ii)}, \\
%m^{(\im_s)}[\xi_{\im_s}- 2k,  \xi_{\im_s}+2k]  &  ,  \\
m^{(\im_s)}[\xi_{\im_s}-  2(k-1) ,  \xi_{\im_s}+2(k+1)]  &  {\rm (iii)}, \\
m^{(\im_s)}[\xi_{\im_s}- 2(k+1),  \xi_{\im_s}+2(k-1)]  &  {\rm (iv)}, \\
m^{(n)}[\xi_{\im_n}- (t-2),  \xi_{n}+t]  &  {\rm (v)},  \\
m^{(n)}[\xi_{\im_n}- (t-1) ,  \xi_{n}+t]  &  {\rm (vi)}.  
\end{cases}
$$
In this case, one can prove that the following skew-symmetric matrix is $\tB^\fe$ in Theorem~\ref{Thm:KKOP_CZ}: For $u,v \in \N$ such that $\{ \im_u,\im_v \} \ne \{ n, n \pm 1\}$,  
$$
(\tB^\fe)_{u,v} 
\seq 
\begin{cases}
  1 & \text{ if }  {\rm (a)} \ u-v=\pm  2\sfd_{\bar{\im}_u}n  \text{ and $u$ satisfies {\rm (ii)},  {\rm (iii)} or {\rm (v)}},\\
& \quad \text{ or }  {\rm (b)} \ |u-v| <  2\sfd_{\bar{\im}_u}n, \  \im_u \sim \im_v \text{ and    $u$ satisfies {\rm (i)}, {\rm (iv)} or  {\rm (vi)}}, \\
  -1 & \text{ if }  {\rm (c)} \ u-v=\pm 2\sfd_{\bar{\im}_u}n  \text{ and $u$ satisfies {\rm (i)}, {\rm (iv)} or {\rm (vi)}}, \\
&\quad \text{ or }  {\rm (d)} \ |u-v| < 2\sfd_{\bar{\im}_u}n, \ \im_u \sim \im_v \text{ and   $u$ satisfies {\rm (ii)}, {\rm (iii)} or  {\rm (v)}}, \\
0 & \text{ otherwise.}
\end{cases}
$$
and the remained entries $-(\tB^\fe)_{v,u} =(\tB^\fe)_{u,v} $ for $\im_u = n \pm 1$  and $\im_v = n$ 
are given as follows:  
$$
(\tB^\fe)_{u,v} 
\seq 
\begin{cases}
 1 & \text{ if } {\rm (g)} \ |u-v| < n, \    \text{and $v$ satisfies {\rm (v)}},   \\
  -1 & \text{ if } {\rm (e)} \ |u-v| < 3n, \  \im_u=n-1,  \text{ $u$ satisfies {\rm (ii)}  and $v$ satsifies {\rm (vi)}},   \\
  & \ \text{ or } {\rm (f)} \ |u-v| < 3n, \  \im_u=n+1,  \text{ $u$ satisfies {\rm (iv)}  and $v$ satsifies {\rm (vi)}},   \\
0 & \text{ otherwise.}
\end{cases}
$$

Interestingly enough, the quiver $\Gamma_{\tB^\fe}$ corresponding to $\tB^\fe$ in the above is isomorphic to 
the quiver $Q_{\infty}(\mathrm{B}_{n})$ in \cite[Figure 1]{IIKKKB} which is related to the periodicity
of cluster algebras.

\smallskip

Let us see the particular example for $n=3$. In this case, we can take $\xi_3=-3$ and 
$$ \bfi = (1,2,3,5,4,3,1,2,3,5,4,3,\ldots).$$ 
Using $\{ m_u^\fe \}_{u \in \N}$ as the set of vertices for the corresponding quiver $\Gamma_{\tB^\fe}$ of $\tB^\fe$, $\Gamma_{\tB^\fe}$ can be drawn as follows:
\begin{align}\label{eq: tB fe}
 \raisebox{11ex}{ \scalebox{0.58}{ \xymatrix@!C=7mm@R=3mm{
\cdots \ar[r]&m^{(1)}_{19}[-8,4]    &&&&  m^{(1)}_{13}[-8,0]     \ar[rrrr]  \ar[llll]   &&&&  m^{(1)}_7[-4,0]       \ar[ddll] &&&&   m^{(1)}_1[-4,-4]   \ar[llll]    \\ \\
\cdots \ar[rrr]& &&m^{(2)}_{14}[-6,2]    \ar[uurr]  \ar[dr] \ar[dlll]  &&&&   m^{(2)}_8[-6,-2]      \ar[llll]\ar[rrrr]    &&&&   m^{(2)}_2[-2,-2]      \ar[dlll]  \ar[uurr] \ar[dr] \\
\cdots   \ar[rr]&&m^{(3)}_{18}[-7,3]   \ar[ur]\ar[dl]&&   m^{(3)}_{15}[-7,1]   \ar[ll] \ar[rr] &&  m^{(3)}_{12}[-5,1]      \ar[dl] \ar[ur]  &&    m^{(3)}_9[-5,-1]   \ar[ll]  \ar[rr]  &&  m^{(3)}_6[-3,-1]    \ar[dl]   \ar[ur]  &&     m^{(3)}_3[-3,-3]         \ar[ll] \\
\cdots&  m^{(4)}_{17}[-8,0]   \ar[l] \ar[rrrr] &&&&  m^{(4)}_{11}[-4,0]    \ar[ul]\ar[ddrr] \ar[urrr]    \ &&&&   m^{(4)}_5[-4,-4]       \ar[ddrr]    \ar[llll]    \\ \\
&\cdots  \ar[rr] &&  m^{(5)}_{16}[-6,2]   \ar[uull]&&&& m^{(5)}_{10}[-6,-2]       \ar[rrrr]\ar[llll]  &&&&   m^{(5)}_4[-2,-2]       }}}
\end{align}
%Here $m_u^{(\im)}[p,s]$ means $m^\fe_u = m^{(\im)}[p,s]$. Note that the quiver $\Gamma_{\tB^\fe}$ in~\eqref{eq: tB fe} is isomorphic to the quiver $Q_{\infty}(B_{3})$ in \cite[Figure 1]{IIKKKB}.
\end{Ex}

\begin{Rem} 
As Example~\ref{Ex: An quiver periodicity}, there exist admissible chains of $i$-boxes and 
Q-data with $\sigma={\rm id}$, whose corresponding quivers are isomorphic to the product quivers $\overset{\to}{\mathrm{A}}_{\infty} \otimes \overset{\to}{\mathrm{D}}_n$ $(n \ge 4)$ and $\overset{\to}{\mathrm{A}}_{\infty} \otimes \overset{\to}{\mathrm{E}}_n$ $(n=6,7,8)$ in \cite{Ke13}, for simply-laced types $\fg$. Similar to Example~\ref{Ex: Bn quiver periodicity}, for non simply-laced type $\fg$, there exist admissible chains of i-boxes and 
Q-data with $\sigma \ne {\rm id}$, whose corresponding quivers are isomorphic to $Q_\infty(\fg)$
in \cite{IIKKKB,IIKKKB2}. 
\end{Rem}

\section{Example of substitution formulas}\label{sec:QCC}
In this appendix, we show some examples of our substitution formulas. We mainly consider the situation when $t=1$ for simplicity, that is, the morphism $\tPsi_{t=1}$. Its quantum analogue $\tPsi$ can be calculated in a parallel manner. 
\begin{Ex} \label{EX: A2 to A2} 
Let us calculate explicitly the substitution formula from type $\mathrm{A}_{2}$ to type $\mathrm{A}_2$ itself arising from a braid move. Consider the height functions $\xi', \xi$ given as follows:
\[
(\xi'_1, \xi'_2)=(0,-1), \qquad
(\xi_1,\xi_2)=(0,1).
\]
%\begin{align*}
%\xi' &: \{ 1,2 \} \to \Z, \quad 1 \mapsto 0, \quad 2 \mapsto -1  \quad \text{ and } \quad 
%\xi : \{ 1,2 \} \to \Z, \quad 1 \mapsto 0, \quad 2 \mapsto 1.
%\end{align*}
Then the followings are the infinite sequences $\bfi' = (\im'_u)_{u \in \N}, \bfi = (\im_u)_{u \in \N}\in \{1,2\}^{\N}$ satisfying the condition~\eqref{eq:cond2} in Example~\ref{Ex:periodic}:
\begin{align*}
    &\bfi' =( \hspace{-3ex}\underset{\text{\scalebox{0.8}{reduced word for $w_{\circ}$}}}{\underbrace{  1, 2, 1},} \hspace{-3ex} \underbrace{2,1,2}, \underbrace{1,2,1},\dots), \qquad     \bfi=(\hspace{-3ex} \underset{\text{\scalebox{0.8}{reduced word for $w_{\circ}$}}}{\underbrace{2,1, 2},} \hspace{-3ex} \underbrace{1,2,1}, \underbrace{2,1,2},\dots).
\end{align*}
The quiver $\Gamma_\bfi$ corresponding to $\bfi$ can be depicted as
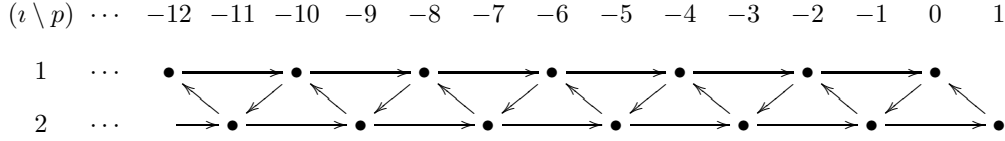
\[
\raisebox{3mm}{
\scalebox{0.9}{\xymatrix@!C=1mm@R=3mm{
(\im\setminus p) &\cdots&-12&-11&-10&-9& -8 & -7 & -6 &-5&-4 &-3& -2 &-1& 0 & 1\\
1&\cdots& \bullet \ar@{<-}[dr]&&\bullet \ar@{<-}[ll]\ar@{<-}[dr]&&\bullet \ar@{<-}[ll]\ar@{<-}[dr]&& \bullet \ar@{<-}[ll]\ar@{<-}[dr] &&\bullet \ar@{<-}[ll]\ar@{<-}[dr]
&&\bullet \ar@{<-}[ll]\ar@{<-}[dr] && \bullet\ar@{<-}[ll]\ar@{<-}[dr] &\\
2&\cdots&& \bullet \ar@{<-}[l]\ar@{<-}[ur]&&\bullet \ar@{<-}[ll]\ar@{<-}[ur]&& \bullet \ar@{<-}[ll]\ar@{<-}[ur]&& \bullet  \ar@{<-}[ll]\ar@{<-}[ur] &&\bullet \ar@{<-}[ll]\ar@{<-}[ur]
&& \ar@{<-}[ll] \bullet \ar@{<-}[ur]&& \bullet \ar@{<-}[ll]
}}}
\]
The infinite sequence $\bfi$ can be obtained from $\bfi'$, and vise versa, by applying the braid moves. Hence applying the mutations at the following vertices
\begin{equation}
(2, 1),  (1,-2), (2,-5), (1, -8),(2,-11), \ldots \label{eq:A2toA2}
\end{equation}
%from left to right to $\Gamma_{\bfi}$, 
we obtain the quiver $\Gamma_{\bfi'}$.
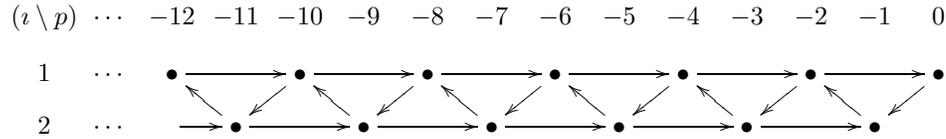
\[
\raisebox{3mm}{
\scalebox{0.9}{\xymatrix@!C=1mm@R=3mm{
(\im\setminus p) &\cdots&-12&-11&-10&-9& -8 & -7 & -6 &-5&-4 &-3& -2 &-1& 0  \\
1&\cdots&\bullet \ar@{<-}[dr]&&\bullet\ar@{<-}[ll]\ar@{<-}[dr]&&\bullet \ar@{<-}[ll]\ar@{<-}[dr]&& \bullet \ar@{<-}[ll]\ar@{<-}[dr] &&\bullet\ar@{<-}[ll]\ar@{<-}[dr]
&&\bullet\ar@{<-}[ll]\ar@{<-}[dr] &&\bullet \ar@{<-}[ll]  &\\
2&\cdots&&\bullet \ar@{<-}[l]\ar@{<-}[ur]&&\bullet \ar@{<-}[ll]\ar@{<-}[ur]&&\bullet \ar@{<-}[ll]\ar@{<-}[ur]&& \bullet \ar@{<-}[ll]\ar@{<-}[ur] &&\bullet\ar@{<-}[ll]\ar@{<-}[ur]
&& \ar@{<-}[ll] \bullet \ar@{<-}[ur]  
}}}
\]
Here the correspondence of the labelling of vertices is given by
\begin{equation} \label{eq:vertexcorrep3}
\begin{aligned}
(2,1-6m) &\mapsto (1,-6m), \ \qquad\qquad (1,-6m) \mapsto (1,-2-6m),\allowdisplaybreaks\\
 (2,-1-6m) &\mapsto (2,-1-6m),   \quad (1,-2-6m) \mapsto (2,-3-6m), \allowdisplaybreaks\\
 (2,-3-6m) &\mapsto (2,-5-6m),\quad (1,-4-6m) \mapsto (1,-4-6m).
\end{aligned}
\end{equation}
By the exchange relation of cluster algebra and \eqref{eq:vertexcorrep3}, we have 
\begin{align*}
\hat{\bftau}^*(X'_{1,-6m})&=\dfrac{X_{2,3-6m}X_{1,-6m}+X_{1,2-6m}X_{2,-1-6m}}{X_{2,1-6m}},  \allowdisplaybreaks\\
\hat{\bftau}^*(X'_{1,-2-6m})&=X_{1,-6m}, \qquad \qquad \\
\hat{\bftau}^*(X'_{1,-4-6m})&=X_{1,-4-6m}, \allowdisplaybreaks\\
\hat{\bftau}^*(X'_{2,-1-6m})&=X_{2,-1-6m}, \qquad \quad \\
\hat{\bftau}^*(X'_{2,-3-6m})&=\dfrac{ X_{1,-6m}X_{2,-3-6m} + X_{2,-1-6m}X_{1,-4-6m}}{ X_{1,-2-6m}},\\
\hat{\bftau}^*(X'_{2,-5-6m})&=X_{2,-3-6m}. \allowdisplaybreaks
\end{align*}
Here $X_{2,3}=X_{1,2}\coloneqq 1$. 
Hence,
\begin{align*}
&\tPsi_{\le \xi, \le \xi'}(m^{(\im)}[p, \xi'_{\im}]) \\ &  = \begin{cases}
(Y_{2,1-6m}^{-1}Y_{1,-6m} + Y_{2,-1-6m})m^{(1)}[2-6m,0]  & \text{ if } (\im,p) = (1, -6m), \\
m^{(1)}[-6m,0] & \text{ if } (\im,p) = (1, -2-6m), \\
m^{(1)}[-4-6m,0] & \text{ if } (\im,p) = (1, -4-6m), \\
m^{(2)}[-1-6m,1]   & \text{ if } (\im,p) = (2, -1-6m), \\
(Y_{1,-2-6m}^{-1}Y_{2,-3-6m}+  Y_{1,-4-6m})m^{(2)}[-1-6m,1] & \text{ if } (\im,p) = (2, -3-6m), \\
m^{(2)}[-3-6m,1]  & \text{ if } (\im,p) = (2, -5-6m).
\end{cases}
\end{align*}
Here $m^{(1)}[2,0]\coloneqq 1$. 
This calculation implies that, for $(\im, p)\in \hI'_{\le \xi'}$,
\begin{align*}
\tPsi_{\le \xi, \le \xi'}(Y_{\im, p})
&=\tPsi_{\le \xi, \le \xi'}(m^{(\im)}[p, \xi'_{\im}])\tPsi_{\le \xi, \le \xi'}(m^{(\im)}[p+2, \xi'_{\im}])^{-1} \\
& = \begin{cases}
Y_{2,1-6m}^{-1}Y_{1,-6m} + Y_{2,-1-6m}  & \text{ if } (\im,p) = (1, -6m), \\
\dfrac{1}{Y_{2,1-6m}^{-1} +Y_{2,-1-6m}Y_{1,-6m}^{-1} }  & \text{ if } (\im,p) = (1, -2-6m), \\
Y_{1,-4-6m}Y_{1,-2-6m} & \text{ if } (\im,p) = (1, -4-6m), \\
 Y_{2,-1-6m}Y_{2,1-6m}     & \text{ if } (\im,p) = (2, -1-6m), \\
Y_{1,-2-6m}^{-1}Y_{2,-3-6m}+  Y_{1,-4-6m} & \text{ if } (\im,p) = (2, -3-6m), \\
\dfrac{1}{Y_{1,-2-6m}^{-1}+  Y_{1,-4-6m}Y_{2,-3-6m}^{-1}} & \text{ if } (\im,p) = (2, -5-6m).
\end{cases}
\end{align*}
Therefore, for $(\im, p)\in \hI'$,
\begin{align*}
\tPsi_{t=1}(Y_{\im, p}) = \begin{cases}
Y_{2,1-6m}^{-1}Y_{1,-6m} + Y_{2,-1-6m}  & \text{ if } (\im,p) = (1, -6m), \\
\dfrac{1}{Y_{2,1-6m}^{-1} +Y_{2,-1-6m}Y_{1,-6m}^{-1} }  & \text{ if } (\im,p) = (1, -2-6m), \\
Y_{1,-4-6m}Y_{1,-2-6m} & \text{ if } (\im,p) = (1, -4-6m), \\
 Y_{2,-1-6m}Y_{2,1-6m}     & \text{ if } (\im,p) = (2, -1-6m), \\
Y_{1,-2-6m}^{-1}Y_{2,-3-6m}+  Y_{1,-4-6m} & \text{ if } (\im,p) = (2, -3-6m), \\
\dfrac{1}{Y_{1,-2-6m}^{-1}+  Y_{1,-4-6m}Y_{2,-3-6m}^{-1}} & \text{ if } (\im,p) = (2, -5-6m).
\end{cases}
\end{align*}
This is the substitution formula from $A_2$ to itself.

For an instance, we have
$$
\chi_q(L(Y_{2,-7})) = Y_{2,-7} + Y_{1,-6} Y_{2,-5}^{-1} +  Y_{1,-4}^{-1}.
$$
By applying the above formula, we have
\begin{align*}
& Y_{2,-7}Y_{2,-5} +     Y_{2,-5}^{-1}Y_{1,-6}Y_{1,-2}^{-1} + Y_{2,-7}Y_{1,-2}^{-1} \\
& \hspace{5ex} + Y_{2,-5}^{-1}Y_{1,-6} Y_{1,-4}Y_{2,-3}^{-1}+ Y_{2,-7}Y_{1,-4}Y_{2,-3}^{-1}  +  Y_{1,-4}^{-1}Y_{1,-2}^{-1}
= \chi_q(L(Y_{2,-7}Y_{2,-5})).
\end{align*}
We remark here that the isomorphism $\Psi_{t=1}$ is categorified by the autofunctor 
$\mathscr{S}_2$ on $\Cc_{\Z,\mathfrak{sl}_3}$ in \cite{KKOP21} (see also Proposition~\ref{Prop:braid}).
\end{Ex}

\begin{Ex} %\sejin{}
We illustrate an example of our substitution formula from type $\mathrm{B}_2$ to type $\mathrm{A}_3$.
Consider the height functions $\xi'$ of type $\mathrm{B}_2$ and $\xi$ of type $\mathrm{A}_2$ given as follows:
\[
(\xi'_1,\xi'_2,\xi'_3) = (-3,0,-1), \qquad 
(\xi_1, \xi_2, \xi_3) = (-1, 0, -1).
\]
Then the followings are the infinite sequences $\bfi' = (\im'_u)_{u \in \N}, \bfi = (\im_u)_{u \in \N}\in \{1,2,3\}^{\N}$ satisfying the condition~\eqref{eq:cond2} in Example~\ref{Ex:periodic}:
\begin{align*}
    &\bfi' =(\underset{\text{\scalebox{0.8}{reduced word for $w_{\circ}$}}}{\underbrace{2, 3, {\color{red}2, 1, 2}, 3}}, \underbrace{2,1,{\color{red}2,3,2},1}, \underbrace{2,3,{\color{red}2,1,2},3},\dots),\\
    &\bfi =(\underset{\text{\scalebox{0.8}{reduced word for $w_{\circ}$}}}{\underbrace{2, 3, {\color{red}1, 2, 1}, 3}}, \underbrace{2,1,{\color{red}3,2,3},1}, \underbrace{2,3,{\color{red}1,2,1},3},\dots).
\end{align*}
The quiver $\Gamma_\bfi$ corresponding to $\bfi$ can be depicted as
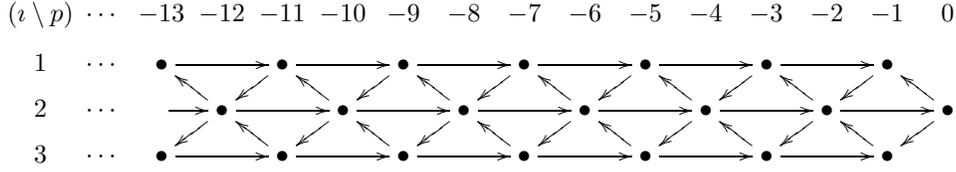
\[
\raisebox{3mm}{
\scalebox{0.9}{\xymatrix@!C=0.5mm@R=2mm{
(\im\setminus p) &\cdots&-13&-12&-11&-10&-9& -8 & -7 & -6 &-5&-4 &-3& -2 &-1& 0 \\
1&\cdots&\bullet \ar@{<-}[dr]&&\bullet \ar@{<-}[ll]\ar@{<-}[dr]&&\bullet \ar@{<-}[ll]\ar@{<-}[dr]&& \bullet \ar@{<-}[ll]\ar@{<-}[dr] &&\bullet \ar@{<-}[ll]\ar@{<-}[dr]
&&\bullet \ar@{<-}[ll]\ar@{<-}[dr] && \bullet \ar@{<-}[ll]\ar@{<-}[dr] &\\
2&\cdots&&\bullet \ar@{<-}[l]\ar@{<-}[dr]\ar@{<-}[ur]&&\bullet \ar@{<-}[ll]\ar@{<-}[dr]\ar@{<-}[ur]&&\bullet \ar@{<-}[ll]\ar@{<-}[dr]\ar@{<-}[ur]&& \bullet \ar@{<-}[ll]\ar@{<-}[dr]\ar@{<-}[ur] &&\bullet \ar@{<-}[ll]\ar@{<-}[dr]\ar@{<-}[ur]
&& \ar@{<-}[ll]\bullet \ar@{<-}[dr]\ar@{<-}[ur]&& \bullet\ar@{<-}[ll]\\ 
3&\cdots&\bullet \ar@{<-}[ur]&&\bullet \ar@{<-}[ll]\ar@{<-}[ur]&&\bullet \ar@{<-}[ll]\ar@{<-}[ur]&& \bullet \ar@{<-}[ll]\ar@{<-}[ur]  &&\bullet\ar@{<-}[ll]\ar@{<-}[ur] &&\bullet \ar@{<-}[ll]\ar@{<-}[ur] && \bullet \ar@{<-}[ll]\ar@{<-}[ur] &}}}
\]
The infinite sequence $\bfi$ can be obtained from $\bfi'$, and vise versa, by applying the braid moves in the red parts of $\bfi$ and $\bfi'$ above. Hence, by applying the mutations at the following vertices 
\begin{equation}
(1, -1), (3,-5), (1, -9), (3,-13), \dots, (1,-1-8m), (3,-5-8m),\dots \label{eq:B2toA3}  
\end{equation}
we obtain the quiver $\Gamma_{\bfi'}$.  
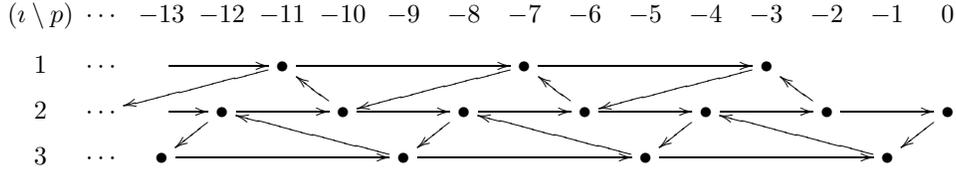
\[
\raisebox{3mm}{
\scalebox{0.9}{\xymatrix@!C=0.5mm@R=2mm{
(\im\setminus p) &\cdots&-13&-12&-11&-10&-9& -8 & -7 & -6 &-5&-4 &-3& -2 &-1& 0 \\
1&\cdots&&&\bullet \ar@{<-}[ll]\ar@{<-}[dr]&&&& \bullet \ar@{<-}[llll]\ar@{<-}[dr] &&&&\bullet \ar@{<-}[llll]\ar@{<-}[dr] &&&\\
2&\cdots\ar@{<-}[urrr]&&\bullet \ar@{<-}[l]\ar@{<-}[drrr]&&\bullet \ar@{<-}[ll]\ar@{<-}[urrr]&&\bullet \ar@{<-}[ll]\ar@{<-}[drrr]&& \bullet \ar@{<-}[ll]\ar@{<-}[urrr] &&\bullet \ar@{<-}[ll]\ar@{<-}[drrr]
&& \ar@{<-}[ll]\bullet && \bullet\ar@{<-}[ll]\\ 
3&\cdots&\bullet \ar@{<-}[ur]&&&&\bullet \ar@{<-}[llll]\ar@{<-}[ur]&&&&\bullet\ar@{<-}[llll]\ar@{<-}[ur] &&&& \bullet \ar@{<-}[llll]\ar@{<-}[ur] &}}}
\]
Here the correspondence of the labelling of vertices is given by 
\begin{comment}
\begin{equation}
\begin{cases}
    (1, -1-8m)\mapsto (2, -2-12m)\\
    (1, -3-8m)\mapsto (1, -3-12m)\\
    (1, -5-8m)\mapsto (1, -7-12m)\\
    (1, -7-8m)\mapsto (1, -11-12m)\\
    (2, -4m)\mapsto (2, -6m)\\
    (2, -2-4m)\mapsto (2, -4-6m)\\
    (3, -1-8m)\mapsto (3, -1-12m)\\
    (3, -3-8m)\mapsto (3, -5-12m)\\
    (3, -5-8m)\mapsto (2, -8-12m)\\
    (3, -7-8m)\mapsto (3, -9-12m)
\end{cases}\label{eq:vertexcorrep}
\end{equation}
\end{comment}
\begin{equation} \label{eq:vertexcorrep}
\begin{aligned}
&    (1, -1-8m)\mapsto (2, -2-12m),\qquad     (1, -3-8m)\mapsto (1, -3-12m), \allowdisplaybreaks\\
&    (1, -5-8m)\mapsto (1, -7-12m),\qquad     (1, -7-8m)\mapsto (1, -11-12m), \allowdisplaybreaks\\
&   \hspace{4ex}  (2, -4m)\mapsto (2, -6m),    \hspace{9.8ex}   (2, -2-4m)\mapsto (2, -4-6m), \allowdisplaybreaks\\
&    (3, -1-8m)\mapsto (3, -1-12m),\qquad     (3, -3-8m)\mapsto (3, -5-12m), \allowdisplaybreaks\\
&    (3, -5-8m)\mapsto (2, -8-12m),\qquad    (3, -7-8m)\mapsto (3, -9-12m),
\end{aligned}
\end{equation}
for $m\in\N_0$. Note that $8$ and $12$ are the numbers $2rh^{\vee}$ for types $\mathrm{A}_3$ and $\mathrm{B}_2$ respectively. By the exchange relation of cluster algebras and \eqref{eq:vertexcorrep}, we have 
\begin{comment}
\begin{align*}
&\hat{\bftau}^*(X'_{(\im, p)})=\\
&\begin{cases}
  X_{1, -3-8m}&\text{if }(\im, p)=(1, -3-12m), \\
  X_{1, -5-8m}&\text{if }(\im, p)=(1, -7-12m), \\
  X_{1, -7-8m}&\text{if }(\im, p)=(1, -11-12m), \\
  X_{2, -4m}&\text{if }(\im, p)=(2, -6m), \\
  (X_{2, -2-8m}X_{1, 1-8m}+X_{1, -3-8m}X_{2, -8m})/ X_{1, -1-8m}&\text{if }(\im, p)=(2, -2-12m), \\ & \qquad \qquad \ (X_{1, 1}\coloneqq 1)\\
  X_{2, -2-4m}&\text{if }(\im, p)=(2, -4-6m), \\
   (X_{3, -7-8m}X_{2, -4-8m}+X_{2, -6-8m}X_{3, -3-8m})/ X_{3, -5-8m}&\text{if }(\im, p)=(2, -8-12m),\\
  X_{3, -1-8m}&\text{if }(\im, p)=(3, -1-12m), \\
  X_{3, -3-8m}&\text{if }(\im, p)=(3, -5-12m), \\
  X_{3, -7-8m}&\text{if }(\im, p)=(3, -9-12m).
\end{cases}    
\end{align*}
\end{comment}
\begin{align*}
\hat{\bftau}^*(X'_{1, -3-12m})&= X_{1, -3-8m},  \qquad
\hat{\bftau}^*(X'_{1, -7-12m})= X_{1, -5-8m}, \allowdisplaybreaks\\
 \hat{\bftau}^*(X'_{1, -11-12m})&= X_{1, -7-8m}, \hspace{7.9ex}
\hat{\bftau}^*(X'_{2, -6m})= X_{2, -4m} \allowdisplaybreaks\\
 \hat{\bftau}^*(X'_{2, -2-12m}) &= \dfrac{X_{2, -2-8m}X_{1, 1-8m}+X_{1, -3-8m}X_{2, -8m}}{X_{1, -1-8m}} \allowdisplaybreaks\\
  \hat{\bftau}^*(X'_{2, -4-6m}) &=   X_{2, -2-4m},  \qquad
 \hat{\bftau}^*(X'_{3, -1-12m}) =   X_{3, -1-8m}, \allowdisplaybreaks\\
  \hat{\bftau}^*(X'_{3, -5-12m}) &=   X_{3, -3-8m}, \qquad
 \hat{\bftau}^*(X'_{3, -9-12m}) =   X_{3, -7-8m}, \allowdisplaybreaks\\
  \hat{\bftau}^*(X'_{2, -8-12m})& = \dfrac{ X_{3, -7-8m}X_{2, -4-8m}+X_{2, -6-8m}X_{3, -3-8m} }{ X_{3, -5-8m}}.
\end{align*}
Here $X_{1,1}\seq 1$.   
\begin{comment}
\begin{align*}
&\tPsi_{\le \xi, \le \xi'}(m^{(\im)}[p, \xi'_{\im}])=\\
&\begin{cases}
  m^{(1)}[-3-8m, -1] &\text{if }(\im, p)=(1, -3-12m), \\
  m^{(1)}[-5-8m, -1] &\text{if }(\im, p)=(1, -7-12m), \\
  m^{(1)}[-7-8m, -1] &\text{if }(\im, p)=(1, -11-12m), \\
  m^{(2)}[-4m, 0] &\text{if }(\im, p)=(2, -6m), \\
  (Y_{2, -2-8m}Y_{1, -1-8m}^{-1}+Y_{1, -3-8m})m^{(2)}[-8m, 0]
  &\text{if }(\im, p)=(2, -2-12m),\\
  m^{(2)}[-2-4m, 0] &\text{if }(\im, p)=(2, -4-6m), \\
   (Y_{3, -7-8m}+Y_{2, -6-8m}Y_{3, -5-8m}^{-1})m^{(2)}[-4-8m, 0]
   &\text{if }(\im, p)=(2, -8-12m),\\
  m^{(3)}[-1-8m, -1] &\text{if }(\im, p)=(3, -1-12m), \\
  m^{(3)}[-3-8m, -1] &\text{if }(\im, p)=(3, -5-12m), \\
  m^{(3)}[-7-8m, -1] &\text{if }(\im, p)=(3, -9-12m). 
\end{cases}    
\end{align*}
\end{comment}
Hence,  
\begin{align*} %\tPsi_{\le \xi, \le \xi'}(m^{(\im)}[p, \xi'_{\im}])=
\tPsi_{\le \xi, \le \xi'}(m^{(1)}[-3-12m, -3])& =   m^{(1)}[-3-8m, -1], \allowdisplaybreaks \\
\tPsi_{\le \xi, \le \xi'}(m^{(1)}[-7-12m, -3]) & =   m^{(1)}[-5-8m, -1],  \allowdisplaybreaks\\
\tPsi_{\le \xi, \le \xi'}(m^{(1)}[-11-12m, -3])& =   m^{(1)}[-7-8m, -1],  \allowdisplaybreaks\\
\tPsi_{\le \xi, \le \xi'}(m^{(2)}[-6m,0])&=   m^{(2)}[-4m, 0],  \allowdisplaybreaks\\
\tPsi_{\le \xi, \le \xi'}(m^{(2)}[-2-12m,0])& =     (Y_{2, -2-8m}Y_{1, -1-8m}^{-1}+Y_{1, -3-8m})m^{(2)}[-8m, 0], \allowdisplaybreaks\\
\tPsi_{\le \xi, \le \xi'}(m^{(2)}[-4-6m, 0])& =   m^{(2)}[-2-4m, 0] , \allowdisplaybreaks\\
\tPsi_{\le \xi, \le \xi'}(m^{(3)}[-1-12m,-1])&=   m^{(3)}[-1-8m, -1], \allowdisplaybreaks\\
\tPsi_{\le \xi, \le \xi'}(m^{(3)}[-5-12m, -1])& =   m^{(3)}[-3-8m, -1],  \allowdisplaybreaks\\
\tPsi_{\le \xi, \le \xi'}(m^{(3)}[-9-12m,-1])&=   m^{(1)}[-7-8m, -1], \allowdisplaybreaks\\
\tPsi_{\le \xi, \le \xi'}(m^{(2)}[-8-12m,0])& =  (Y_{3, -7-8m}+Y_{2, -6-8m}Y_{3, -5-8m}^{-1})m^{(2)}[-4-8m, 0].
\end{align*}
%
\begin{comment}
This calculation implies that, for $(i, p)\in \hI'_{\le \fD^{\prime -2}\xi'}$,
\begin{align*}
&\trPsi_{\le \xi', \le \xi}(Y_{i, p})=\tPsi_{\le \xi, \le \xi'}(m^{(\im)}[p, \xi'_{\im}])\tPsi_{\le \xi, \le \xi'}(m^{(\im)}[p+2d_i, \xi'_{\im}])^{-1}\quad (\bar{\im}=i)\\
&=\begin{cases}
  Y_{1, -3-8m}Y_{1, -1-8m} &\text{if }(i, p)=(1, -3-12m), \\
  Y_{1, -5-8m} &\text{if }(i, p)=(1, -7-12m), \\
  Y_{1, -7-8m} &\text{if }(i, p)=(1, -11-12m), \\
  Y_{2, -8m} &\text{if }(i, p)=(2, -12m), \\
 Y_{2, -2-8m}Y_{1, -1-8m}^{-1}+Y_{1, -3-8m}
  &\text{if }(i, p)=(2, -2-12m),\\
   \displaystyle \frac{1}{Y_{1, -1-8m}^{-1}+Y_{2, -2-8m}^{-1}Y_{1, -3-8m}} &\text{if }(i, p)=(2, -4-12m), \\
   Y_{2, -4-8m} &\text{if }(i, p)=(2, -6-12m), \\
   Y_{3, -7-8m}+Y_{2, -6-8m}Y_{3, -5-8m}^{-1}
   &\text{if }(i, p)=(2, -8-12m),\\
   \displaystyle \frac{1}{Y_{2, -6-8m}^{-1}Y_{3, -7-8m}+Y_{3, -5-8m}^{-1}} 
   &\text{if }(i, p)=(2, -10-12m), \\
  Y_{3, -1-8m} &\text{if }(i, p)=(1, -1-12m), \\
  Y_{3, -3-8m} &\text{if }(i, p)=(1, -5-12m), \\
  Y_{3, -7-8m}Y_{3, -5-8m} &\text{if }(i, p)=(1, -9-12m),
\end{cases}
\qquad\text{for some }m\in \Z_{>0}.
\end{align*}
\end{comment}
Therefore, for $(i, p)\in \hI'$,
\begin{align*}
&\tPsi_{t=1}(Y_{i, p})=\begin{cases}
  Y_{1, -3-8m}Y_{1, -1-8m} &\text{if }(i, p)=(1, -3-12m), \\
  Y_{1, -5-8m} &\text{if }(i, p)=(1, -7-12m), \\
  Y_{1, -7-8m} &\text{if }(i, p)=(1, -11-12m), \\
  Y_{2, -8m} &\text{if }(i, p)=(2, -12m), \\
 Y_{2, -2-8m}Y_{1, -1-8m}^{-1}+Y_{1, -3-8m}
  &\text{if }(i, p)=(2, -2-12m),\\
   \displaystyle \frac{1}{Y_{1, -1-8m}^{-1}+Y_{2, -2-8m}^{-1}Y_{1, -3-8m}} &\text{if }(i, p)=(2, -4-12m), \\
   Y_{2, -4-8m} &\text{if }(i, p)=(2, -6-12m), \\
   Y_{3, -7-8m}+Y_{2, -6-8m}Y_{3, -5-8m}^{-1}
   &\text{if }(i, p)=(2, -8-12m),\\
   \displaystyle \frac{1}{Y_{2, -6-8m}^{-1}Y_{3, -7-8m}+Y_{3, -5-8m}^{-1}} 
   &\text{if }(i, p)=(2, -10-12m), \\
  Y_{3, -1-8m} &\text{if }(i, p)=(1, -1-12m), \\
  Y_{3, -3-8m} &\text{if }(i, p)=(1, -5-12m), \\
  Y_{3, -7-8m}Y_{3, -5-8m} &\text{if }(i, p)=(1, -9-12m),
\end{cases}
\end{align*}
for some $m\in \Z$, as Example~\ref{EX: A2 to A2}.
This is the substitution formula from type $\mathrm{B}_2$ to type $\mathrm{A}_3$.

For example, we have 
\[
\chi_q(L^{\mathrm{B}_2}(Y_{1, -7}))=Y_{1, -7}+Y_{2, -6}Y_{2, -4}Y_{1, -3}^{-1}
+Y_{2, -6}Y_{2, -2}^{-1}+Y_{1, -5}Y_{2, -4}^{-1}Y_{2, -2}^{-1}+Y_{1, -1}^{-1}.
\]
By applying the above formulas, we have 
\begin{align*}
    &Y_{1, -5}+
    \frac{Y_{2, -4}}{(Y_{1, -1}^{-1}+Y_{2, -2}^{-1}Y_{1, -3})Y_{1, -3}Y_{1, -1}}
    + \frac{Y_{2, -4}}{Y_{2, -2}Y_{1, -1}^{-1}+Y_{1, -3}} \allowdisplaybreaks\\
 & \hspace{30ex} + \frac{Y_{3, -3}(Y_{1, -1}^{-1}+Y_{2, -2}^{-1}Y_{1, -3})}{Y_{2, -2}Y_{1, -1}^{-1}+Y_{1, -3}}
+ Y_{3, -1}^{-1}\allowdisplaybreaks\\
&=Y_{1, -5}+
    Y_{2, -4}\frac{Y_{1, -3}^{-1}Y_{1, -1}^{-1}+Y_{2, -2}^{-1}}{Y_{1, -1}^{-1}+Y_{2, -2}^{-1}Y_{1, -3}}
+Y_{3, -3}Y_{2, -2}^{-1}
+ Y_{3, -1}^{-1}\allowdisplaybreaks\\
&=Y_{1, -5}+
    Y_{2, -4}Y_{1, -3}^{-1}
+Y_{3, -3}Y_{2, -2}^{-1}
+ Y_{3, -1}^{-1}=\chi_q(L^{\mathrm{A}_3}(Y_{1, -5})).
\end{align*}
%In this example, the representation in type $B_2^{(1)}$ has dimension $5$ and the corresponding representation in type $A_3^{(1)}$ has dimension $4$. This illustrates the fact that the transformation does not preserve dimensions.
\end{Ex}

\bibliographystyle{alpha}
\bibliography{ref}

\newcommand{\etalchar}[1]{$^{#1}$}
\begin{thebibliography}{KKOP21b}

\bibitem[Bit21]{Bit}
L\'{e}a Bittmann.
\newblock A quantum cluster algebra approach to representations of simply laced
  quantum affine algebras.
\newblock {\em Math. Z.}, 298(3-4):1449--1485, 2021.

\bibitem[BZ05]{BZ05}
Arkady Berenstein and Andrei Zelevinsky.
\newblock Quantum cluster algebras.
\newblock {\em Adv. Math.}, 195(2):405--455, 2005.

\bibitem[CP94]{CP}
Vyjayanthi Chari and Andrew Pressley.
\newblock {\em A guide to quantum groups}.
\newblock Cambridge University Press, Cambridge, 1994.

\bibitem[CP95]{CP95}
Vyjayanthi Chari and Andrew Pressley.
\newblock Quantum affine algebras and their representations.
\newblock In {\em Representations of groups ({B}anff, {AB}, 1994)}, volume~16
  of {\em CMS Conf. Proc.}, pages 59--78. Amer. Math. Soc., Providence, RI,
  1995.

\bibitem[DWZ10]{DWZ10}
Harm Derksen, Jerzy Weyman, and Andrei Zelevinsky.
\newblock Quivers with potentials and their representations {II}: applications
  to cluster algebras.
\newblock {\em J. Amer. Math. Soc.}, 23(3):749--790, 2010.

\bibitem[FHOO22]{FHOO}
Ryo Fujita, David Hernandez, Se-jin Oh, and Hironori Oya.
\newblock Isomorphisms among quantum {G}rothendieck rings and propagation of
  positivity.
\newblock {\em J. Reine Angew. Math.}, 785:117--185, 2022.

\bibitem[FM01]{FM01}
Edward Frenkel and Evgeny Mukhin.
\newblock Combinatorics of {$q$}-characters of finite-dimensional
  representations of quantum affine algebras.
\newblock {\em Comm. Math. Phys.}, 216(1):23--57, 2001.

\bibitem[FO21]{FO21}
Ryo Fujita and Se-jin Oh.
\newblock Q-data and representation theory of untwisted quantum affine
  algebras.
\newblock {\em Comm. Math. Phys.}, 384(2):1351--1407, 2021.

\bibitem[FR99]{FR99}
Edward Frenkel and Nicolai Reshetikhin.
\newblock The {$q$}-characters of representations of quantum affine algebras
  and deformations of {$\mathscr{W}$}-algebras.
\newblock In {\em Recent developments in quantum affine algebras and related
  topics ({R}aleigh, {NC}, 1998)}, volume 248 of {\em Contemp. Math.}, pages
  163--205. Amer. Math. Soc., Providence, RI, 1999.

\bibitem[FZ07]{FZ07}
Sergey Fomin and Andrei Zelevinsky.
\newblock Cluster algebras. {IV}. {C}oefficients.
\newblock {\em Compos. Math.}, 143(1):112--164, 2007.

\bibitem[GHKK18]{GHKK}
Mark Gross, Paul Hacking, Sean Keel, and Maxim Kontsevich.
\newblock Canonical bases for cluster algebras.
\newblock {\em J. Amer. Math. Soc.}, 31(2):497--608, 2018.

\bibitem[GLS13]{GLS13}
C.~Gei\ss, B.~Leclerc, and J.~Schr\"{o}er.
\newblock Cluster structures on quantum coordinate rings.
\newblock {\em Selecta Math. (N.S.)}, 19(2):337--397, 2013.

\bibitem[GLS20]{GLS20}
Christof Geiss, Bernard Leclerc, and Jan Schr\"{o}er.
\newblock Quantum cluster algebras and their specializations.
\newblock {\em J. Algebra}, 558:411--422, 2020.

\bibitem[GY17]{GY17}
K.~R. Goodearl and M.~T. Yakimov.
\newblock Quantum cluster algebra structures on quantum nilpotent algebras.
\newblock {\em Mem. Amer. Math. Soc.}, 247(1169):vii+119, 2017.

\bibitem[Her04]{Her04}
David Hernandez.
\newblock Algebraic approach to {$q,t$}-characters.
\newblock {\em Adv. Math.}, 187(1):1--52, 2004.

\bibitem[Her05]{Her05}
David Hernandez.
\newblock Monomials of {$q$} and {$q, t$}-characters for non simply-laced
  quantum affinizations.
\newblock {\em Math. Z.}, 250(2):443--473, 2005.

\bibitem[Her06]{Her06}
David Hernandez.
\newblock The {K}irillov-{R}eshetikhin conjecture and solutions of {T}-systems.
\newblock {\em J. Reine Angew. Math.}, 596:63--87, 2006.

\bibitem[HL10]{HL10}
David Hernandez and Bernard Leclerc.
\newblock Cluster algebras and quantum affine algebras.
\newblock {\em Duke Math. J.}, 154(2):265--341, 2010.

\bibitem[HL15]{HL15}
David Hernandez and Bernard Leclerc.
\newblock Quantum {G}rothendieck rings and derived {H}all algebras.
\newblock {\em J. Reine Angew. Math.}, 701:77--126, 2015.

\bibitem[HL16]{HL16}
David Hernandez and Bernard Leclerc.
\newblock A cluster algebra approach to {$q$}-characters of
  {K}irillov-{R}eshetikhin modules.
\newblock {\em J. Eur. Math. Soc. (JEMS)}, 18(5):1113--1159, 2016.

\bibitem[HO19]{HO19}
David Hernandez and Hironori Oya.
\newblock Quantum {G}rothendieck ring isomorphisms, cluster algebras and
  {K}azhdan-{L}usztig algorithm.
\newblock {\em Adv. Math.}, 347:192--272, 2019.

\bibitem[IIK{\etalchar{+}}13a]{IIKKKB}
Rei Inoue, Osamu Iyama, Bernhard Keller, Atsuo Kuniba, and Tomoki Nakanishi.
\newblock Periodicities of t-systems and y-systems, dilogarithm identities, and
  cluster algebras {I}: type ${B}_r$.
\newblock {\em Publications of the Research Institute for Mathematical
  Sciences}, 49(1):1--42, 2013.

\bibitem[IIK{\etalchar{+}}13b]{IIKKKB2}
Rei Inoue, Osamu Iyama, Bernhard Keller, Atsuo Kuniba, and Tomoki Nakanishi.
\newblock Periodicities of t-systems and y-systems, dilogarithm identities, and
  cluster algebras {II}: types $ {C}_r$, $ {F}_4$ and $ {G}_2$.
\newblock {\em Publications of the Research Institute for Mathematical
  Sciences}, 49(1):43--85, 2013.

\bibitem[Kel13]{Ke13}
Bernhard Keller.
\newblock The periodicity conjecture for pairs of dynkin diagrams.
\newblock {\em Annals of Mathematics}, pages 111--170, 2013.

\bibitem[KK19]{KK19}
Masaki Kashiwara and Myungho Kim.
\newblock Laurent phenomenon and simple modules of quiver {H}ecke algebras.
\newblock {\em Compos. Math.}, 155(12):2263--2295, 2019.

\bibitem[KKKO18]{KKKO18}
Seok-Jin Kang, Masaki Kashiwara, Myungho Kim, and Se-jin Oh.
\newblock Monoidal categorification of cluster algebras.
\newblock {\em J. Amer. Math. Soc.}, 31(2):349--426, 2018.

\bibitem[KKOP21a]{KKOP21}
Masaki Kashiwara, Myungho Kim, Se-jin Oh, and Euiyong Park.
\newblock Braid group action on the module category of quantum affine algebras.
\newblock {\em Proceedings of the Japan Academy, Ser. A, Mathematical
  Sciences}, 97(3):13--18, 2021.

\bibitem[KKOP21b]{KKOP2}
Masaki Kashiwara, Myungho Kim, Se-jin Oh, and Euiyong Park.
\newblock Monoidal categorification and quantum affine algebras {II}.
\newblock Preprint, \arxiv{2103.10067}v3, 2021.

\bibitem[KKOP22]{KKOP22}
Masaki Kashiwara, Myungho Kim, Se-jin Oh, and Euiyong Park.
\newblock Cluster algebra structures on module categories over quantum affine
  algebras.
\newblock {\em Proceedings of the London Mathematical Society},
  124(3):301--372, 2022.

\bibitem[KO23]{KO22}
Masaki Kashiwara and Se-jin Oh.
\newblock The {$(q,t)$}-{C}artan matrix specialized at {$q=1$} and its
  applications.
\newblock {\em Math. Z.}, 303(2):Paper No. 42, 49, 2023.

\bibitem[Lus93]{LusztigBook}
George Lusztig.
\newblock {\em Introduction to quantum groups}, volume 110 of {\em Progress in
  Mathematics}.
\newblock Birkh\"{a}user Boston, Inc., Boston, MA, 1993.

\bibitem[Nak]{NakajimaQCB}
Hiraku Nakajima.
\newblock Modules of quantized {C}oulomb branches.
\newblock in preparation.

\bibitem[Nak01]{Nak01}
Hiraku Nakajima.
\newblock Quiver varieties and finite-dimensional representations of quantum
  affine algebras.
\newblock {\em J. Amer. Math. Soc.}, 14(1):145--238, 2001.

\bibitem[Nak03]{Nakajima03II}
Hiraku Nakajima.
\newblock {$t$}-analogs of {$q$}-characters of {K}irillov-{R}eshetikhin modules
  of quantum affine algebras.
\newblock {\em Represent. Theory}, 7:259--274 (electronic), 2003.

\bibitem[Nak04]{Nak04}
Hiraku Nakajima.
\newblock Quiver varieties and {$t$}-analogs of {$q$}-characters of quantum
  affine algebras.
\newblock {\em Ann. of Math. (2)}, 160(3):1057--1097, 2004.

\bibitem[NW23]{NW19}
Hiraku Nakajima and Alex Weekes.
\newblock Coulomb branches of quiver gauge theories with symmetrizers.
\newblock {\em J. Eur. Math. Soc. (JEMS)}, 25(1):203--230, 2023.

\bibitem[Qin17]{Qin17}
Fan Qin.
\newblock Triangular bases in quantum cluster algebras and monoidal
  categorification conjectures.
\newblock {\em Duke Math. J.}, 166(12):2337--2442, 2017.

\bibitem[Tra11]{Tran}
Thao Tran.
\newblock {$F$}-polynomials in quantum cluster algebras.
\newblock {\em Algebr. Represent. Theory}, 14(6):1025--1061, 2011.

\bibitem[VV03]{VV03}
M.~Varagnolo and E.~Vasserot.
\newblock Perverse sheaves and quantum {G}rothendieck rings.
\newblock In {\em Studies in memory of {I}ssai {S}chur ({C}hevaleret/{R}ehovot,
  2000)}, volume 210 of {\em Progr. Math.}, pages 345--365. Birkh\"{a}user
  Boston, Boston, MA, 2003.

\end{thebibliography}
\end{document}